\documentclass[reqno,12pt]{amsart}

\theoremstyle{plain}
\usepackage{amsmath}
\usepackage{amssymb}
\usepackage{amsxtra}
\usepackage{color}
\usepackage{float}
\usepackage{pdflscape}
\usepackage{slashed}

\usepackage{varioref}
\usepackage{colortbl}
\usepackage[bookmarksnumbered,colorlinks,plainpages]{hyperref}
\usepackage{cite}
\usepackage{graphicx}
\usepackage{amsthm}
\usepackage[all,2cell]{xy}
\usepackage{amsfonts}
\usepackage{xypic,leqno,amscd,amssymb,pstricks,latexsym,amsbsy,xypic,mathrsfs,verbatim}
\usepackage{multirow}
\usepackage{stmaryrd}
\usepackage{makecell}
\usepackage{amsthm}
\usepackage[all,2cell]{xy}
\usepackage{amsfonts}
\usepackage{etex}
\usepackage{mathtools}
\usepackage{tikz}
\usepackage{yhmath}

\usetikzlibrary{positioning,arrows}
\usetikzlibrary{decorations.pathreplacing}
\newtheorem*{Thm-A}{Theorem A}
\newtheorem*{Thm-B}{Theorem B}

\newtheorem{thm}{Theorem}[section]

\newtheorem{cor}[thm]{Corollary}
\newtheorem{lem}[thm]{Lemma}

\newtheorem{prop}[thm]{Proposition}
\newtheorem{conven}{Convention}

\theoremstyle{definition}
\newtheorem{defn}[thm]{Definition}
\newtheorem{rem}[thm]{\bf Remark}
\newtheorem{exm}[thm]{Example}
\numberwithin{equation}{section}

\def\Hom{{\rm Hom}}

\def\Ext{{\rm Ext}}
\def\Aut{{\rm Aut}}

\def\co{{\mathcal O}}

\def\vc{{\vec{c}}}

\def\bbX{{\mathbb X}}
\def\bbL{{\mathbb L}}

\begin{document}
\title[Geometric model for weighted projective lines of type $(p,q)$]{Geometric model for weighted projective lines of type $(p,q)$}

\author[J. Chen, S. Ruan and H. Zhang] {Jianmin Chen, Shiquan Ruan and Hongxia Zhang$^*$}

\thanks{$^*$ the corresponding author}
\footnote[0]{\textcolor{black}{2020} \textit{Mathematics Subject Classification}. 05C10, 05E10, 16D90, 16G70, 57K20.}

\dedicatory{}
\commby{}

\begin{abstract}
We give a geometric model for the category of coherent sheaves over the weighted projective line of type $(p,q)$ in terms of an annulus with marked points on its boundary. We establish a bijection between indecomposable sheaves over the weighted projective line and certain homotopy classes of oriented curves in the annulus, and prove that the dimension of extension group between indecomposable sheaves equals to the positive intersection number between the corresponding curves.

By using the geometric model, we provide a combinatorial description for the titling graph of tilting bundles, which is composed by quadrilaterals (or degenerated to a line).
Moreover, we obtain that the automorphism group of the coherent sheaf category is isomorphic to the mapping class group of the marked annulus, and show the compatibility of their actions on the tilting graph of coherent sheaves and on the triangulation of the geometric model respectively. A geometric
description of the perpendicular category with respect to an exceptional sheaf is presented at the end of the paper.
\end{abstract}

\maketitle
\section{Introduction}

Geometric models for categories have been studied by various authors in recent years. For instance, a geometric construction of cluster categories of type $A$ was given by Caldero-Chapton-Schiffler in \cite{CCS2006}, of type $D$ was given by Schiffler in \cite{S08} and of type $A_\infty$ was given by Holm-J{\o}rgensen in \cite{HJ2012}. There is also some progress for geometric models of abelian categories. Baur-Marsh in \cite{BM2012} provided a geometric model for tube categories. In \cite{W2008}, Warkentin established a bijection between string modules over a quiver of affine type $A$ and certain oriented curves in a marked annulus.
By these geometric realizations, many algebraic properties (e.g. the extension dimensions, Auslander-Reiten triangles, Auslander-Reiten sequences) of these categories can be studied in geometric terms. We refer to \cite{BM07,BM08,BS2021,BT2019,DLL19,FST2008,HZZ23,S08,T15} for related topics.

Weighted projective lines and their coherent sheaves categories were introduced by Geigle and Lenzing in \cite{GL87}, in order to give a geometric realization of canonical algebras in the sense of Ringel \cite{Ringel1984}. The study of weighted projective lines has a high contact with many mathematical branches, such as Lie theory \cite{Crawley-Boevey2010,DengRuanXiao2020,DouJiangXiao2012,Schiffmann2004}, representation theory \cite{M2004,CLLR2021, FG383,HR1999} and singularity theory \cite{EbelingPloog2010,Hubner1989,Hubner1996,Lenzing1994,Lenzing1998,Lenzing2011}, which in particular including
the aspects of Arnold's strange duality \cite{EbelingPloog2010,EbelingTakahashi2013} and homological mirror symmetry \cite{Ebeling2003,Scerbak1978}.

By Geigle-Lenzing \cite{GL87}, the coherent sheaf category ${\rm coh}\mbox{-}\mathbb{X}(p,q)$ of type $(p,q)$ is derived equivalent to the finitely generated module category ${\rm mod}\tilde{A}_{p,q}$ of the canonical algebra $\tilde{A}_{p,q}$ (affine type $A$).
Inspired by \cite{BT2019, W2008}, we hope to give a geometric model for the category ${\rm coh}\mbox{-}\mathbb{X}(p,q)$ in terms of a marked annulus.

Let $A_{p,q}$ be an annulus with $p$ marked points on the inner boundary and $q$ marked points on the outer boundary. We establish a bijection between indecomposable sheaves over ${\rm coh}\mbox{-}\mathbb{X}(p,q)$ and certain homotopy classes of oriented curves in the annulus $A_{p,q}$ (see Theorem \ref{corresponding}). Under this correspondence, Auslander-Reiten sequences in ${\rm coh}\mbox{-}\mathbb{X}(p,q)$ can
be realized via elementary moves in $A_{p,q}$ (see Proposition \ref{irreducible morphism}), and the dimension of extension space between two indecomposable coherent sheaves equals to the positive intersection number of the correspondence curves (see Theorem \ref{dimension and positive intersection}). Moreover,
we obtain that the tilting sheaves in ${\rm coh}\mbox{-}\mathbb{X}(p,q)$ are in natural bijection with the triangulations of $A_{p,q}$ (see Theorem \ref{triangulation and tilting}), and the flip of an arc is compatible with the tilting mutation (see Proposition \ref{mutation and flip}).

We point out here that the geometric model given in Theorem \ref{corresponding} has intuitive difference from the geometric realization of ${\rm mod}\tilde{A}_{p,q}$ in \cite{BT2019, W2008}. More precisely, the indecomposable modules in the postprojective (\emph{resp.} preinjective) components of ${\rm mod}\tilde{A}_{p,q}$ correspond to the bridging curves whose orientation is from the outer boundary to the inner boundary (\emph{resp.} from the inner boundary to the outer boundary). However, all line bundles in ${\rm coh}\mbox{-}\mathbb{X}(p,q)$ correspond to the bridging curves whose orientations are from the outer boundary to the inner boundary.

There is another interesting point, since the dimensions of the Hom-space and Ext-space between two indecomposable sheaves have explicit formulas due to the structure of the category ${\rm coh}\mbox{-}\mathbb{X}(p,q)$, Theorem \ref{dimension and positive intersection} makes the positive intersection number of the correspondence curves in $A_{p,q}$ easily calculated. Therefore, it seems that the category of coherent sheaves over $\mathbb{X}(p,q)$ provides a nice categorification model for the annulus $A_{p,q}$.

The geometric model has applications on the automorphism group of ${\rm coh}\mbox{-}\mathbb{X}(p,q)$ and the tilting graph of coherent sheaves. Denote by
$$ \mathcal{T}_{A_{p,q}}:=\{{\rm Triangulations\ of}\ A_{p,q}\} {\text{\quad and \quad}} \mathcal{T}_{\mathbb{X}}:=\{{\rm Tilting \ sheaves \ in} \ {\rm coh}\mbox{-}\mathbb{X}(p,q)\}.$$
Then there is a bijection $\phi:  \mathcal{T}_{A_{p,q}} \rightarrow\mathcal{T}_{\mathbb{X}}$, see (\ref{map}).

 There is an isomorphism $\psi$ between the mapping class group $\mathcal{MG} (A_{p,q})$ of the marked annulus $A_{p,q}$ and the automorphism group $\Aut({\rm coh}\mbox{-}\mathbb{X}(p,q))$ of ${\rm coh}\mbox{-}\mathbb{X}(p,q)$, see \eqref{automor psi}. Any automorphism of ${\rm coh}\mbox{-}\mathbb{X}(p,q)$ preserves tilting sheaves. Hence there is a natural group action of $\Aut({\rm coh}\mbox{-}\mathbb{X}(p,q))$ on $\mathcal{T}_{\mathbb{X}}$.
On the other hand, the mapping class group $\mathcal{MG} (A_{p,q})$ naturally acts on the set of triangulations of $A_{p,q}$.
It turns out that these two actions are compatible.
That is, we have the following commutative diagram, where the commutativity is in the sense of (\ref{compatible of groups action}).
\begin{figure}[h]
\begin{tikzpicture}
\node (1) at (0,1.5) {{\tiny{$\mathcal{MG} (A_{p,q})$}}};
\node (2) at (0,-0.1){{\tiny{$\mathcal{T}_{A_{p,q}}$}}};
\node (3) at (3.2,1.5) {{\tiny{$\Aut({\rm coh}\mbox{-}\mathbb{X}(p,q))$}}};
\node (4) at (3.3,-0.1) {{\tiny{$\mathcal{T}_{\mathbb{X}}$}}};
\draw[->] (-0.1,0.15) arc (255:-75:0.6);
\draw[->] (3.1,0.15) arc (255:-75:0.6);
\draw[->](1) --node[above]{\tiny{$\psi$}}(3);
\draw[->](2) --node[above]{\tiny{$\phi$}}(4);
\end{tikzpicture}
\end{figure}

\begin{thm}\label{compatible}
For any $f\in \mathcal{MG} (A_{p,q})$ and any triangulation $\Gamma$ of $A_{p,q}$, we have
\begin{align}\label{compatible of groups action}
\phi(f(\Gamma))=\psi(f)(\phi(\Gamma))
\end{align}
\end{thm}

The \emph{tilting graph} $\mathcal{G}(\mathcal{T}_{\mathbb{X}})$ of ${\rm coh}\mbox{-}\mathbb{X}(p,q)$ has as vertices the isomorphism classes of tilting sheaves in ${\rm coh}\mbox{-}\mathbb{X}(p,q)$, while two vertices are connected by an edge if and only if the associated tilting sheaves differ by precisely one indecomposable direct summand. The full subgraph of $\mathcal{G}(\mathcal{T}_{\mathbb{X}})$  consisting of tilting bundles will be denoted by $\mathcal{G}(\mathcal{T}^{\nu}_{\mathbb{X}})$.

The connectedness of titling graph for weighted projective lines has been investigated widely in the literature through category aspect, see for example \cite{BKL2008, HU2005, GengSF2020, FG383}. However, the explicit shape of the tilting graph is still unknown.
By using the above geometric model, we provide a combinatorial description of $\mathcal{G}(\mathcal{T}^{\nu}_{\mathbb{X}})$.

Let $\Lambda_{(p,q)}$ be a graph with vertices
$$\Lambda_{(p,q)}^0=\{(c_1, \cdots, c_p)\in\mathbb{Z}^{p}|c_1\leq \cdots\leq c_p\leq c_1+q\},$$ and there exists an edge between two vertex $(c_1, \cdots, c_p)$ and $(d_1, \cdots, d_p)$ if and only if $$\sum_{i=1}^{p}|c_i-d_i|=1.$$

\begin{thm}\label{description of tilting graph of vector bundles}
The tilting graph $\mathcal{G}(\mathcal{T}^{\nu}_{\mathbb{X}})$ coincides with the graph $\Lambda_{(p,q)}$.
\end{thm}

Denote by $\eta$ the bijection from $\Lambda_{(p,q)}$ to $\mathcal{G}(\mathcal{T}^{\nu}_{\mathbb{X}})$ obtained in Theorem \ref{description of tilting graph of vector bundles}.
Let $H_{p,q}=\langle r_1, r_2|r_1r_2=r_2r_1, r_1^{p}=r_2^{q}\rangle$ and
\begin{equation*}
\widetilde{H}_{p,q}=
  \left\{
  \begin{array}{ll}
   H_{p,q}, & p\neq q;\\
   H_{p,q}\times\mathbb{Z}_2, &  p=q.\\
  \end{array}
\right.\end{equation*}
Then $\widetilde{H}_{p,q}$ coincides with the mapping class group $\mathcal{MG} (A_{p,q})$, hence there exists a group isomorphism from $\widetilde{H}_{p,q}$ to $\Aut({\rm coh}\mbox{-}\mathbb{X}(p,q))$, still denoted by $\psi$. We construct an unexpected group action of $\widetilde{H}_{p,q}$ on the graph $\Lambda_{(p,q)}$ (c.f. Proposition \ref{bijective map of vertex such that commutative}), which is compatible with the group action of $\Aut({\rm coh}\mbox{-}\mathbb{X}(p,q))$ on $\mathcal{G}(\mathcal{T}^{\nu}_{\mathbb{X}})$ in the following sense.

\begin{thm}\label{compatible2}
For any $f\in \widetilde{H}_{p,q}$ and any vertex $\nu$ in $\Lambda_{(p,q)}$, we have
\begin{align}\label{compatible of groups action2}
\eta(f(\nu))=\psi(f)(\eta(\nu)).
\end{align}
Consequently, we have the following commutative diagram
\begin{figure}[h]
\begin{tikzpicture}
\node (1) at (0,1.55) {{\tiny{$\widetilde{H}_{p,q}$}}};
\node (2) at (0,-0.1){{\tiny{$\Lambda_{(p,q)}$}}};
\node (3) at (3.38,1.55) {{\tiny{$\Aut({\rm coh}\mbox{-}\mathbb{X}(p,q))$}}};
\node (4) at (3.3,-0.1) {{\tiny{$\mathcal{G}(\mathcal{T}^{\nu}_{\mathbb{X}})$}}};
\draw[->] (-0.1,0.15) arc (255:-75:0.6);
\draw[->] (3.1,0.15) arc (255:-75:0.6);
\draw[->](1) --node[above]{\tiny{$\psi$}}(3);
\draw[->](2) --node[above]{\tiny{$\eta$}}(4);
\end{tikzpicture}
\end{figure}
\end{thm}

The paper is organized as follows. In Section 2, we recall some basic facts on weighted projective lines of type $(p,q)$. In Section 3, we show that
the marked annulus $A_{p,q}$ gives a geometric model of the category ${\rm coh}\mbox{-}\mathbb{X}(p,q)$. In Section 4, we establish a bijection between tilting sheaves in ${\rm coh}\mbox{-}\mathbb{X}(p,q)$ and triangulations of $A_{p,q}$, and show that the flip of an arc is compatible with the tilting mutation of an indecomposable sheaf. Sections 5 and 6 focus on tilting mutation and tilting graphs, with the aim to prove Theorem \ref{description of tilting graph of vector bundles}. We give a geometric interpretation of the automorphism group of ${\rm coh}\mbox{-}\mathbb{X}(p,q)$ in Section 7, and prove Theorem \ref {compatible} and Theorem \ref{compatible2}. In the final Section 8, we present a geometric description of the perpendicular category of an exceptional sheaf in ${\rm coh}\mbox{-}\mathbb{X}(p,q)$.

\section{Weighted projective lines of type (p,q)}\label{wpl}

In this section, we recall from \cite{GL87, Len2011} some basic facts about the weighted projective lines of type $(p,q)$ for $p,q\in \mathbb{Z}_{\geq1}$.

Let $\mathbb{L}(p,q)$ be the abelian group on generators $\vec{x}_1, \vec{x}_2$ with relations $$p\vec{x}_1=q\vec{x}_2:=\vec{c}.$$
Then each $\vec{x}\in\mathbb{L}(p, q)$ can be uniquely written in \emph{normal form}
$$\vec{x}=l_1\vec{x}_1+l_2\vec{x}_2+l\vec{c},\;\;{\rm where}\,\, 0\leq l_1\leq p-1,\,0\leq l_2\leq q-1\,\,{\rm and}\;l\in \mathbb{Z}.$$
$\mathbb{L}(p,q)$ is an ordered group whose cone of positive elements is $\mathbb{N}\vec{x}_1+\mathbb{N}\vec{x}_2$.
We equip $\mathbb{L}(p,q)$ with the structure of a partially ordered set:
$\vec{x}\leq\vec{y}$ if and only if $\vec{y}-\vec{x}\in \mathbb{N}\vec{x}_1+\mathbb{N}\vec{x}_2$.

Let $\mathbf{k}$ be an algebraically closed field and ${\boldsymbol\lambda}=
(\lambda_1,\lambda_2)$ be a sequence of
pairwise distinct closed points on the projective line $\mathbb{P}_{\mathbf k}^1$. A \emph{weighted projective line} $\mathbb{X}(p,q)$ of weight type $(p,q)$ and parameter sequence ${\boldsymbol\lambda}$ is obtained from the projective line $\mathbb{P}_{\mathbf k}^1$ by attaching the weight $p,\; q$ to $\lambda_1, \lambda_2$, respectively. The parameter sequence can be normalized as $\lambda_1=\infty,\;\lambda_2=0.$

The \emph{homogeneous coordinate algebra} $S(p,q)$ of the weighted projective line $\mathbb{X}(p,q)$ is given by $\mathbf{k}[x_1, x_2]$,
which is $\mathbb{L}(p, q)$-graded by means of $\deg x_i=\vec{x}_i$ for $i=1,\,2$. That is, $S(p,q)=\bigoplus_{\vec{x}\in \mathbb{L}(p,q)} S(p,q)_{\vec{x}}$, where $S(p,q)_{\vec{x}}$ is the homogeneous component of degree $\vec{x}$. In particular, if we write $\vec{x}=l_1\vec{x}_1+l_2\vec{x}_2+l\vec{c}$ in its normal form, then $S(p,q)_{\vec{x}}\neq 0$ if and only if $l\geq 0$. Moreover, $\{x_1^{l_1+pa}x_2^{l_2+qb}\; |\; a+b=l, a, b\geq 0\}$ form a ${\mathbf k}$-basis of $S(p,q)_{\vec{x}}$; see \cite[Proposition 1.3]{GL87}.

We recall the definition of the category ${\rm coh}\mbox{-}\mathbb{X}(p,q)$ of coherent sheaves over $\mathbb{X}(p,q)$ by a convenient description via graded $S(p,q)$-modules. Let ${\rm mod}^{\bbL(p,q)}\mbox{-}S(p,q)$ be the abelian category of finitely generated $\bbL(p,q)$-graded $S(p,q)$-modules, and ${\rm mod}_0^{\bbL(p,q)}\mbox{-}S(p,q)$ be its Serre subcategory formed by finite dimensional modules. Denote by ${\rm qmod}^{\bbL(p,q)}\mbox{-}S(p,q):={\rm mod}^{\bbL(p,q)}\mbox{-}S(p,q)/{{\rm mod}_0^{\bbL(p,q)}\mbox{-}S(p,q)}$ the quotient abelian category. By \cite[Theorem 1.8]{GL87} the sheafification functor yields an equivalence
$$
{\rm qmod}^{\bbL(p,q)}\mbox{-}S(p,q)\stackrel{\sim}\longrightarrow {\rm coh}\mbox{-}\mathbb{X}(p,q).
$$
From now on we will identify these two categories.

Notice that in this paper, we will abbreviate ${\rm Hom}_{{\rm coh}\mbox{-}\mathbb{X}(p,q)}(-,-)$ and ${\rm Ext}^{1}_{{\rm coh}\mbox{-}\mathbb{X}(p,q)}(-,-)$ as ${\rm Hom}(-,-)$ and ${\rm Ext}^{1}(-,-)$ respectively.

\begin{prop}[\cite{GL87, Len2011}]\label{the properties of coherent sheaves}
The category ${\rm coh}\mbox{-}\mathbb{X}(p,q)$ is connected, ${\rm Hom}$-finite and $\mathbf{k}$-linear with the following properties:

\begin{itemize}
\item[(1)] ${\rm coh}\mbox{-}\mathbb{X}(p,q)$ satisfies Serre duality in the form ${\rm Ext}^{1}(X,Y)\cong{\rm DHom}(Y, \tau(X)),$ where the $\mathbf{k}$-equivalence $\tau:{\rm coh}\mbox{-}\mathbb{X}(p,q)\longrightarrow {\rm coh}\mbox{-}\mathbb{X}(p,q)$ is the shift $X\mapsto X(\vec{\omega})$ with the dualizing element $\vec{\omega}=-(\vec{x}_1+\vec{x}_2)\in \mathbb{L}(p,q).$

\item[(2)] ${\rm coh}\mbox{-}\mathbb{X}(p,q)={\rm vec}\mbox{-}\mathbb{X}(p,q)\bigvee{\rm coh}_{0}\mbox{-}\mathbb{X}(p,q),$ where ${\rm vec}\mbox{-}\mathbb{X}(p,q)$ denotes the full subcategory of ${\rm coh}\mbox{-}\mathbb{X}(p,q)$ consisting of vector bundles over $\mathbb{X}(p,q)$, and ${\rm coh}_{0}\mbox{-}\mathbb{X}(p,q)$ denotes the full subcategory of ${\rm coh}\mbox{-}\mathbb{X}(p,q)$ consisting of all objects of finite length, $\bigvee$ means that each indecomposable object of ${\rm coh}\mbox{-}\mathbb{X}(p,q)$ is either in ${\rm vec}\mbox{-}\mathbb{X}(p,q)$ or in ${\rm coh}_{0}\mbox{-}\mathbb{X}(p,q)$, and there are no non-zero morphism from ${\rm coh}_{0}\mbox{-}\mathbb{X}(p,q)$ to ${\rm vec}\mbox{-}\mathbb{X}(p,q).$

\item[(3)] Each indecomposable bundle over $\mathbb{X}(p,q)$ is a line bundle.

\item[(4)] Each line bundle $L$ is exceptional, i.e., ${\rm Hom}(L,L)\cong\mathbf{k}$ and ${\rm Ext}^{1}(L,L)=0$.

\item[(5)] For any $\vec{x}, \vec{y}\in\mathbb{L}(p,q)$, there has
$${\rm Hom}(\co(\vec{x}), \co(\vec{y}))\cong S(p,q)_{\vec{y}-\vec{x}}.$$
In particular, ${\rm Hom}(\co(\vec{x}), \co(\vec{y}))\neq 0$ if and only if $\vec{x}\leq \vec{y}.$

\item[(6)] Denote by $K_{0}({\rm coh}\mbox{-}\mathbb{X}(p,q))$ the \emph{Grothendieck group} of the category ${\rm coh}\mbox{-}\mathbb{X}(p,q)$. Then the classes $[\co(\vec{x})] \; (0\leq \vec{x}\leq \vec{c})$ form a $\mathbb{Z}$-basis of $K_{0}({\rm coh}\mbox{-}\mathbb{X}(p, q))$. Moreover, the rank of Grothendieck group $K_{0}({\rm coh}\mbox{-}\mathbb{X}(p,q))$ equals $p+q.$
\end{itemize}
\end{prop}

By \cite[Proposition 1.1]{LR06}, the subcategory ${\rm coh}_0\mbox{-}\mathbb{X}(p,q)$ decomposes into a coproduct
$\coprod_{\lambda\in \mathbb{P}_{\mathbf k}^1}\mathcal{U}_{\lambda}$ of connected uniserial length categories, whose associated Auslander-Reiten quiver is a stable tube $\mathbb{ZA}_{\infty}/(\tau^r)$ for some $r\in \mathbb{Z}_{\geq 1}$ (see e.g.\cite{SS2007}). Here, $r$ is called the rank of the stable tube
$\mathbb{ZA}_{\infty}/(\tau^r)$, which depends on $\lambda$. Precisely, $r=p,\; q$ for $\lambda=\infty,\;0$, respectively and $r=1$ for $\lambda \in \mathbf k^{*}=\mathbf k\setminus \{0\}$. Objects that lie at the bottom of the stable tube are all simple objects of ${\rm coh}\mbox{-}\mathbb{X}(p,q)$. A stable tube of $r=1$ is called a \emph{homogeneous stable tube}. Each $\lambda \in \mathbf k^{*}$ is associated with a unique simple sheaf $S_{\lambda}$, called \emph{ordinary simple}; while $\lambda=\infty$ (\emph{resp.} $\lambda=0$) is associated with $p$ (\emph{resp.} $q$) simple sheaves $S_{\infty,i}\;(i\in\mathbb{Z}/p\mathbb{Z})$ (\emph{resp.} $S_{0,i}\; (i\in\mathbb{Z}/q\mathbb{Z}$)) called \emph{exceptional simples}.

Now we recall some important short exact sequences in ${\rm coh}\mbox{-}\mathbb{X}(p,q)$. For each ordinary simple sheaf $S_{\lambda}$, there is an exact sequence
$$0\longrightarrow \mathcal{O}\stackrel{u_{\lambda}}{\longrightarrow}\mathcal{O}(\vec{c})\longrightarrow S_{\lambda}\longrightarrow 0,$$
where the homomorphism $u_{\lambda}: \mathcal{O}\longrightarrow \mathcal{O}(\vec{c})$ is given by
multiplication with $x_2^{q}-\lambda x_1^{p}$.
By contrast, if $\lambda\in \{\infty, 0\}$, there are exact sequences
\begin{align}\label{important exact sequence}
&0\longrightarrow \mathcal{O}((i-1)\vec{x}_1)\stackrel{u_{\infty}}{\longrightarrow} \mathcal{O}(i\vec{x}_1)\longrightarrow S_{\infty,i}\longrightarrow 0,\;\;i\in\mathbb{Z}/p\mathbb{Z};\\
&0\longrightarrow \mathcal{O}((i-1)\vec{x}_2)\stackrel{u_{0}}{\longrightarrow} \mathcal{O}(i\vec{x}_2)\longrightarrow S_{0,i}\longrightarrow 0,\;\;i\in\mathbb{Z}/q\mathbb{Z};
\end{align}
where $u_{\infty}$ (\emph{resp.} $u_{0}$) is given by multiplication with $x_1$ (\emph{resp.} $x_2$).

For each $\vec{x}=l_1\vec{x}_1+l_2\vec{x}_2\in \mathbb{L}(p,q)$, one has
$$S_{\lambda}(\vec{x})=S_{\lambda}\;\;{\rm for}\;\, \lambda \in \mathbf k^{*};$$
and $$S_{\infty,i}(\vec{x})=S_{\infty, i+l_{1}},\;\; S_{0,i}(\vec{x})=S_{0, i+l_{2}}.$$

For $\lambda\in \{\infty,0\}$ and each $j \in \mathbb{Z}_{\geq 1}$, let $S_{\lambda,i}^{(j)}$ denote the indecomposable object in $\mathcal{U}_{\lambda}$ of length $j$ with top $S_{\lambda,i}.$ More precisely, the composition series of $S_{\lambda,i}^{(j)}$ has the following form:
$$S_{\lambda, i-j+1} \hookrightarrow S_{\lambda, i-j+2}^{(2)} \hookrightarrow \cdots\hookrightarrow S_{\lambda, i-2}^{(j-2)} \hookrightarrow S_{\lambda, i-1}^{(j-1)}\hookrightarrow  S_{\lambda, i}^{(j)}$$
with $S_{\lambda, i-k}^{(j-k)}/S_{\lambda, i-k-1}^{(j-k-1)}\cong S_{\lambda, i-k}$ for $0\leq k\leq j-2.$

\section{Geometric model for the category ${\rm coh}\mbox{-}\mathbb{X}(p,q)$}\label{gmwpl}

In this section, we aim to give a geometric model for ${\rm coh}\mbox{-}\mathbb{X}(p,q)$ in terms of an annulus with $p$ marked points on the inner boundary and $q$ marked points on the outer boundary. We will establish a correspondence between indecomposable coherent sheaves and homotopy classes of oriented curves in the marked annulus and then study algebraic properties (e.g. Auslander-Reiten sequences, the dimension of extension space of degree 1, exact sequences) of ${\rm coh}\mbox{-}\mathbb{X}(p,q)$ in geometric aspects.

\subsection{The universal cover of $A_{p,q}$.}

Let $A_{p,q}$ be an annulus with $p$ marked points on the inner boundary and $q$ marked points on the outer boundary. Without loss of generality, assume $1\leq p\leq q.$ In this subsection, we recall from \cite{BT2019} for the universal cover of $A_{p,q}$.

Assume the marked points are distributed in equidistance on the two boundaries of $A_{p,q}$. We label the marked points on the inner boundary with $0_{\partial},\,(\frac{1}{p})_{\partial},\,\cdots,\, (\frac{p-1}{p})_{\partial}$ in an anti-clockwise direction, and label the marked points on the outer boundary with $0_{\partial^{\prime}},$ $(\frac{q-1}{q})_{\partial^{\prime}},\,\cdots,$
$(\frac{1}{q})_{\partial^{\prime}}$ in a clockwise direction; c.f. Figure \ref{marked annulus}.

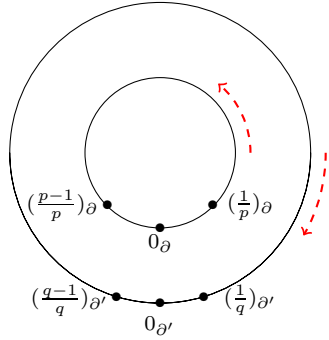
\begin{figure}[H]
\begin{tikzpicture}
\draw[](0,0)arc(0:-540:2);
\draw[](-1,0)arc(0:360:1);
\draw[dashed,red,thick,->](-0.8,0)arc(0:50:1.2);
\draw[dashed,red,thick,->](0.2,0)arc(0:-30:2.2);
\node()at(-2,-1){\tiny$\bullet$};
\node()at(-2,-1.2){\tiny{$0_{\partial}$}};
\node()at(-1.3,-0.7){\tiny$\bullet$};
\node()at(-0.8,-0.7){\tiny{$(\frac{1}{p})_{\partial}$}};
\node()at(-2.7,-0.7){\tiny$\bullet$};
\node()at(-3.3,-0.7){\tiny{$(\frac{p-1}{p})_{\partial}$}};
\node()at(-2,-2){\tiny$\bullet$};
\node()at(-2,-2.3){\tiny{$0_{\partial^{\prime}}$}};
\node()at(-2.58,-1.92){\tiny$\bullet$};
\node()at(-3.2,-1.96){\tiny{$(\frac{q-1}{q})_{\partial^{\prime}}$}};
\node()at(-1.42,-1.92){\tiny$\bullet$};
\node()at(-0.8,-1.96){\tiny{$(\frac{1}{q})_{\partial^{\prime}}$}};
\end{tikzpicture}
\caption{Marked annulus $A_{p,q}$}\label{marked annulus}
\end{figure}

Let ${\rm Cyl_{p,q}}$ be a rectangle of height 1 and width $1$ in $\mathbb{R}^{2}$ with $p$ marked points on the upper boundary $\partial$ and $q$ marked points on the lower boundary $\partial^{\prime}$, identifying its two vertical sides. Denote the points on the upper boundary of ${\rm Cyl_{p,q}}$ by $i_{\partial},$ for $0\leq i\leq \frac{p-1}{p}:$
$$0_{\partial}:=(0, 1),\;\,(\frac{1}{p})_{\partial}:=(\frac{1}{p}, 1),\;\,\cdots,\;\,(\frac{p-1}{p})_{\partial}:=(\frac{p-1}{p}, 1),$$
and denote the points on the lower boundary of ${\rm Cyl_{p,q}}$ by $j_{\partial^{\prime}}$, for $0\leq j\leq\frac{q-1}{q}:$
$$0_{\partial^{\prime}}:=(0, 0),\;\,(\frac{1}{q})_{\partial^{\prime}}:=(\frac{1}{q}, 0),\;\,\cdots,\;\,(\frac{q-1}{q})_{\partial^{\prime}}:=(\frac{q-1}{q}, 0).$$ Then $A_{p,q}$ can be identified with ${\rm Cyl_{p,q}}$ in the sense of viewing the upper (\emph{resp.} lower) boundary of ${\rm Cyl_{p,q}}$ as the inner (\emph{resp.} outer) boundary of $A_{p,q}$, where the marked points of ${\rm Cyl_{p,q}}$ are oriented  from left to right on the upper boundary $\partial$ and from right to left on the lower boundary $\partial^{\prime}$, as in Figure \ref{Annulus via rectangle}.
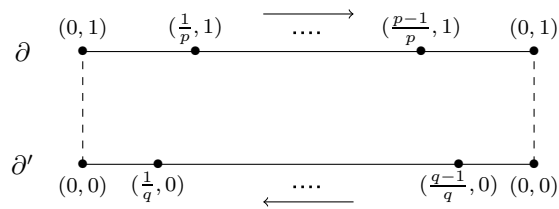
\begin{figure}[H]
\begin{tikzpicture}
\draw[-](-5,-0.5)--(1,-0.5);
\draw[->](-2.6,1.5)--(-1.4,1.5);
\draw[->](-1.4,-1)--(-2.6,-1);
\draw[-](-5,1)--(1,1);
\draw[-,dashed](-5,-0.5)--(-5,1);
\draw[-,dashed](1,-0.5)--(1,1);
\node()at(-5,1){\tiny$\bullet$};
\node()at(-5,-0.5){\tiny$\bullet$};
\node()at(-3.5, 1){\tiny$\bullet$};
\node()at(-4,-0.5){\tiny$\bullet$};
\node()at(1,-0.5){\tiny$\bullet$};
\node()at(-0.5,1){\tiny$\bullet$};
\node()at(1,1){\tiny$\bullet$};
\node()at(0,-0.5){\tiny$\bullet$};
\node()at(-5,1.3){\tiny{$(0,1)$}};
\node()at(-5,-0.8){\tiny{$(0,0)$}};
\node()at(-3.5,1.3){\tiny{$(\frac{1}{p}, 1)$}};
\node()at(-0.5,1.3){\tiny{$(\frac{p-1}{p}, 1)$}};
\node()at(-4,-0.8){\tiny{$(\frac{1}{q}, 0)$}};
\node()at(0,-0.8){\tiny{$(\frac{q-1}{q}, 0)$}};
\draw[line width=1pt,dotted](-2.2,1.25)--(-1.8,1.25);
\draw[line width=1pt,dotted](-2.2,-0.8)--(-1.8,-0.8);
\node()at(1,1.3){\tiny{$(0,1)$}};
\node()at(1,-0.8){\tiny{$(0, 0)$}};
\node()at(-5.8,-0.5){\footnotesize$\partial^{\prime}$};
\node()at(-5.8,1){\footnotesize$\partial$};
\end{tikzpicture}
\caption{Annulus via rectangle ${\rm Cyl_{p,q}}$}\label{Annulus via rectangle}
\end{figure}

It will be most convenient to work with in the universal cover $(\mathbb{U}, \pi_{p,q})$ of ${\rm Cyl_{p,q}}$, where $\mathbb{U}=\{(x,y)\in\mathbb{R}^{2}|0\leq y\leq 1\}$ inherits the orientation of ${\rm Cyl_{p,q}}$, and the covering map $$\pi:=\pi_{p,q}: \mathbb{U}\to {\rm Cyl_{p,q}},\;\;(x,y)\mapsto(x-\lfloor x\rfloor, y).$$
Here, $\lfloor x\rfloor$ denotes the integer part of $x$. Naturally,  $\pi$ is also a covering map of $A_{p,q}$.

Denote the marked point $(i, 1)$ on the upper boundary $\partial$ of $\mathbb{U}$ by $i_{\partial}$ and the marked point $(j, 0)$ on the lower boundary $\partial^{\prime}$ of $\mathbb{U}$ by $j_{\partial^{\prime}}$, for $ i\in \frac{\mathbb{Z}}{p}$ and $j\in \frac{\mathbb{Z}}{q}$, see Figure \ref{Universal cover of $A_{p,q}$}.

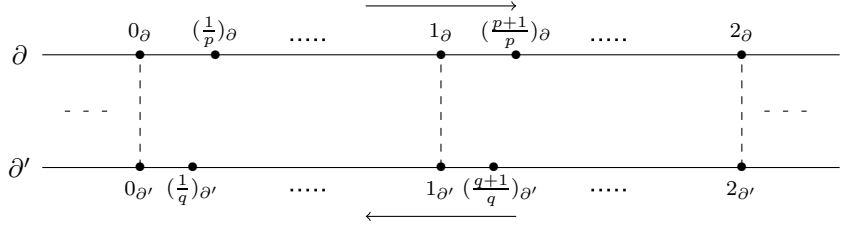
\begin{figure}[H]
\begin{tikzpicture}
\draw[-,dashed](-4,-0.5)--(-4,1);
\draw[-,dashed](0,-0.5)--(0,1);
\draw[-,dashed](4,-0.5)--(4,1);
\draw[-](-5.3,-0.5)--(5.3,-0.5);
\draw[-](-5.3,1)--(5.3,1);
\draw[->](-1,1.65)--(1,1.65);
\draw[->](1,-1.15)--(-1,-1.15);
\node()at(-4,-0.5){\tiny$\bullet$};
\node()at(0,-0.5){\tiny$\bullet$};
\node()at(-4,1){\tiny$\bullet$};
\node()at(0,1){\tiny$\bullet$};
\node()at(-3,1){\tiny$\bullet$};
\node()at(-3.3,-0.5){\tiny$\bullet$};
\node()at(4,1){\tiny$\bullet$};
\node()at(4,-0.5){\tiny$\bullet$};
\node()at(1,1){\tiny$\bullet$};
\node()at(0.7,-0.5){\tiny$\bullet$};
\node()at(-4,1.3){\tiny{$0_{\partial}$}};
\node()at(-5.6,1){\footnotesize{${\partial}$}};
\node()at(-5.6,-0.5){\footnotesize{${\partial^{\prime}}$}};
\node()at(-4,-0.8){\tiny{$0_{\partial^{\prime}}$}};
\node()at(-3,1.3){\tiny{$(\frac{1}{p})_{\partial}$}};
\node()at(-3.3,-0.8){\tiny{$(\frac{1}{q})_{\partial^{\prime}}$}};
\node()at(1,1.3){\tiny{$(\frac{p+1}{p})_{\partial}$}};
\node()at(0.8,-0.8){\tiny{$(\frac{q+1}{q})_{\partial^{\prime}}$}};
\draw[line width=1pt,dotted](-2,1.2)--(-1.5,1.2);
\draw[line width=1pt,dotted](-2,-0.8)--(-1.5,-0.8);
\draw[line width=1pt,dotted](2,1.2)--(2.5,1.2);
\draw[line width=1pt,dotted](2,-0.8)--(2.5,-0.8);
\node()at(4,1.3){\tiny{$2_{\partial}$}};
\node()at(4,-0.8){\tiny{$2_{\partial^{\prime}}$}};
\node()at(0,1.3){\tiny{$1_{\partial}$}};
\node()at(0,-0.8){\tiny{$1_{\partial^{\prime}}$}};
\draw[dash pattern=on 2pt off 5pt on 2pt off 5pt](-5,0.25)--(-4.3,0.25);
\draw[dash pattern=on 2pt off 5pt on 2pt off 5pt](4.3,0.25)--(5,0.25);
\end{tikzpicture}
\caption{Universal cover $\mathbb{U}$ of $A_{p,q}$}\label{Universal cover of $A_{p,q}$}
\end{figure}
\subsection{Arcs in $A_{p,q}$ and $\mathbb{U}$.}

In this subsection, we recall some definitions of curves and arcs in $A_{p,q}$ and $\mathbb{U}$. In this paper, we consider curves and arcs up to homotopy. The following definition refers to \cite{V2018, BZ2011}.

\begin{defn}
A {\emph {curve}} in $\mathbb{U}$ is a continuous function $c:[0,1]\longrightarrow \mathbb{U}$ such that $c(t)\in \mathbb{U}^0:=\mathbb{U}\setminus\{\partial, \partial^{\prime}\}$ for any $t\in (0,1)$. An {\emph {arc}} in $\mathbb{U}$ is a curve whose endpoints are marked points of $\mathbb{U}$, satisfying that it does not intersect itself in the interior of $\mathbb{U}$, the interior of the arc is disjoint from the boundary of $\mathbb{U}$, and it does not cut out a monogon or digon.
\end{defn}

If an arc in $\mathbb{U}$ starts at a marked point $x_{b_1}$ and ends at a marked point $y_{b_2}$ with $x, y\in \frac{\mathbb{Z}}{p}$ or $\frac{\mathbb{Z}}{q},
\,b_{i}\in \{\partial,\,\partial^{\prime}\}$ for $i=1,\,2$, it will be denoted by $[x_{b_1}, y_{b_2}]$.
The following definition refers to \cite{BT2019}.

\begin{defn}\label{peripheral arc and bridging arc}
Let $\alpha=[x_{b_1}, y_{b_2}]$ be an arc in $\mathbb{U}$. If $b_1\neq b_2,$ $\alpha$ is called a \emph{bridging arc}. If $b_1=b_2=\partial$ and $y-x\geq \frac{2}{p}$, or $b_1=b_2=\partial^{\prime}$ and $y-x\geq \frac{2}{q}$, then $\alpha$ is called a \emph{peripheral arc}.
\end{defn}

In order to give a geometric model for line bundles in ${\rm coh}\mbox{-}\mathbb{X}(p,q)$, we only need to consider the bridging arcs that oriented from the boundary $\partial^{\prime}$ to the boundary $\partial$. Such an arc $[x_{\partial^{\prime}}, y_{\partial}]$ is called a \emph{positive bridging arc}.

Similarly, one can define curves and arcs in the marked annulus $A_{p,q}$. Note that the image of an arc in $\mathbb{U}$ under $\pi: \mathbb{U}\to A_{p,q}$ maybe intersect itself in the interior of $A_{p,q}$. We define \emph{bridging (\emph{resp.} peripheral) curve} in $A_{p,q}$ as the image $\pi(\alpha)$ for a \emph{bridging (\emph{resp.} peripheral) arc} $\alpha$ in $\mathbb{U}$. If $\pi(\alpha)$ is an arc additionally, then it will be called a \emph{bridging (\emph{resp.} peripheral) arc} in $A_{p,q}$.

In order to give a geometric model for homogeneous torsion sheaves in the subcategory $\coprod_{\lambda\in \mathbf{k}^{*}}\mathcal{U}_{\lambda}$ of ${\rm coh}\mbox{-}\mathbb{X}(p,q)$, we introduce the notion of $\mathbf k^{*}$-parameterized $n$-cycles as follows.

\begin{defn}
For $n\in \mathbb{Z}_{\geq1}$, an $n$-cycle in $A_{p,q}$ is a curve $\pi(c)$, where $c$ is a curve in $\mathbb{U}^{0}$ with $c(1)-c(0)=(n,0).$ In particular, the 1-cycle will be called a \emph{loop}. For the notion of $\mathbf k^{*}$-parameterized $n$-cycles we refer to the set $\{(\lambda, L^{n})\,|\,\lambda\in \mathbf k^{*}\}$, where $L$ is a loop in $A_{p,q}$ with the orientation in an anti-clockwise direction.
\end{defn}

Recall from Proposition \ref{the properties of coherent sheaves} that the category ${\rm coh}\mbox{-}\mathbb{X}(p,q)$ contains two parts: vector bundles and torsion sheaves.
Each indecomposable vector bundle is a line bundle.
The torsion sheaves consists of two stable tubes of rank $p$ and $q$ respectively, and
a family of homogeneous tubes indexed by $\mathbf k^{*}$.

Let ${\rm ind}({\rm coh}\mbox{-}\mathbb{X}(p,q))$ denote the full subcategory of ${\rm coh}\mbox{-}\mathbb{X}(p,q)$ consisting of all indecomposable objects. Then the objects of ${\rm ind}({\rm coh}\mbox{-}\mathbb{X}(p,q))$ are line bundles $\co(\vec{x})\,(\vec{x}\in \mathbb{L}(p,q))$, and torsion sheaves $S_{\infty,i}^{(j)}\,(i\in\mathbb{Z}/p\mathbb{Z})$, $S_{0,i}^{(j)}\,(i\in\mathbb{Z}/q\mathbb{Z})$ and $S_{\lambda}^{(j)}\,(\lambda \in \mathbf k^{*})$ for $j \in \mathbb{Z}_{\geq 1}$.

Now we can state the main result of this subsection.

\begin{thm}\label{corresponding}
A parametrization of the isoclasses of ${\rm ind}({\rm coh}\mbox{-}\mathbb{X}(p,q))$ is given as follows:
\begin{itemize}
 \item[(1)] the indecomposable vector bundles are indexed by the homotopy classes of positive bridging curves $\pi([(\frac{j}{q})_{\partial^{\prime}}, (\frac{i}{p})_{\partial}])$ in $A_{p,q}$;
 \item[(2)] the objects in the stable tube of rank $p$ (\emph{resp.} $q$) are indexed by the homotopy classes of peripheral curves $\pi([(\frac{i}{p})_{\partial}, (\frac{j}{p})_{\partial}])$ (\emph{resp.} $\pi([(\frac{i}{q})_{\partial^{\prime}}, (\frac{j}{q})_{\partial^{\prime}}])$) in $A_{p,q}$;
    \item[(3)] the objects in homogeneous tubes are indexed by $\mathbf k^{*}$-parameterized $n$-cycles in $A_{p,q}$.
 \end{itemize}
\end{thm}

\begin{proof}
Recall that $\pi: \mathbb{U}\to A_{p,q}$ is the covering map. Let $\mathcal{C}$ be the set consisting of the following elements with $i, j\in\mathbb{Z}$ and $n\geq 1$:
\begin{itemize}
\item[-] positive bridging curves: $\pi([(\frac{j}{q})_{\partial^{\prime}}, (\frac{i}{p})_{\partial}])$;
    \item[-] peripheral curves: $\pi([(\frac{i}{p})_{\partial}, (\frac{j}{p})_{\partial}])$ and $\pi([(\frac{i}{q})_{\partial^{\prime}}, (\frac{j}{q})_{\partial^{\prime}}])$, with $j-i\geq 2;$
 \item[-]  $\mathbf k^{*}$-parameterized $n$-cycles: $\{(\lambda, L^{n})\,|\,\lambda\in \mathbf k^{*}\}$.
\end{itemize}
Then the following assignments:
\begin{align}\label{map}\nonumber
&\pi([(\frac{j}{q})_{\partial^{\prime}}, \;(\frac{i}{p})_{\partial}])\mapsto\co(i\vec{x}_1-j\vec{x}_2),\quad\;\pi([(\frac{i-j-1}{p})_{\partial},\;(\frac{i}{p})_{\partial}])\mapsto S_{\infty,i}^{(j)}, \nonumber \\
&\\
&\pi([(\frac{-i}{q})_{\partial^{\prime}},\;(\frac{j-i+1}{q})_{\partial^{\prime}}])\mapsto S_{0,i}^{(j)},\quad\;
 (\lambda, L^{j}) \mapsto S_{\lambda}^{(j)},\nonumber
\end{align}
define a bijective map $\phi:\mathcal{C}\longrightarrow{\rm ind}({\rm coh}\mbox{-}\mathbb{X}(p,q))$. This proves Theorem \ref{corresponding}.
\end{proof}

\subsection{Elementary moves}

Basing on the bijection given in Theorem \ref{corresponding}, we consider the geometric interpretation of Auslander-Reiten sequences in ${\rm coh}\mbox{-}\mathbb{X}(p,q)$. Firstly, we introduce the notion of elementary moves of positive bridging curves and peripheral curves in $A_{p,q}$. The following definition refers to \cite{BT2019, BZ2011}.

\begin{defn}\label{elementary move}
For any positive bridging curve or peripheral curve $\gamma$ in $A_{p,q}$, denote by $_{s}\gamma$ the \emph{elementary move of $\gamma$ on its starting point}, meaning that the curve $_{s}\gamma$ is obtained from $\gamma$ moving its starting point
to the next marked point on the same boundary, such that $_{s}\gamma$ is rotated in clockwise direction around the ending point of $\gamma.$
Similarly, denote by $\gamma_{e}$ the \emph{elementary move of $\gamma$ on its ending point}. Iterated elementary moves are denoted $_{s}\gamma_{e}=$ $_{s}(\gamma_{e})=(_{s}\gamma)_{e},\, _{s^{2}}\gamma=$ $_{s}(_{s}\gamma)$ and $\gamma_{e^{2}}=(\gamma_{e})_{e}$, respectively.
\end{defn}

For the follow-up discussion, we recall from \cite{BT2019} for the intuitional expression of elementary moves of positive bridging curves and peripheral curves in $A_{p,q}$.

\begin{itemize}
 \item [(1)] If $\gamma$ is a positive bridging curve in $A_{p,q}$, then $\gamma=\pi([(\frac{i}{q})_{\partial^{\prime}},\, (\frac{j}{p})_{\partial}])$ for some $i, j\in \mathbb{Z}$,
 and $$_{s}\gamma=\pi([(\frac{i-1}{q})_{\partial^{\prime}},\, (\frac{j}{p})_{\partial}]),\quad\; \gamma_{e}=\pi([(\frac{i}{q})_{\partial^{\prime}},\, (\frac{j+1}{p})_{\partial}]).$$
The corresponding picture in $\mathbb{U}$ is:

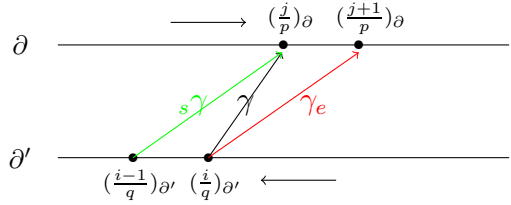
\begin{figure}[H]
\begin{tikzpicture}
\draw[-](0,0)--(6,0);
\draw[-](0,1.5)--(6,1.5);
\draw[->](1.5,1.8)--(2.5,1.8);
\draw[<-](2.7,-0.3)--(3.7,-0.3);
\node()at(-0.5,0){\footnotesize$\partial^{\prime}$};
\node()at(-0.5,1.5){\footnotesize$\partial$};
\node()at(2,0){\tiny$\bullet$};
\node()at(2.1,-0.36){\tiny$(\frac{i}{q})_{\partial^{\prime}}$};
\node()at(3,1.5){\tiny$\bullet$};
\node()at(3.1,1.85){\tiny$(\frac{j}{p})_{\partial}$};
\draw[->](2,0)--(2.98,1.4);
\node()at(1,0){\tiny$\bullet$};
\node()at(1.1,-0.36){\tiny$(\frac{i-1}{q})_{\partial^{\prime}}$};
\draw[->,green](1,0)--(2.98,1.4);
\node()at(4,1.5){\tiny$\bullet$};
\node()at(4.16,1.85){\tiny$(\frac{j+1}{p})_{\partial}$};
\draw[->,red](2,0)--(4,1.4);
\node()at(1.8,0.7){${\green _{s}\gamma}$};
\node()at(3.4,0.7){${\red \gamma_{e}}$};
\node()at(2.5,0.7){$\gamma$};
\end{tikzpicture}
\end{figure}

\item[(2)] If $\gamma$ is a peripheral curve in $A_{p,q}$ with endpoints on the inner boundary, then $\gamma=\pi([(\frac{i}{p})_{\partial}, (\frac{j}{p})_{\partial}])$ for some $i, j\in \mathbb{Z},\,j-i\geq 2,$ and
    $$_{s}\gamma=\pi([(\frac{i+1}{p})_{\partial}, (\frac{j}{p})_{\partial}])\,\,(j-i\,\textgreater\,\, 2),\quad\; \gamma_{e}=\pi([(\frac{i}{p})_{\partial}, (\frac{j+1}{p})_{\partial}]).$$
In $\mathbb{U}$, they look as follows:

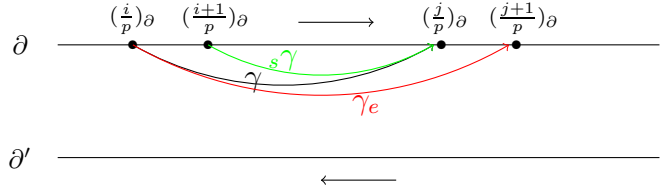
\begin{figure}[H]
\begin{tikzpicture}
\draw[-](0,0)--(8,0);
\draw[-](0,1.5)--(8,1.5);
\draw[->](3.2,1.8)--(4.2,1.8);
\draw[<-](3.5,-0.3)--(4.5,-0.3);
\node()at(-0.5,0){\footnotesize$\partial^{\prime}$};
\node()at(-0.5,1.5){\footnotesize$\partial$};
\node()at(1,1.5){\tiny$\bullet$};
\node()at(6.1,1.5){\tiny$\bullet$};
\node()at(1,1.85){\tiny$(\frac{i}{p})_{\partial}$};
\node()at(5.15,1.85){\tiny$(\frac{j}{p})_{\partial}$};
\node()at(2,1.5){\tiny$\bullet$};
\node()at(5.1,1.5){\tiny$\bullet$};
\node()at(2.1,1.85){\tiny$(\frac{i+1}{p})_{\partial}$};
\node()at(6.2,1.85){\tiny$(\frac{j+1}{p})_{\partial}$};
\draw[<-](5,1.5)arc(-60:-120:4);
\draw[<-,green](5,1.5)arc(-60:-120:3);
\node()at(2.6,1){$\gamma$};
\node()at(3,1.3){${\green _{s}\gamma}$};
\draw[<-,red](6,1.5)arc(-60:-120:5);
\node()at(4.1,0.7){${\red\gamma_{e}}$};
\end{tikzpicture}
\end{figure}
\item [(3)] If $\gamma$ is a peripheral curve in $A_{p,q}$ with endpoints on the outer boundary, then $\gamma=\pi([(\frac{i}{q})_{\partial^{\prime}},\, (\frac{j}{q})_{\partial^{\prime}}])$ for some $i, j\in \mathbb{Z},\,j-i\geq 2$, and $$_{s}\gamma=\pi([(\frac{i-1}{q})_{\partial^{\prime}},\, (\frac{j}{q})_{\partial^{\prime}}]),\quad \;\gamma_{e}=\pi([(\frac{i}{q})_{\partial^{\prime}},\, (\frac{j-1}{q})_{\partial^{\prime}}])\,\,(j-i\,\textgreater\; 2).$$
In $\mathbb{U},$ they are described as follows:

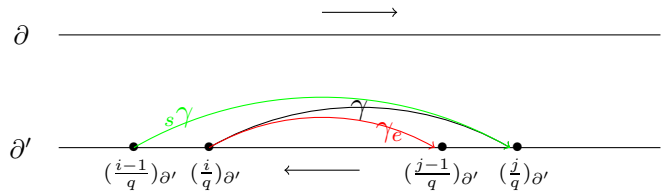
\begin{figure}[H]
\begin{tikzpicture}
\draw[-](0,0)--(8,0);
\draw[-](0,1.5)--(8,1.5);
\draw[->](3.5,1.8)--(4.5,1.8);
\draw[<-](3,-0.3)--(4,-0.3);
\node()at(-0.5,0){\footnotesize$\partial^{\prime}$};
\node()at(-0.5,1.5){\footnotesize$\partial$};
\node()at(2,0){\tiny$\bullet$};
\node()at(6.1,0){\tiny$\bullet$};
\node()at(2.1,-0.36){\tiny$(\frac{i}{q})_{\partial^{\prime}}$};
\node()at(6.2,-0.36){\tiny$(\frac{j}{q})_{\partial^{\prime}}$};
\node()at(1,0){\tiny$\bullet$};
\node()at(5.1,0){\tiny$\bullet$};
\node()at(1.1,-0.36){\tiny$(\frac{i-1}{q})_{\partial^{\prime}}$};
\node()at(5.1,-0.36){\tiny$(\frac{j-1}{q})_{\partial^{\prime}}$};
\draw[<-](6,0)arc(60:120:4);
\node()at(4,0.5){$\gamma$};
\draw[<-,green](6,0)arc(60:120:5);
\node()at(1.6,0.4){${\green _{s}\gamma}$};
\draw[<-,red](5,0)arc(60:120:3);
\node()at(4.4,0.2){${\red \gamma_{e}}$};
\end{tikzpicture}
\end{figure}
\end{itemize}

For the sake of simplicity, from now on, we denote by
$$D^{\frac{i}{p}}_{\frac{j}{q}}:=[(\frac{j}{q})_{\partial^{\prime}}, (\frac{i}{p})_{\partial}],\;\;D^{\frac{i}{p}, \frac{j}{p}}:=[(\frac{i}{p})_{\partial}, (\frac{j}{p})_{\partial}]\;\;{\rm and}\;\;D_{\frac{i}{q}, \frac{j}{q}}:=[(\frac{i}{q})_{\partial^{\prime}}, (\frac{j}{q})_{\partial^{\prime}}]$$
for positive bridging arcs and peripheral arcs in $\mathbb{U}$,
and denote their image under $\pi: \mathbb{U}\to A_{p,q}$ as $[D^{\frac{i}{p}}_{\frac{j}{q}}], [D^{\frac{i}{p}, \frac{j}{p}}]$ and $[D_{\frac{i}{q}, \frac{j}{q}}]$ respectively. It follows that for any $k\in\mathbb{Z}$,
\begin{align}\label{periodicity of covering map}
[D^{\frac{i}{p}+k}_{\frac{j}{q}+k}]=[D^{\frac{i}{p}}_{\frac{j}{q}}],\;\;[D^{\frac{i}{p}+k,\, \frac{j}{p}+k}]=[D^{\frac{i}{p}, \frac{j}{p}}]\;\;{\rm and}\;\;[D_{\frac{i}{q}+k,\, \frac{j}{q}+k}]=[D_{\frac{i}{q}, \frac{j}{q}}].
\end{align}

\begin{prop}\label{irreducible morphism}
Let $X\in{\rm ind}({\rm coh}\mbox{-}\mathbb{X}(p,q))$ be a line bundle or a torsion sheaf supported at an exceptional point. Assume $\phi^{-1}(X)=\gamma$. Then
the irreducible morphisms in ${\rm coh}\mbox{-}\mathbb{X}(p,q)$ starting at $X$ are obtained by elementary moves on endpoints of $\gamma$. Moreover, the Auslander-Reiten sequence starting at $X$  has the form
$$0\longrightarrow \phi(\gamma)\longrightarrow \phi(_{s}\gamma)\oplus\phi(\gamma_{e})\longrightarrow\phi(_{s}\gamma_{e})\longrightarrow 0.$$
\end{prop}

\begin{proof}
If $X$ is a line bundle, then $X=\co(i\vec{x}_1-j\vec{x}_2)$ for some $i, j\in\mathbb{Z}$, and then $\gamma=[D^{\frac{i}{p}}_{\frac{j}{q}}]$. Notice that the irreducible morphisms in ${\rm coh}\mbox{-}\mathbb{X}(p,q)$ starting at $\co(i\vec{x}_1-j\vec{x}_2)$ are given by
$$\co(i\vec{x}_1-j\vec{x}_2)\longrightarrow \co((i+1)\vec{x}_1-j\vec{x}_2)=\phi(\gamma_{e})$$ and $$\co(i\vec{x}_1-j\vec{x}_2)\longrightarrow \co(i\vec{x}_1-(j-1)\vec{x}_2)=\phi(_{s}\gamma),$$ and $$\tau^{-1}(\co(i\vec{x}_1-j\vec{x}_2))=(\co(i\vec{x}_1-j\vec{x}_2))(\vec{x}_1+\vec{x}_2)
=\co((i+1)\vec{x}_1-(j-1)\vec{x}_2)=\phi(_{s}\gamma_{e}).$$ Then the result follows.

If $X$ is a torsion sheaf supported at $\infty$, then $X=S_{\infty,i}^{(j)}$ for some $i\in\mathbb{Z}/p\mathbb{Z}$ and $j\in \mathbb{Z}_{\geq 1}$, and then $\gamma=[D^{\frac{i-j-1}{p},\frac{i}{p}}].$ If $j=1$, we have $\phi(_{s}\gamma)=0$ and there is only one irreducible morphism starting at $S_{\infty,i}^{(1)}$, that is, $S_{\infty,i}^{(1)}\longrightarrow S_{\infty,i+1}^{(2)}=\phi(\gamma_{e})$. If $j\geq 2$, there are two irreducible morphisms starting at $S_{\infty,i}^{(j)}$, given by $S_{\infty,i}^{(j)}\longrightarrow S_{\infty,i+1}^{(j+1)}=\phi(\gamma_{e})$ and $S_{\infty,i}^{(j)}\longrightarrow S_{\infty,i}^{(j-1)}=\phi(_{s}\gamma)$. Moreover, $\tau^{-1}(S_{\infty,i}^{(j)})=S_{\infty,i+1}^{(j)}=\phi(_{s}\gamma_{e}).$
Similarly, if $X$ is a torsion sheaf supported at the exceptional point $0$, we can also obtain the result. We are done.
\end{proof}

\begin{rem}
For any indecomposable torsion sheaf $X$ supported at an ordinary point $\lambda\in\mathbf k^{*}$, we know that $X=S_{\lambda}^{(j)}$ for some $j\in \mathbb{Z}_{\geq 1}$. By Theorem \ref{corresponding}, we have $\phi^{-1}(S_{\lambda}^{(j)})=(\lambda, L^{j}).$
If we denote by $_s(\lambda, L^{k})=(\lambda, L^{k-1})$, $(\lambda, L^{k})_{e}=(\lambda, L^{k+1})$ and $_{s}(\lambda, L^{k})_{e}=$ $_{s}((\lambda, L^{k})_{e})=(_{s}(\lambda, L^{k}))_{e}=(\lambda, L^{k})$, then Proposition \ref{irreducible morphism} holds for any $X\in{\rm ind}({\rm coh}\mbox{-}\mathbb{X}(p,q))$.
\end{rem}

\subsection{Translation quiver and equivalence between categories}
In this subsection, we define a translation quiver $(\Gamma_{A_{p,q}}, \tau^{\prime})$ associated to $\mathcal{C}.$ The vertices of the quiver are the oriented curves in $\mathcal{C}.$ There is an arrow $\gamma\rightarrow\gamma^{\prime}$ if and only if $\gamma^{\prime}=$ $_{s}\gamma$ or $\gamma_{e}.$ For the positive bridging curves and peripheral curves in $\mathcal{C},$ let the translation $\tau^{\prime}$ be induced by the map $(\frac{i}{p})_{\partial}\mapsto (\frac{i-1}{p})_{\partial}$ and $(\frac{j}{q})_{\partial^{\prime}}\mapsto (\frac{j+1}{q})_{\partial^{\prime}}$ for $i,\,j \in\mathbb{Z},$ that is,
$$\tau^{\prime}([D^{\frac{i}{p}}_{\frac{j}{q}}])=[D^{\frac{i-1}{p}}_{\frac{j+1}{q}}],\,\;
\tau^{\prime}([D^{\frac{i}{p}, \frac{j}{p}}])=[D^{\frac{i-1}{p}, \frac{j-1}{p}}],\;\,
\tau^{\prime}([D_{\frac{i}{q}, \frac{j}{q}}])=[D_{\frac{i+1}{q}, \frac{j+1}{q}}].$$
For the $\mathbf k^{*}$-parameterized $k$-cycles $(\lambda, L^{k})$ in $\mathcal{C},$ let
$\tau^{\prime}((\lambda, L^{k}))=(\lambda, L^{k}).$

Let $\mathcal{C}_{A_{p,q}}$ be the mesh category of the translation quiver $(\Gamma_{A_{p,q}}, \tau^{\prime})$. Let $F:\mathcal{C}_{A_{p,q}}\longrightarrow{\rm ind}({\rm coh}\mbox{-}\mathbb{X}(p,q))$ be the functor between mesh categories defined as follows. On objects, let $F(\gamma)=\phi(\gamma).$
To define $F$ on morphisms, it suffices to define it on the elementary moves. Define $F(\gamma\rightarrow$ $_{s}\gamma)$ (\emph{resp.} $F(\gamma\rightarrow \gamma_{e})$) to be the irreducible morphism $F(\gamma)\rightarrow F(_{s}\gamma)$ (\emph{resp.} $F(\gamma)\rightarrow F(\gamma_{e})$).

Let $(\Gamma_{{\rm coh}\mbox{-}\mathbb{X}(p,q)}, \tau)$ denote the Auslander-Reiten quiver of the category ${\rm coh}\mbox{-}\mathbb{X}(p,q).$

\begin{thm}\label{equivalence of categories}

The functor $F:\mathcal{C}_{A_{p,q}}\longrightarrow{\rm ind}({\rm coh}\mbox{-}\mathbb{X}(p,q))$ is an equivalence of categories. In particular,
\begin{itemize}
  \item[(1)] $F$ induces an isomorphism of translation quivers $(\Gamma_{A_{p,q}},\tau^{\prime})\longrightarrow (\Gamma_{{\rm coh}\mbox{-}\mathbb{X}(p,q)},\tau).$
\item[(2)] $\tau^{\prime}$ corresponds to the Auslander-Reiten translation $\tau$ in the following sense
      $$F \circ\tau^{\prime}=\tau\circ F.$$
  \item[(3)] $F$ is an exact functor with respect to the induced abelian structure on $\mathcal{C}_{A_{p,q}}.$
  \end{itemize}
\end{thm}

\begin{proof}
By Theorem \ref{corresponding} and Proposition \ref{irreducible morphism}, we get the statement $(1)$. For the statement $(2),$ we only consider the case $F \circ\tau^{\prime}([D^{\frac{i}{p}}_{\frac{j}{q}}])=\tau\circ F([D^{\frac{i}{p}}_{\frac{j}{q}}])$, the others being similar. In fact,
\begin{flalign}
F\circ\tau^{\prime}([D^{\frac{i}{p}}_{\frac{j}{q}}])=F([D^{\frac{i-1}{p}}_{\frac{j+1}{q}}])=&\mathcal{O}((i-1)\vec{x}_1-(j+1)\vec{x}_2)
\nonumber\\
=&\tau(\mathcal{O}(i\vec{x}_1-j\vec{x}_2))\nonumber\\
=&\tau\circ F([D^{\frac{i}{p}}_{\frac{j}{q}}]).\nonumber
\end{flalign}
Therefore, the statement $(2)$ holds.

Since both categories $\mathcal{C}_{A_{p,q}}$ and ${\rm ind}({\rm coh}\mbox{-}\mathbb{X}(p,q))$ are the mesh categories of their translation quivers $(\Gamma_{A_{p,q}}, \tau^{\prime})$ and $(\Gamma_{{\rm coh}\mbox{-}\mathbb{X}(p,q)}, \tau)$ respectively, therefore, statement $(1)$ implies that $F$ is an equivalence of categories. In particular, this equivalence induces the structure of an abelian category on $\mathcal{C}_{A_{p,q}}$. With respect to this structure, $F$ is an exact functor since every equivalence between abelian categories is exact.
\end{proof}

\subsection{Positive intersection}

In this subsection, we want to give a geometric explanation for the dimension of extension group $\Ext^i(X,Y)$ for $X,Y\in{\rm ind}({\rm coh}\mbox{-}\mathbb{X}(p,q))$. Since ${\rm coh}\mbox{-}\mathbb{X}(p,q)$ is hereditary, the extension group is zero for $i\geq 2$. Hence we only need to consider the case $i=1$.

Note that if $X,Y$ are supported at distinct points, then $\Ext^1(X,Y)=0$. If they are supported at the same ordinary point, then $X=S_{\lambda}^{(j_1)}, Y=S_{\lambda}^{(j_2)}$ for some $\lambda\in\mathbf{k}^*$ and $j_1, j_2\in \mathbb{Z}_{\geq 1}$. In this case, ${\rm dim}_{\mathbf{k}}\Ext^1(X,Y)=\min\{j_1,j_2\}$. Hence, we only consider the remaining cases in the following.
We will show that ${\rm dim}_{\mathbf{k}}\Ext^1(X,Y)$ can be interpreted as the positive intersection number between the corresponding oriented curves in the marked annulus $A_{p,q}$.

\begin{defn}[\cite{W2008}]\label{positive intersection}
A point of intersection of two oriented curves $\alpha$ and $\beta$ in $A_{p,q}$ is called \emph{positive intersection} of $\alpha$ and $\beta$, if $\beta$ intersects $\alpha$ from the right, i.e., the picture looks like as:
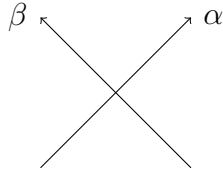
\begin{figure}[H]
\begin{tikzpicture}
\draw[->](0,0)--(2,2);
\draw[->](2,0)--(0,2);
\node()at(-0.3,2){$\beta$};
\node()at(2.3,2){$\alpha$};
\end{tikzpicture}
\caption{Positive intersection}
\end{figure}
\end{defn}
Denoted by
$$I^{+}(\alpha,\beta)=\mathop{{\rm min}}\limits_{\alpha^{\prime}\sim\alpha,\beta^{\prime}\sim\beta}|\alpha^{\prime}\cap^{+}\beta^{\prime}|,$$
called the \emph{positive intersection number} of $\alpha,\beta$, where $\alpha^{\prime}\cap^{+}\beta^{\prime}$ is the set of the positive intersections of $\alpha^{\prime}$ and $\beta^{\prime}$, excluding their endpoints, and the sign $a\sim b$ means that the two curves are homotopy. It follows that $I^{+}(\alpha,\beta)=0$ or 1 if $\alpha, \beta$ are oriented arcs in $\mathbb{U}$.

\begin{thm}\label{dimension and positive intersection}
Let $X, Y$ be two indecomposable objects in ${\rm coh}\mbox{-}\mathbb{X}(p,q).$
Assume $X, Y$ are not supported at ordinary points.
Then
\begin{equation}\label{positive intersection theorem}{\rm dim_{\mathbf{k}}Ext^{1}}(X,Y)=I^{+}(\phi^{-1}(X),\phi^{-1}(Y)).
\end{equation}
\end{thm}

\begin{proof}
If $X, Y\in \mathcal{U}_{\infty}\coprod \mathcal{U}_{0}$, then the result follows from \cite[Theorem 3.8]{BM2012}.
If $X\in{\rm vec}\mbox{-}\mathbb{X}(p,q)$ and $Y\in \mathcal{U}_{\infty}\coprod \mathcal{U}_{0}$,
then according to Proposition \ref{the properties of coherent sheaves} (2), we have ${\rm Ext}^{1}(X, Y)=0$, and by Definition \ref{positive intersection}, we also have $I^{+}(\phi^{-1}(X),\phi^{-1}(Y))=0.$
So we only need to consider the following two cases:

\vspace{1mm}(1) $X, Y\in {\rm vec}\mbox{-}\mathbb{X}(p,q)$.

\vspace{2mm} By Proposition \ref{the properties of coherent sheaves} (3), we know that each indecomposable bundle over $\mathbb{X}(p,q)$ is a line bundle, thus $X=\co(\vec{x})$ and $Y=\co(\vec{y})$ for some $\vec{x},\,\vec{y}\in\mathbb{L}(p,q).$
Write $\vec{x}=l_1\vec{x}_1+l_2\vec{x}_2+l\vec{c}$ and $\vec{y}=k_1\vec{x}_1+k_2\vec{x}_2+k\vec{c}$ in normal forms.
According to Proposition \ref{the properties of coherent sheaves} $(1)$ and $(5)$,
we have
$$\Ext^{1}(\co(\vec{x}),\co(\vec{y}))\cong D\Hom(\co(\vec{y}), \co(\vec{x}-\vec{x}_1-\vec{x}_2))\cong S_{\vec{x}-\vec{x}_1-\vec{x}_2-\vec{y}}.$$
Observe that $\vec{x}-\vec{x}_1-\vec{x}_2-\vec{y}=(l_1-k_1-1)\vec{x}_1+(l_2-k_2-1)\vec{x}_2+(l-k)\vec{c},$ and $-p\leq l_1-k_1-1\leq p-2,\,\, -q\leq l_2-k_2-1\leq q-2.$ Denote $h_{1}:=l_1-k_1-1$ and $h_{2}:=l_2-k_2-1$, we get
\begin{equation}
{\rm dim_{\mathbf{k}}Ext^{1}}(X,Y)=
  \left\{
  \begin{array}{ll}

    l-k+1& \;{\rm if}\;\,0\leq h_{1}\leq p-2,\, 0\leq h_{2}\leq q-2,\, k\leq l;\nonumber\\
   l-k& \;{\rm if}\;\,-p\leq h_{1}\leq -1,\, 0\leq h_{2}\leq q-2,\, k\leq l;\\
   l-k& \;{\rm if}\;\,0\leq h_{1}\leq p-2,\, -q\leq h_{2}\leq -1,\, k\leq l;\\
   l-k-1& \;{\rm if}\;-p\leq h_{1}\leq -1,\, -q\leq h_{2}\leq -1,\, k+1\leq l;\\
   0& {\rm otherwise.}
  \end{array}
\right.\end{equation}

On the other hand, we consider $I^{+}(\phi^{-1}(X),\phi^{-1}(Y)).$ Since  $\phi^{-1}(X)=[D^{\frac{l_1}{p}}_{-\frac{l_2}{q}-l}]$ and $\phi^{-1}(Y)=[D^{\frac{k_1}{p}}_{-\frac{k_2}{q}-k}].$
By similar arguments as in \cite{BM2012}, a point of intersection of $\phi^{-1}(X)$ and $\phi^{-1}(Y)$ is positive if and only if
there exists $m\in \mathbb{Z}$ such that $\frac{l_1}{p}+m > \frac{k_1}{p}$ and $-\frac{l_2}{q}-l+m < -\frac{k_2}{q}-k$ (c.f. Figure \ref{case1}).

\begin{figure}[H]
\begin{tikzpicture}
\draw[->](0,0)--(2,2);
\draw[->](2,0)--(0,2);
\node()at(-0.1,2.3){$\frac{k_1}{p}$};
\node()at(2.1,2.3){$\frac{l_1}{p}+m$};
\node()at(-0.1,-0.3){$-\frac{l_2}{q}-l+m$};
\node()at(2.1,-0.3){$-\frac{k_2}{q}-k$};
\end{tikzpicture}
\caption{}\label{case1}
\end{figure}
\noindent These two inequalities can be combined into $$\frac{k_1-l_1}{p}\;\textless\; m\;\textless \;\frac{q(l-k)+l_2-k_2}{q}.$$

\noindent Therefore,
$$I^{+}(\phi^{-1}(X),\phi^{-1}(Y))=|\{m\in\mathbb{Z}|\frac{k_1-l_1}{p}\;\textless\; m\;\textless \;\frac{q(l-k)+l_2-k_2}{q}\}|.$$

In order to prove that ${\rm dim_{\mathbf{k}}Ext^{1}}(X,Y)=I^{+}(\phi^{-1}(X),\phi^{-1}(Y))$, we only consider the case: $0\leq h_{1}\leq p-2,\, 0\leq h_{2}\leq q-2,\, k\leq l$; the others being similar. In this case, we have
$$I^{+}(\phi^{-1}(X),\phi^{-1}(Y))=|\{m\in\mathbb{Z}|0\leq m\leq l-k\}|=l-k+1={\rm dim_{\mathbf{k}}Ext^{1}}(X,Y).$$
We are done.

(2) $X\in \mathcal{U}_{\infty}\coprod \mathcal{U}_{0}$ and $Y\in {\rm vec}\mbox{-}\mathbb{X}(p,q)$.

We only consider the case $X\in \mathcal{U}_{\infty}$, that is, $X=S_{\infty,i}^{(j)}$ for $i\in\mathbb{Z}/p\mathbb{Z},\,j \in \mathbb{Z}_{\geq 1}$, the case $X\in \mathcal{U}_{0}$ being similar.

By Proposition \ref{the properties of coherent sheaves} (3), $Y=\co(\vec{x})$ for some $\vec{x}\in\mathbb{L}(p,q),$ write $\vec{x}=l_1\vec{x}_1+l_2\vec{x}_2+l\vec{c}$ in normal form. By exact sequence (\ref{important exact sequence}), we get ${\rm dim_{\mathbf{k}}Hom}(\co(\vec{x}+\vec{x}_1+\vec{x}_2), S_{\infty,l_1+1})=1.$
Notice that $S_{\infty,i}^{(j)}$ has a composition series of the form
$$S_{\infty, i-j+1} \hookrightarrow S_{\infty, i-j+2}^{(2)} \hookrightarrow \cdots\hookrightarrow S_{\infty, i-2}^{(j-2)} \hookrightarrow S_{\infty, i-1}^{(j-1)}\hookrightarrow  S_{\infty, i}^{(j)}$$
with $S_{\infty, i-k}^{(j-k)}/S_{\infty, i-k-1}^{(j-k-1)}\cong S_{\infty, i-k}$ for $0\leq k\leq j-2.$
Write $j=np+r$ with $0\leq r\leq p-1.$ We get
\begin{equation}
{\rm dim_{\mathbf{k}}Ext^{1}}(S_{\infty,i}^{(j)},\co(\vec{x}))=
  \left\{
  \begin{array}{ll}
   n& \;{\rm if}\; p\nmid i-l_1-k\;{\rm for\;each}\;1\leq k\leq r,\nonumber\\
 n+1& {\rm otherwise}.
  \end{array}
\right.\end{equation}

On the other hand, we consider $I^{+}(\phi^{-1}(S_{\infty,i}^{(j)}),\phi^{-1}(\co(\vec{x}))).$ Since
$$\phi^{-1}(S_{\infty,i}^{(j)})=[D^{\frac{i-j-1}{p},\frac{i}{p}}],\quad\quad\phi^{-1}(\co(\vec{x}))=[D^{\frac{l_1}{p}}_{-\frac{l_2}{q}-l}].$$ By similar arguments as in \cite{BM2012}, a point of intersection of $\phi^{-1}(S_{\infty,i}^{(j)})$ and $\phi^{-1}(\co(\vec{x}))$ is positive if and only if there exists $m\in\mathbb{Z}$ such that $\frac{i-j-1}{p}\;\textless\; \frac{l_1}{p}+m\;\textless \; \frac{i}{p}$, (c.f. Figure \ref{case2}). \begin{figure}[H]
\begin{tikzpicture}
\draw[->](0.8,-0.4)--(2,2.2);
\draw[->](0.1,2)arc(240:300:3.2);
\node()at(-0.1,2.3){$\frac{i-j-1}{p}$};
\node()at(2.1,2.5){$\frac{l_1}{p}+m$};
\node()at(3.45,2.3){$\frac{i}{p}$};
\node()at(0.6,-0.75){$-\frac{l_2}{q}-l+m$};
\end{tikzpicture}
\caption{}\label{case2}
\end{figure}
{\noindent Hence $$I^{+}(\phi^{-1}(S_{\infty,i}^{(j)}),\phi^{-1}(\co(\vec{x})))=|\{m\in\mathbb{Z}|i-j-l_1\leq pm \leq i-l_1-1\}|.
$$
Observe that
$$\lfloor \frac{(i-l_1-1)-(i-j-l_1)+1}{p}\rfloor
=\lfloor \frac{j}{p}\rfloor
=\lfloor \frac{np+r}{p}\rfloor
=n+\lfloor \frac{r}{p}\rfloor=n.
$$
It follows that
\begin{equation*}
I^{+}(\phi^{-1}(S_{\infty,i}^{(j)}),\phi^{-1}(\co(\vec{x})))=
  \left\{
  \begin{array}{ll}
   n& \;{\rm if}\; p\nmid i-l_1-k\;{\rm for\;each}\;1\leq k\leq r,\nonumber\\
 n+1& {\rm otherwise}.
  \end{array}
\right.\end{equation*}}

\noindent Hence ${\rm dim_{\mathbf{k}}Ext^{1}}(X,Y)=I^{+}(\phi^{-1}(X),\phi^{-1}(Y))$.
This finishes the proof.
\end{proof}

\begin{rem}
If one of $X,Y$ is supported at an ordinary point, then \eqref{positive intersection theorem} still holds.
\end{rem}

\begin{prop}\label{exact seq}
Let $\alpha,\beta$ be positive bridging arcs or peripheral arcs in $\mathbb{U}$ with $I^{+}(\alpha,\beta)=1$. Then there is a natural short exact sequence in $ {\rm coh}\mbox{-}\mathbb{X}(p,q)$ associated to the intersection in geometric terms. More precisely, consider the following figure of the chosen intersection, 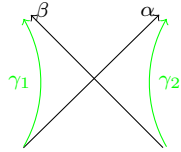
\begin{figure}[H]
\begin{tikzpicture}
\node(1)at(-0.1,-0.06){\tiny{}};
\node(2)at(2.05,-0.06){\tiny{}};
\node(3)at(0,2){\tiny{}};
\node(4)at(2,2){\tiny{}};
\draw[->](1)--(4);
\draw[->](2)--(3);
\draw[->,green](0.05,0.1) arc(330:390:1.7);
\draw[->,green](1.95,0.1) arc(210:150:1.7);
\node()at(1.7,1.9){\tiny{$\alpha$}};
\node()at(0.3,1.9){\tiny{$\beta$}};
\node()at(2,0.95){\tiny{${\green \gamma_2}$}};
\node()at(0,0.95){\tiny{${\green \gamma_1}$}};
\end{tikzpicture}
\caption{Short exact sequence}\label{labelx}
\end{figure}
\noindent we have a short exact sequence
\begin{align}\label{exact sequence}
0\longrightarrow \phi(\pi(\beta))\longrightarrow \phi(\pi(\gamma_1))\oplus\phi(\pi(\gamma_2))\longrightarrow\phi(\pi(\alpha))\longrightarrow 0.
\end{align}
\end{prop}

\begin{proof}
We discuss case by case according to the types of $\alpha$ and $\beta$. If $\alpha,\,\beta$ are peripheral arcs, then the result follows from \cite[Section 6]{BBM2014} and \cite[Remark 4.25]{W2008}. Hence we only consider the following two cases.

(1) $\alpha,\, \beta$ are positive bridging arcs.

Then $\alpha=D^{\frac{i}{p}}_{\frac{j}{q}}$ and $\beta=D^{\frac{k}{p}}_{\frac{l}{q}}$ for some integers $k\textless i$ and $j\textless l.$ Hence $\gamma_1=D^{\frac{k}{p}}_{\frac{j}{q}}$ and $\gamma_2=D^{\frac{i}{p}}_{\frac{l}{q}}$.
Note that $$\phi(\pi(\alpha))=\mathcal{O}(i\vec{x}_1-j\vec{x}_2),\; \quad \phi(\pi(\beta))=\mathcal{O}(k\vec{x}_1-l\vec{x}_2).$$
And we have the following short exact sequence in ${\rm coh}\mbox{-}\mathbb{X}(p,q)$:
$$0\longrightarrow \mathcal{O}(k\vec{x}_1-l\vec{x}_2)\longrightarrow \mathcal{O}(k\vec{x}_1-j\vec{x}_2)\oplus\mathcal{O}(i\vec{x}_1-l\vec{x}_2)\longrightarrow\mathcal{O}(i\vec{x}_1-j\vec{x}_2)\longrightarrow 0.$$
Then (\ref{exact sequence}) holds since $$\phi(\pi(\gamma_1))=\mathcal{O}(k\vec{x}_1-j\vec{x}_2),\; \quad \phi(\pi(\gamma_2))=\mathcal{O}(i\vec{x}_1-l\vec{x}_2).$$

(2) $\alpha$ is a peripheral arc and $\beta$ is a positive bridging arc.

\vspace{1mm} We only consider the case that both the endpoints of $\alpha$ are in the upper boundary of $\mathbb{U}$, the other being similar.

By assumption, $\alpha=D^{\frac{i-j-1}{p},\,\frac{i}{p}}$ and $\beta=D^{\frac{k}{p}}_{\frac{l}{q}}$ for some $i,j, k, l\in \mathbb{Z}$ with $i-j-1\;\textless \; k\textless\;i.$ Then we have
$$\phi(\pi(\alpha))=S_{\infty,i}^{(j)},\; \phi(\pi(\beta))=\mathcal{O}(k\vec{x}_1-l\vec{x}_2).$$
In ${\rm coh}\mbox{-}\mathbb{X}(p,q)$, there is a short exact sequence:
$$0\longrightarrow \mathcal{O}(k\vec{x}_1-l\vec{x}_2)\longrightarrow \mathcal{O}(i\vec{x}_1-l\vec{x}_2)\oplus S_{\infty, k}^{(k+j-i)}\longrightarrow S_{\infty,i}^{(j)}\longrightarrow 0.$$
Therefore, the short exact sequence (\ref{exact sequence}) holds.
\end{proof}

\section{The geometric interpretation of tilting sheaf}

In this section, we investigate the correspondence between tilting sheaves in the category ${\rm coh}\mbox{-}\mathbb{X}(p,q)$ and triangulations in the marked annulus $A_{p,q}$, and then study the relationship between the flip of an arc in a triangulation and the mutation of indecomposable direct summand of the corresponding tilting sheaf.

\subsection{Tilting sheaves and triangulations}
First let us recall from \cite{GL87, GengSF2020} for the definition of tilting sheaves and tilting bundles in ${\rm coh}\mbox{-}\mathbb{X}(p,q)$ as follows.

\begin{defn}\label{tilting} A sheaf $T$ in ${\rm coh}\mbox{-}\mathbb{X}(p,q)$ is called a \emph{tilting sheaf} if

\begin{itemize}
  \item[(1)] $T$ is rigid, that is, ${\rm Ext}^{1}(T, T)=0$.
  \item[(2)] For any object $X\in{\rm coh}\mbox{-}\mathbb{X}(p,q)$, the condition ${\rm Hom}(T,X)=0={\rm Ext}^{1}(T,X)$ implies that $X=0$.
\end{itemize}

Moreover, if a tilting sheaf $T$ is given by a vector bundle, i.e., $T$ has no direct summand of finite length, then $T$ is called a \emph{tilting bundle}.
\end{defn}

\begin{prop}[\cite{HR1999}]\label{tilting2}
 Assume that $\mathcal{H}$ is a hereditary abelian category with a tilting object. Then $T$ is a tilting object in $\mathcal{H}$ if and only if ${\rm Ext}^{1}(T,T)=0$ and the number of non-isomorphic indecomposable direct summands of $T$ equals the rank of the Grothendieck group $K_{0}(\mathcal {H})$.
\end{prop}

Geigle-Lenzing \cite{GL87} constructed a canonical tilting sheaf for any weighted projective line. As a consequence, the rank of the Grothendieck group $K_{0}({\rm coh}\mbox{-}\mathbb{X}(p,q))$ equals to $p+q$, and any tilting sheaf in ${\rm coh}\mbox{-}\mathbb{X}(p,q)$ contains $p+q$ indecomposable direct summands. A rigid sheaf in ${\rm coh}\mbox{-}\mathbb{X}(p,q)$ with $p+q-1$-many (pairwise non-isomorphic) indecomposable direct summands is called an \emph{almost complete tilting sheaf}. Let $\overline{T}$ be an almost complete tilting sheaf in ${\rm coh}\mbox{-}\mathbb{X}(p,q)$.
If there exists an indecomposable sheaf $E$, such that $\overline{T}\oplus E$ is a tilting sheaf, then $E$ is called a \emph{complement} of $\overline{T}$.

Recall from \cite{ABGP2010} that a \emph{triangulation} $\Gamma$ of $A_{p,q}$ is a maximal collection of arcs that do not intersect in the interior of $A_{p,q}$.
According to Proposition \cite[Proposition 2.1]{ABGP2010}, any triangulation $\Gamma$ of $A_{p,q}$ consists of $p+q$-many arcs.
In our model, we always fix the orientations for the arcs in $\Gamma$, namely, we take the arcs from the set $\mathcal{C}$ appeared in Theorem \ref{corresponding}. For example, all the bridging arcs are oriented from the outer boundary to the inner boundary.

\begin{thm}\label{triangulation and tilting}
The tilting sheaves in the category ${\rm coh}\mbox{-}\mathbb{X}(p,q)$ are in natural bijection with the triangulations of $A_{p,q}$.
\end{thm}

\begin{proof}
Let $\Gamma$ be a triangulation of $A_{p,q}$ and $\gamma_1, \gamma_2, \cdots, \gamma_{p+q}$ be the set of arcs in $\Gamma$. Let $$T=\bigoplus\limits_{i=1}^{p+q}\phi(\gamma_i).$$
According to Theorem \ref{dimension and positive intersection} and the definition of triangulation, we have $${\rm dim_{\mathbf{k}}Ext^{1}}(\phi(\gamma_i),\, \phi(\gamma_j))=I^{+}(\gamma_i,\, \gamma_j)=0$$ for $1\leq i,\, j\leq p+q$. It follows that ${\rm dim_{\mathbf{k}}Ext^{1}}(T,T)=0.$ Combining with Proposition \ref{tilting2} and Propositions \ref{the properties of coherent sheaves} $(6)$, we get that $T=\bigoplus\limits_{i=1}^{p+q}\phi(\gamma_i)$ is a tilting sheaf in ${\rm coh}\mbox{-}\mathbb{X}(p,q)$.

On the other hand, if $T=\bigoplus\limits_{i=1}^{p+q}T_{i}$ is a tilting sheaf in ${\rm coh}\mbox{-}\mathbb{X}(p,q)$, then we have $\Ext^{1}(T,T)=0.$ By Theorem \ref{dimension and positive intersection}, $\{\phi^{-1}(T_1), \phi^{-1}(T_2), \cdots, \phi^{-1}(T_{p+q})\}$ is a collection of arcs that do not intersect in the interior of $A_{p,q}$, which provides a triangulation of $A_{p,q}$ by Proposition \cite[Proposition 2.1]{ABGP2010}. Therefore, we complete the proof of Theorem.
\end{proof}

\begin{rem}
The following general result has been stated in \cite{Hubner1996} for weighted projective lines of arbitrary type. We give a short proof for weight type $(p,q)$ by using the correspondence in Theorem \ref{triangulation and tilting}.
\end{rem}

\begin{cor}\label{two complements}
Let $\overline{T}$ be an almost complete tilting sheaf in the category ${\rm coh}\mbox{-}\mathbb{X}(p,q)$. Then there exist precisely two complements of $\overline{T}$.
\end{cor}

\begin{proof}
By the proof of Theorem \ref{triangulation and tilting}, we see that $\overline{T}$ can be represented by $p+q-1$ non-intersecting arcs in $\mathcal{C}.$ To get a triangulation of $A_{p,q}$, we should add just one further arc, which must be a diagonal in a quadrilateral. Since there are exactly two diagonals in a quadrilateral, we get two arcs belong to $\mathcal{C}$ by attaching suitable orientations and represent the two complements of $\overline{T}$.
\end{proof}

\subsection{Flips and mutations}

In this subsection, we study the compatibility of the flip of an arc in a triangulation and the mutation of indecomposable direct summand of the corresponding tilting sheaf.

\begin{defn}(\cite{GengSF2020})
Let $T=\overline{T}\oplus X$ and $T^{\prime}=\overline{T}\oplus X^{\prime}$ be two tilting sheaves in ${\rm coh}\mbox{-}\mathbb{X}(p,q)$, where $X$ and $X^{\prime}$ are two non-isomorphic indecomposable objects in ${\rm coh}\mbox{-}\mathbb{X}(p,q)$. Then $T^{\prime}$ is called the \emph{mutation} of $T$ at $X$ and denoted by $T^{\prime}=\mu_{X}(T)$.
\end{defn}

In this paper, we call $X^{\prime}$ the \emph{mutation} of $X$ w.r.t $T$ and denote by $\mu_{T}(X)=X^{\prime}.$

Let $\gamma$ be an arc of a triangulation $\Gamma$ of $A_{p,q}$. Then $\gamma$ is one diagonal of the quadrilateral formed by the two triangles of $\Gamma$ that contain $\gamma$. Recall from \cite{ABGP2010} that the \emph{flip} of $\gamma$ replaces the arc $\gamma$ by the other diagonal $\gamma^{\prime}$ of the same quadrilateral (c.f. Figure \ref{figure of flip}). Keeping all other arcs unchanged, one obtains a new triangulation $\mu_{\gamma}(\Gamma)$. For convenient, we call $\gamma^{\prime}$ the \emph{flip} of $\gamma$ w.r.t $\Gamma$, and denote by $\gamma^{\prime}=\mu_{\Gamma}(\gamma).$
\begin{figure}[H]
\begin{tikzpicture}
\draw[](0,0)--(1,1);
\draw[](0,0)--(1,-1);
\draw[](1,1)--(2,0);
\draw[](1,-1)--(2,0);
\node()at(0,0){\tiny$\bullet$};
\node()at(1,1){\tiny$\bullet$};
\node()at(1,-1){\tiny$\bullet$};
\node()at(2,0){\tiny$\bullet$};
\draw[red](1,-0.95)--(1,0.95);
\node()at(1.2,0){${\red \gamma}$};
\draw[<->](2.8,0)--(3.4,0);
\draw[](4,0)--(5,1);
\draw[](4,0)--(5,-1);
\draw[](5,1)--(6,0);
\draw[](5,-1)--(6,0);
\node()at(4,0){\tiny$\bullet$};
\node()at(5,1){\tiny$\bullet$};
\node()at(5,-1){\tiny$\bullet$};
\node()at(6,0){\tiny$\bullet$};
\draw[red](4.05,0)--(5.95,0);
\node()at(5,0.2){${\red \gamma^{\prime}}$};
\end{tikzpicture}
\caption{Flip of an arc $\gamma$}\label{figure of flip}
\end{figure}
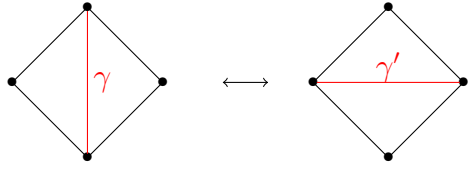

For a triangulation $\Gamma$ of $A_{p,q}$, by choosing suitable orientation for bridging arcs, we can always assume $\mu_{\gamma}(\Gamma)$ is also a triangulation of $A_{p,q}$.
By Theorem \ref{triangulation and tilting}, we know that a triangulation $\Gamma$ corresponds to a tilting sheaf $T$ in ${\rm coh}\mbox{-}\mathbb{X}(p,q)$, and $\phi(\gamma)$ is an indecomposable direct summand of $T$. The following proposition shows that the flip of $\gamma$ and the mutation of $\phi(\gamma)$ are compatible.

\begin{prop}\label{mutation and flip}
Let $\Gamma$ be a triangulation of $A_{p,q}$ and $\gamma$ be an arc in $\Gamma$. Then $$\phi(\mu_{\Gamma}(\gamma))=\mu_{T}(\phi(\gamma)).$$
\end{prop}

\begin{proof}
Let $\gamma,\,\gamma_1,\,\gamma_2,\,\cdots, \gamma_{p+q-1}$ be the set of arcs in the triangulation $\Gamma$. Therefore, $\mu_{\gamma}(\Gamma)$ is a triangulation of $A_{p,q}$ with arcs $\gamma^{\prime},\,\gamma_1,\,\gamma_2,\,\cdots, \gamma_{p+q-1}$. By Theorem \ref{triangulation and tilting}, we get two tilting sheaves $$T=\phi(\gamma)\oplus\phi(\gamma_1)\oplus\phi(\gamma_2)\oplus\cdots\oplus\phi(\gamma_{p+q-1})$$
and $$T^{\prime}=\phi(\gamma^{\prime})\oplus\phi(\gamma_1)\oplus\phi(\gamma_2)\oplus\cdots\oplus\phi(\gamma_{p+q-1})$$
in ${\rm coh}\mbox{-}\mathbb{X}(p,q)$. It follows that $\phi(\mu_{\Gamma}(\gamma))=\phi(\gamma^{\prime})=\mu_{T}(\phi(\gamma))$.
\end{proof}

\section{Tilting bundles and bundle-mutations}

In this section, we use the geometric model to investigate tilting bundles in the category ${\rm coh}\mbox{-}\mathbb{X}(p,q)$, basing on Theorem \ref{triangulation and tilting} and Proposition \ref{mutation and flip}.
Denote by
$$\mathcal{T}^{\nu}_{\mathbb{X}}:=\{{\rm Tilting \ bundles  \ in} \ {\rm coh}\mbox{-}\mathbb{X}(p,q)\}.$$

\subsection{A combinatorial description of $\mathcal{T}^{\nu}_{\mathbb{X}}$}

Let $T=\bigoplus\limits_{i=1}^{p+q}T_{i}$ be a tilting bundle in ${\rm coh}\mbox{-}\mathbb{X}(p,q)$. Under Theorem \ref{corresponding}, we can view indecomposable objects in the category ${\rm coh}\mbox{-}\mathbb{X}(p,q)$ as oriented curves in the marked annulus $A_{p,q}$. Hence each $T_{i}$ corresponds to a positive bridging arc $[D^{a_{i}}_{b_{i}}]$, where $a_{i}\in \frac{\mathbb{Z}}{p}$ and $b_{i}\in \frac{\mathbb{Z}}{q}$ for $1\leq i\leq p+q$. Moreover, by Theorem \ref{triangulation and tilting}, $T$ corresponds to a triangulation $\Gamma$ with arcs $[D^{a_{i}}_{b_{i}}], 1\leq i\leq p+q.$ Therefore, $\{D^{a_{i}}_{b_{i}}, 1\leq i\leq p+q\}$ is a triangulation of a parallelogram in $\mathbb{U}$.
Thanks to \eqref{periodicity of covering map}, without loss of generality we may assume $T$ has the following \emph{normal form}:
\begin{equation}\label{assumption on chains}
0=a_1< a_2\leq \cdots\leq a_{p+q}=1, {\quad} b_1\leq b_2\leq \cdots\leq b_{p+q}.
\end{equation}

\begin{conven}\label{order}
From now on, for a tilting bundle $T=\bigoplus\limits_{i=1}^{p+q}T_i$ in ${\rm coh}\mbox{-}\mathbb{X}(p,q)$, we always assume $T=\bigoplus\limits_{i=1}^{p+q}[D^{a_{i}}_{b_{i}}]$ satisfying \eqref{assumption on chains}, and denote by $\mu_{i}(T)$ for the mutation of $T$ at $T_{i}$, where $i$ is taken modulo $p+q$.
In particular, $\mu_{i-1}(T)=\mu_{p+q}(T)$ for $i=1$, and $\mu_{i+1}(T)=\mu_{1}(T)$ for $i=p+q.$
\end{conven}

We will give a combinatorial description for the set of tilting bundles $\mathcal{T}^{\nu}_{\mathbb{X}}$ in ${\rm coh}\mbox{-}\mathbb{X}(p,q)$. For this, we consider the following set
$$\Lambda_{(p,q)}^0:=\{(c_1, c_2, \cdots, c_p)\in\mathbb{Z}^{p}|c_1\leq c_2\leq \cdots\leq c_p\leq c_1+q\}.$$

\begin{thm}\label{bijection between tilting bundle and combination vertices}
There exists a bijection between $\mathcal{T}^{\nu}_{\mathbb{X}}$ and $\Lambda_{(p,q)}^0$.
\end{thm}

\begin{proof}
For any tilting bundle $T$, by Theorem \ref{triangulation and tilting}, $T$ corresponds to a triangulation $\Gamma$ of (a parallelogram in) $\mathbb{U}$.
Then for any $1\leq i\leq p$, the segment $[(\frac{i-1}{p})_{\partial}, (\frac{i}{p})_{\partial}]$ belongs to a unique triangle in $\Gamma$:

\begin{figure}[H]
\begin{tikzpicture}
\draw[-](0,0)--(0.3,1.8);
\draw[-](0,0)--(6,0);
\draw[-](6,0)--(6.3,1.8);
\draw[-](0.3,1.8)--(6.3,1.8);
\node()at(0,0){\tiny$\bullet$};
\node()at(6,0){\tiny$\bullet$};
\node()at(0.3,1.8){\tiny$\bullet$};
\node()at(6.3,1.8){\tiny$\bullet$};
\node()at(1.1,1.8){\tiny$\bullet$};
\draw[-](0,0)--(1.1,1.8);
\node()at(2.7,1.8){\tiny$\bullet$};
\node()at(3.5,1.8){\tiny$\bullet$};
\node()at(4.3,1.8){\tiny$\bullet$};
\node()at(2.3,0){\tiny$\bullet$};
\node()at(3.7,0){\tiny$\bullet$};
\draw[-,red](2.7,1.8)--(2.3,0);
\draw[-,red](3.5,1.8)--(2.3,0);
\draw[-](3.5,1.8)--(3.7,0);
\draw[-](4.3,1.8)--(3.7,0);
\node()at(5.5,1.8){\tiny$\bullet$};
\node()at(4.8,0){\tiny$\bullet$};
\draw[-](5.5,1.8)--(4.8,0);
\draw[-](6.3,1.8)--(4.8,0);
\node()at(0,-0.3){\tiny$\frac{c_1}{q}$};

\node()at(2.3,-0.3){\tiny$\frac{c_i}{q}$};
\node()at(3.8,-0.3){\tiny$\frac{c_{i+1}}{q}$};
\node()at(4.8,-0.3){\tiny$\frac{c_p}{q}$};
\node()at(0.3,2.1){\tiny$0$};
\node()at(1.1,2.1){\tiny$\frac{1}{p}$};
\node()at(2.7,2.1){\tiny$\frac{i-1}{p}$};
\node()at(3.5,2.1){\tiny$\frac{i}{p}$};
\node()at(4.3,2.1){\tiny$\frac{i+1}{p}$};
\node()at(5.5,2.1){\tiny$\frac{p-1}{p}$};
\node()at(6.3,2.1){\tiny$1$};
\node()at(6,-0.3){\tiny$\frac{c_1}{q}+1$};
\draw[line width=1pt,dotted](0.8,-0.2)--(1.4,-0.2);
\draw[line width=1pt,dotted](4.2,-0.2)--(4.5,-0.2);
\draw[line width=1pt,dotted](5.1,-0.2)--(5.5,-0.2);
\draw[line width=1pt,dotted](1.65,2)--(2.15,2);
\draw[line width=1pt,dotted](4.7,2)--(5.1,2);
\draw[line width=1pt,dotted](2.7,-0.2)--(3.3,-0.2);
\node()at(-1,0){\tiny$\partial^{\prime}$};
\node()at(-1,1.8){\tiny$\partial$};
\end{tikzpicture}
\caption{$p$-triangles}\label{p-triangles}
\end{figure}
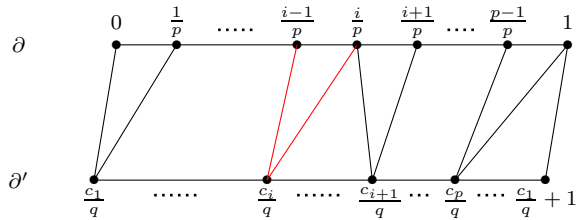
\noindent Therefore, we obtain a sequence $(c_1, c_2, \cdots, c_p)\in\mathbb{Z}^{p}$. Since $\Gamma$ is a triangulation, we obtain that $\frac{c_1}{q}\leq \frac{c_2}{q}\cdots\leq \frac{c_p}{q}\leq\frac{c_1}{q}+1,$ i.e., $c_1\leq c_2\cdots\leq c_p\leq c_1+q.$ That is, $(c_1, c_2, \cdots, c_p)\in \Lambda_{(p,q)}^0.$
This defines a map
\begin{equation}\label{the map phi}
\varphi: \mathcal{T}^{\nu}_{\mathbb{X}}\longrightarrow \Lambda_{(p,q)}^0; \quad T\mapsto (c_1, c_2, \cdots, c_p).\end{equation}

By Theorem \ref{triangulation and tilting}, we know that $\varphi$ is injective.
Now we show that it is surjective.

For any $(c_1, c_2, \cdots, c_p)\in \Lambda_{(p,q)}^0$, we can obtain $p$-triangles as in Figure \ref{p-triangles}.
For any marked point $c_{\partial^{\prime}}$ satisfying $\frac{c_i}{q}\textless c\textless \frac{c_{i+1}}{q}$, there exists a unique bridging arc connecting to $c$ which does not intersect with the above $p$-triangles, namely, the arc $D^{\frac{i}{p}}_{c}$ (see the following picture):
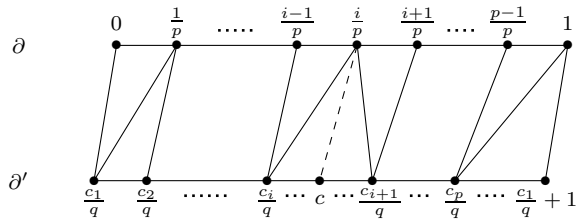
\begin{figure}[H]
\begin{tikzpicture}
\draw[-](0,0)--(0.3,1.8);
\draw[-](0,0)--(6,0);
\draw[-](6,0)--(6.3,1.8);
\draw[-](0.3,1.8)--(6.3,1.8);
\node()at(0,0){\tiny$\bullet$};
\node()at(6,0){\tiny$\bullet$};
\node()at(0.3,1.8){\tiny$\bullet$};
\node()at(6.3,1.8){\tiny$\bullet$};
\node()at(1.1,1.8){\tiny$\bullet$};
\node()at(0.7,0){\tiny$\bullet$};
\draw[-](0,0)--(1.1,1.8);
\node()at(2.7,1.8){\tiny$\bullet$};
\node()at(3.5,1.8){\tiny$\bullet$};
\node()at(4.3,1.8){\tiny$\bullet$};
\node()at(2.3,0){\tiny$\bullet$};
\node()at(3,0){\tiny$\bullet$};
\node()at(3.7,0){\tiny$\bullet$};
\draw[-](2.7,1.8)--(2.3,0);
\draw[-](3.5,1.8)--(2.3,0);
\draw[-](3.5,1.8)--(3.7,0);
\draw[-](4.3,1.8)--(3.7,0);
\draw[-](1.1,1.8)--(0.7,0);
\draw[-,dashed](3.5,1.8)--(3,0);
\node()at(5.5,1.8){\tiny$\bullet$};
\node()at(4.8,0){\tiny$\bullet$};
\draw[-](5.5,1.8)--(4.8,0);
\draw[-](6.3,1.8)--(4.8,0);
\node()at(0,-0.3){\tiny$\frac{c_1}{q}$};
\node()at(0.7,-0.3){\tiny$\frac{c_2}{q}$};
\node()at(2.3,-0.3){\tiny$\frac{c_i}{q}$};
\node()at(3,-0.25){\tiny$c$};
\node()at(3.8,-0.3){\tiny$\frac{c_{i+1}}{q}$};
\node()at(4.8,-0.3){\tiny$\frac{c_p}{q}$};
\node()at(0.3,2.1){\tiny$0$};
\node()at(1.1,2.1){\tiny$\frac{1}{p}$};
\node()at(2.7,2.1){\tiny$\frac{i-1}{p}$};
\node()at(3.5,2.1){\tiny$\frac{i}{p}$};
\node()at(4.3,2.1){\tiny$\frac{i+1}{p}$};
\node()at(5.5,2.1){\tiny$\frac{p-1}{p}$};
\node()at(6.3,2.1){\tiny$1$};
\node()at(6,-0.3){\tiny$\frac{c_1}{q}+1$};
\draw[line width=1pt,dotted](1.2,-0.2)--(1.8,-0.2);
\draw[line width=1pt,dotted](4.2,-0.2)--(4.5,-0.2);
\draw[line width=1pt,dotted](5.1,-0.2)--(5.5,-0.2);
\draw[line width=1pt,dotted](1.65,2)--(2.15,2);
\draw[line width=1pt,dotted](4.7,2)--(5.1,2);
\draw[line width=1pt,dotted](2.55,-0.2)--(2.85,-0.2);
\draw[line width=1pt,dotted](3.2,-0.2)--(3.5,-0.2);
\node()at(-1,1.8){\tiny$\partial$};
\node()at(-1,0){\tiny$\partial^{\prime}$};
\end{tikzpicture}
\end{figure}
\noindent Therefore, the $p$-triangles in Figure \ref{p-triangles} can be extended to a unique triangulation $\Gamma$ of a parallelogram in $\mathbb{U}$. Then by the bijection between tilting bundles and triangulations, we obtain a tilting bundle $T$, which satisfies $\varphi(T)=(c_1, c_2, \cdots, c_p)$ by the construction. Hence $\varphi$ is surjective. We are done.
\end{proof}

\subsection{Bundle-mutations}

Let $T$ and $T^{\prime}$ be two tilting sheaves in ${\rm coh}\mbox{-}\mathbb{X}(p,q)$ such that $T^{\prime}$ is a mutation of $T$. If $T, \,T^{\prime}$ are both tilting bundles, then such a mutation is called a {\em bundle-mutation}, cf. \cite{GengSF2020}.

\begin{prop}\label{position}
Let $T=\bigoplus\limits_{i=1}^{p+q}T_{i}$ be a tilting bundle in ${\rm coh}\mbox{-}\mathbb{X}(p,q)$.

\begin{itemize}
  \item[(1)] $\mu_{i}(T)$ is a tilting bundle if and only if the positions of the arcs in $\mathbb{U}$ corresponding to $T_{i-1}, T_{i}, T_{i+1}$  exactly satisfy one of the following two conditions:

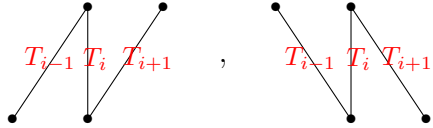
\begin{figure}[H]
\begin{tikzpicture}
\draw[-](-2.5,0)--(-1.5,1.5);
\node()at(-2.5,0){\tiny$\bullet$};
\node()at(-1.5,1.5){\tiny$\bullet$};
\node()at(-1.5,0){\tiny$\bullet$};
\node()at(-0.5,1.5){\tiny$\bullet$};
\draw[-](-1.5,1.5)--(-1.5,0);
\draw[-](-1.5,0)--(-0.5,1.5);
\node()at(-1.4,0.8){\footnotesize${\red T_i}$};
\node()at(-2,0.8){\footnotesize${\red T_{i-1}}$};
\node()at(-0.7,0.8){\footnotesize${\red T_{i+1}}$};
\node()at(0.3,0.75){$,$};

\draw[-](1,1.5)--(2,0);
\draw[-](3,0)--(2,1.5);
\draw[-](2,0)--(2,1.5);
\node()at(2,0){\tiny$\bullet$};
\node()at(1,1.5){\tiny$\bullet$};
\node()at(3,0){\tiny$\bullet$};
\node()at(2,1.5){\tiny$\bullet$};
\node()at(2.1,0.8){\footnotesize${\red T_i}$};
\node()at(1.46,0.8){\footnotesize${\red T_{i-1}}$};
\node()at(2.75,0.8){\footnotesize${\red T_{i+1}}$};
\end{tikzpicture}
\end{figure}

\item[(2)] $\mu_{i}(T)$ is not a tilting bundle if and only if the positions of the arcs in $\mathbb{U}$ corresponding to $T_{i-1}, T_{i}, T_{i+1}$ exactly satisfy one of the following two conditions:

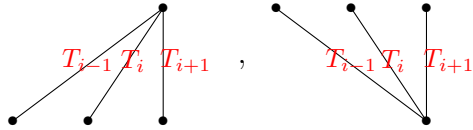
\begin{figure}[H]
\begin{tikzpicture}
\draw[-](-2.5,0)--(-1.5,1.5);
\draw[-](-1.5,1.5)--(-1.5,0);
\draw[-](-3.5,0)--(-1.5,1.5);
\node()at(-2.5,0){\tiny$\bullet$};
\node()at(-1.5,1.5){\tiny$\bullet$};
\node()at(-1.5,0){\tiny$\bullet$};
\node()at(-3.5,0){\tiny$\bullet$};
\node()at(-1.9,0.8){\footnotesize${\red T_i}$};
\node()at(-2.5,0.8){\footnotesize${\red T_{i-1}}$};
\node()at(-1.2,0.8){\footnotesize${\red T_{i+1}}$};
\node()at(-0.45,0.75){$,$};

\draw[-](1,1.5)--(2,0);
\draw[-](2,0)--(2,1.5);
\draw[-](2,0)--(0,1.5);
\node()at(2,0){\tiny$\bullet$};
\node()at(2,1.5){\tiny$\bullet$};
\node()at(1,1.5){\tiny$\bullet$};
\node()at(0,1.5){\tiny$\bullet$};
\node()at(1.56,0.8){\footnotesize${\red T_i}$};
\node()at(1,0.8){\footnotesize${\red T_{i-1}}$};
\node()at(2.3,0.8){\footnotesize${\red T_{i+1}}$};
\end{tikzpicture}
\end{figure}
\end{itemize}
\end{prop}

\begin{proof}
The necessity is obvious. We only prove the sufficiency. By Proposition \ref{mutation and flip}, we have $T_i$ and $\mu_{T}(T_{i})$ are two diagonals of the same quadrilateral. If $\mu_i(T)$ is a tilting bundle, then $\mu_{T}(T_i)$ is a line bundle, it follows that $\mu_{T}(T_i)$ corresponds to a positive bridging arc of $A_{p,q}$ by Theorem \ref{corresponding}. Thus the statement $(1)$ holds. If $\mu_i(T)$ is not a tilting bundle, then $\mu_{T}(T_i)$ corresponds to a peripheral arc of $A_{p,q}$ by Theorem \ref{corresponding}. Therefore, the statement $(2)$ holds.
\end{proof}

\begin{prop}\label{effect of tilting bundle-mutation}
Assume $T=\bigoplus\limits_{i=1}^{p+q}T_{i}$ is a tilting bundle in ${\rm coh}\mbox{-}\mathbb{X}(p,q)$, and $\mu_{i}(T)$ is also a tilting bundle for some $1 \leq i\leq p+q$. Then

\begin{itemize}
  \item[(1)] $\mu_{i-1}(T)$ is a tilting bundle if and only if $\mu_{i-1}(\mu_{i}(T))$ is not a tilting bundle;

  \item[(2)] $\mu_{i+1}(T)$ is a tilting bundle if and only if $\mu_{i+1}(\mu_{i}(T))$ is not a tilting bundle.
  \end{itemize}
\end{prop}

\begin{proof}
We only prove the statement (1), the statement (2) is similar.

If $\mu_{i-1}(T)$ is a tilting bundle, by Proposition \ref{position}, the position of the arcs in $\mathbb{U}$ corresponding to $T_{i-2}, T_{i-1}, T_{i}, T_{i+1}$ exactly satisfies one of the following two conditions:

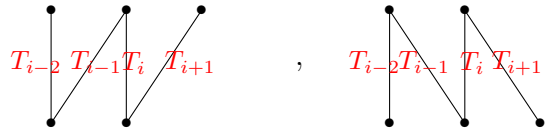
\begin{figure}[H]
\begin{tikzpicture}
\draw[-](-2.5,0)--(-1.5,1.5);
\node()at(-2.5,0){\tiny$\bullet$};
\node()at(-1.5,1.5){\tiny$\bullet$};
\node()at(-1.5,0){\tiny$\bullet$};
\node()at(-2.5,1.5){\tiny$\bullet$};
\node()at(-0.5,1.5){\tiny$\bullet$};
\draw[-](-1.5,1.5)--(-1.5,0);
\draw[-](-1.5,0)--(-0.5,1.5);
\draw[-](-2.5,0)--(-2.5,1.5);
\node()at(-1.4,0.8){\footnotesize ${\red  T_i}$};
\node()at(-1.9,0.8){\footnotesize${\red T_{i-1}}$};
\node()at(-0.65,0.8){\footnotesize${\red T_{i+1}}$};
\node()at(-2.7,0.8){\footnotesize${\red T_{i-2}}$};
\node()at(0.8,0.75){$,$};

\draw[-](2,1.5)--(3,0);
\draw[-](4,0)--(3,1.5);
\draw[-](3,0)--(3,1.5);
\draw[-](2,1.5)--(2,0);
\node()at(3,0){\tiny$\bullet$};
\node()at(2,1.5){\tiny$\bullet$};
\node()at(4,0){\tiny$\bullet$};
\node()at(3,1.5){\tiny$\bullet$};
\node()at(2,0){\tiny$\bullet$};
\node()at(3.1,0.8){\footnotesize${\red T_i}$};
\node()at(2.46,0.8){\footnotesize${\red T_{i-1}}$};
\node()at(3.7,0.8){\footnotesize${\red T_{i+1}}$};
\node()at(1.8,0.8){\footnotesize${\red T_{i-2}}$};
\end{tikzpicture}
\end{figure}
\noindent It is easy to see that $\mu_{i-1}(\mu_{i}(T))$ is not a tilting bundle in both situation.

If $\mu_{i-1}(\mu_{i}(T))$ is not a tilting bundle, we claim that $\mu_{i-1}(T)$ is a tilting bundle. For contradiction, we assume $\mu_{i-1}(T)$ is not a tilting bundle, by Proposition \ref{position}, the position of the arcs in $\mathbb{U}$ corresponding to $T_{i-2}, T_{i-1}, T_{i}, T_{i+1}$ exactly satisfies one of the following two conditions:

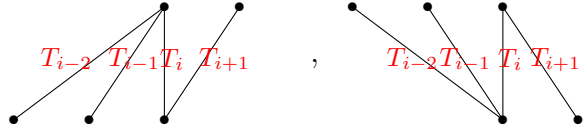
\begin{figure}[H]
\begin{tikzpicture}
\draw[-](-2.5,0)--(-1.5,1.5);
\draw[-](-1.5,1.5)--(-1.5,0);
\draw[-](-1.5,0)--(-0.5,1.5);
\draw[-](-3.5,0)--(-1.5,1.5);
\node()at(-2.5,0){\tiny$\bullet$};
\node()at(-1.5,1.5){\tiny$\bullet$};
\node()at(-1.5,0){\tiny$\bullet$};
\node()at(-0.5,1.5){\tiny$\bullet$};
\node()at(-3.5,0){\tiny$\bullet$};
\node()at(-1.4,0.8){\footnotesize${\red T_i}$};
\node()at(-1.9,0.8){\footnotesize${\red T_{i-1}}$};
\node()at(-0.7,0.8){\footnotesize${\red T_{i+1}}$};
\node()at(-2.8,0.8){\footnotesize${\red T_{i-2}}$};
\node()at(0.5,0.75){$,$};

\draw[-](2,1.5)--(3,0);
\draw[-](3,0)--(3,1.5);
\draw[-](4,0)--(3,1.5);
\draw[-](3,0)--(1,1.5);
\node()at(3,0){\tiny$\bullet$};
\node()at(3,1.5){\tiny$\bullet$};
\node()at(2,1.5){\tiny$\bullet$};
\node()at(4,0){\tiny$\bullet$};
\node()at(1,1.5){\tiny$\bullet$};
\node()at(3.1,0.8){\footnotesize${\red T_i}$};
\node()at(2.5,0.8){\footnotesize${\red T_{i-1}}$};
\node()at(3.7,0.8){\footnotesize${\red T_{i+1}}$};
\node()at(1.8,0.8){\footnotesize${\red T_{i-2}}$};
\end{tikzpicture}
\end{figure}
\noindent It follows that $\mu_{i-1}(\mu_{i}(T))$ is a tilting bundle in both cases, a contradiction. We are done.
\end{proof}

\subsection{Bundle-mutation index $n(T)$}

For any tilting bundle $T=\bigoplus\limits_{i=1}^{p+q}T_{i}$ in ${\rm coh}\mbox{-}\mathbb{X}(p,q)$, let $$I(T)=\{i|\mu_{i}(T) \;{\rm is\;a\; tilting
\; bundle}\},$$
and define the \emph{bundle-mutation index} of $T$ as $n(T)=|I(T)|.$

Under the following bijection from (\ref{the map phi}),
$$\varphi: \mathcal{T}^{\nu}_{\mathbb{X}}\longrightarrow \Lambda_{(p,q)}^0; \quad T\mapsto (c_1, c_2, \cdots, c_p),$$
we define $$J(T)=\{i|c_i=c_{i+1}\;\, {\rm for\;} 1\leq i\leq p-1\},$$ and denote by
$r(T)=|J(T)|.$

\begin{prop}\label{the bound of index}
For any tilting bundle $T$ in ${\rm coh}\mbox{-}\mathbb{X}(p,q)$, we have $n(T)=2(p-r(T))$. In particular, $2 \leq n(T)\leq 2p$.
\end{prop}

\begin{proof}
For any tilting bundle $T$, there exists a unique $(c_1, c_2, \cdots, c_p)\in\Lambda_{(p,q)}^0$ associated to $T$ by Theorem \ref{bijection between tilting bundle and combination vertices}, moreover, $T$ corresponds to the following triangulation of (a parallelogram in) $\mathbb{U}$:
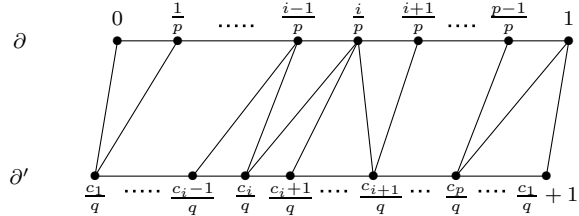
\begin{figure}[H]
\begin{tikzpicture}
\draw[-](0,0)--(0.3,1.8);
\draw[-](0,0)--(6,0);
\draw[-](6,0)--(6.3,1.8);
\draw[-](0.3,1.8)--(6.3,1.8);
\node()at(0,0){\tiny$\bullet$};
\node()at(6,0){\tiny$\bullet$};
\node()at(0.3,1.8){\tiny$\bullet$};
\node()at(6.3,1.8){\tiny$\bullet$};
\node()at(1.1,1.8){\tiny$\bullet$};
\draw[-](0,0)--(1.1,1.8);
\node()at(2.7,1.8){\tiny$\bullet$};
\node()at(3.5,1.8){\tiny$\bullet$};
\node()at(4.3,1.8){\tiny$\bullet$};
\node()at(2,0){\tiny$\bullet$};
\node()at(1.3,0){\tiny$\bullet$};
\node()at(2.6,0){\tiny$\bullet$};
\draw[-](2.6,0)--(3.5,1.8);
\node()at(1.3,-0.3){\tiny$\frac{c_i-1}{q}$};
\node()at(2.6,-0.3){\tiny$\frac{c_i+1}{q}$};
\node()at(3.7,0){\tiny$\bullet$};
\draw[-](2.7,1.8)--(2,0);
\draw[-](3.5,1.8)--(2,0);
\draw[-](2.7,1.8)--(1.3,0);
\draw[-](3.5,1.8)--(3.7,0);
\draw[-](4.3,1.8)--(3.7,0);
\node()at(5.5,1.8){\tiny$\bullet$};
\node()at(4.8,0){\tiny$\bullet$};
\draw[-](5.5,1.8)--(4.8,0);
\draw[-](6.3,1.8)--(4.8,0);
\node()at(0,-0.3){\tiny$\frac{c_1}{q}$};

\node()at(2,-0.3){\tiny$\frac{c_i}{q}$};
\node()at(3.8,-0.3){\tiny$\frac{c_{i+1}}{q}$};
\node()at(4.8,-0.3){\tiny$\frac{c_p}{q}$};
\node()at(0.3,2.1){\tiny$0$};
\node()at(1.1,2.1){\tiny$\frac{1}{p}$};
\node()at(2.7,2.1){\tiny$\frac{i-1}{p}$};
\node()at(3.5,2.1){\tiny$\frac{i}{p}$};
\node()at(4.3,2.1){\tiny$\frac{i+1}{p}$};
\node()at(5.5,2.1){\tiny$\frac{p-1}{p}$};
\node()at(6.3,2.1){\tiny$1$};
\node()at(6,-0.3){\tiny$\frac{c_1}{q}+1$};
\draw[line width=1pt,dotted](0.4,-0.2)--(0.9,-0.2);
\draw[line width=1pt,dotted](4.2,-0.2)--(4.5,-0.2);
\draw[line width=1pt,dotted](5.1,-0.2)--(5.5,-0.2);
\draw[line width=1pt,dotted](1.65,2)--(2.15,2);
\draw[line width=1pt,dotted](4.7,2)--(5.1,2);
\draw[line width=1pt,dotted](3,-0.2)--(3.4,-0.2);
\node()at(-1,0){\tiny$\partial^{\prime}$};
\node()at(-1,1.8){\tiny$\partial$};
\end{tikzpicture}
\end{figure}
We denote by $m_{T}(T_i)=1$ if $\mu_{i}(T)$ is a tilting bundle and $m_{T}(T_i)=0$ otherwise.
Combining the above figure and Proposition \ref{position}, we get

$$m_{T}([D^{b}_{a}])=
  \left\{
  \begin{array}{lll}
    1,&& {\rm if\; } a=\frac{c_i}{q},\, b=\frac{i-1}{p} \;{\rm or}\;b=\frac{i}{p} {\rm \; for\; some\;} i;\\
   0,&& {\rm otherwise} .
  \end{array}
\right.$$
It follows that $n(T)=2(p-r(T))$ and then $2\leq n(T)\leq 2p.$
\end{proof}

By Proposition \ref{the bound of index}, we have the following corollary.

\begin{cor}
Keep notations as above. Let $T$ be a tilting bundle in ${\rm coh}\mbox{-}\mathbb{X}(p,q)$. Then
\begin{itemize}
  \item[(1)] $n(T)=2$ if and only if $c_1=c_2=\cdots=c_p.$
\item[(2)] $n(T)=2p$ if and only if $c_1\textless c_2\textless\cdots\textless c_p.$
  \end{itemize}
\end{cor}

\begin{rem}
Let $T$ be a tilting bundle, which corresponds to a triangulation $\Gamma$ of $A_{p,q}$ by Theorem \ref{triangulation and tilting}.
Let $Q_{\Gamma}$ be the quiver associated to the triangulation $\Gamma$. Then $n(T)=2$ if and only if $\mathbf{k}Q_{\Gamma}$ is a canonical algebra of type $(p,q)$.
\end{rem}

For any $a\in \frac{\mathbb{Z}}{p},\, b\in \frac{\mathbb{Z}}{q}$, denote by $$T^{a}_{b}=[D^{a}_{b}]\oplus [D^{a+\frac{1}{p}}_{b}]\oplus [D^{a+\frac{1}{p}}_{b+\frac{1}{q}}]\oplus [D^{a+\frac{2}{p}}_{b+\frac{1}{q}}]\oplus \cdots\oplus [D^{a+1}_{b+\frac{p-1}{q}}]\oplus [D^{a+1}_{b+\frac{p}{q}}]\oplus \cdots\oplus [D^{a+1}_{b+\frac{q-1}{q}}].$$
Intuitively, $T^{a}_{b}$ corresponds to the following triangulation of a parallelogram in $\mathbb{U}$ up to shift:
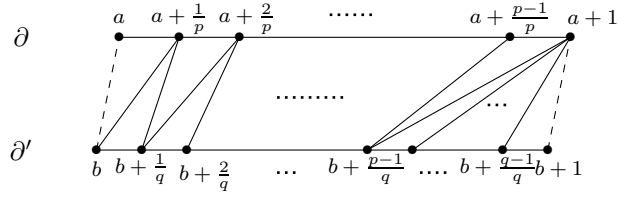
\begin{figure}[H]
\begin{tikzpicture}
\draw[-](0,0)--(6,0);
\draw[-](0.3,1.5)--(6.3,1.5);
\draw[-,dashed](0,0)--(0.3,1.5);
\draw[-,dashed](6,0)--(6.3,1.5);
\node()at(0,0){\tiny$\bullet$};
\node()at(6,0){\tiny$\bullet$};
\node()at(6.15,-0.25){\tiny$b+1$};
\node()at(0.3,1.5){\tiny$\bullet$};
\node()at(0.3,1.75){\tiny$a$};
\node()at(0,-0.25){\tiny$b$};
\node()at(6.3,1.5){\tiny$\bullet$};
\node()at(0.6,0){\tiny$\bullet$};
\node()at(1.2,0){\tiny$\bullet$};
\node()at(1.1,1.5){\tiny$\bullet$};
\node()at(1.9,1.5){\tiny$\bullet$};
\draw[-](5.4,0)--(6.3,1.5);
\node()at(5.4,0){\tiny$\bullet$};
\node()at(5.35,-0.25){\tiny$b+\frac{q-1}{q}$};
\draw[-](3.6,0)--(6.3,1.5);
\node()at(3.6,0){\tiny$\bullet$};
\node()at(3.6,-0.25){\tiny$b+\frac{p-1}{q}$};
\draw[-](4.2,0)--(6.3,1.5);
\draw[-](0,0)--(1.1,1.5);
\draw[-](0.6,0)--(1.1,1.5);
\draw[-](0.6,0)--(1.9,1.5);
\draw[-](1.2,0)--(1.9,1.5);
\draw[-](5.5,1.5)--(3.6,0);
\node()at(1.1,1.75){\tiny$a+\frac{1}{p}$};
\node()at(2,1.75){\tiny$a+\frac{2}{p}$};
\node()at(0.6,-0.25){\tiny$b+\frac{1}{q}$};
\node()at(1.45,-0.35){\tiny$b+\frac{2}{q}$};
\node()at(5.5,1.75){\tiny$a+\frac{p-1}{p}$};
\node()at(6.6,1.75){\tiny$a+1$};
\node()at(5.5,1.5){\tiny$\bullet$};
\node()at(4.2,0){\tiny$\bullet$};
\draw[line width=1pt,dotted](2.4,-0.3)--(2.7,-0.3);
\draw[line width=1pt,dotted](4.3,-0.3)--(4.7,-0.3);
\draw[line width=1pt,dotted](3.1,1.8)--(3.7,1.8);
\draw[line width=1pt,dotted](5.2,0.6)--(5.5,0.6);
\draw[line width=1pt,dotted](2.4,0.7)--(3.3,0.7);
\node()at(-1,1.5){\footnotesize{${\partial}$}};
\node()at(-1,0){\footnotesize{${\partial^{\prime}}$}};
\end{tikzpicture}
\end{figure}
\noindent

We have the following observation.

\begin{prop}\label{translation}
For any $a\in \frac{\mathbb{Z}}{p},\, b\in \frac{\mathbb{Z}}{q}$, $T^{a}_{b}$ is a tilting bundle with $n(T^{a}_{b})=2p$.
Moreover, $$\mu_{2p-1}\cdot\mu_{2p-3}\cdots\cdot\mu_{1}(T^{a}_{b})=T^{a}_{b-\frac{1}{q}}\;\;{\rm and}\;\;\mu_{2p}\cdot\mu_{2p-2}\cdots\cdot\mu_{2}(T^{a}_{b})=T^{a}_{b+\frac{1}{q}}.$$
\end{prop}

\begin{proof}
The first result follows from Theorem \ref{triangulation and tilting} and Proposition \ref{the bound of index} immediately. Note that the triangulation $\mu_{2p-1}\cdot\mu_{2p-3}\cdots\cdot\mu_{1}(T^{a}_{b})$ can be obtained from that of $T^{a}_{b}$ (black arcs and green arcs) by iterated mutations, and each mutation $\mu_{2i-1}$ for $1\leq i\leq p$ replaces one green arc by a red arc as below:

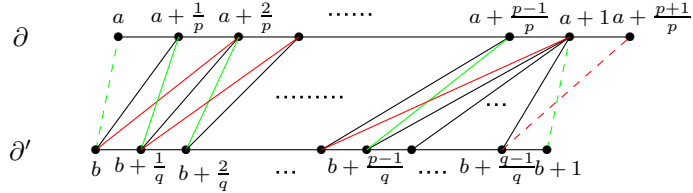
\begin{figure}[H]
\begin{tikzpicture}
\draw[-](0,0)--(6,0);
\draw[-](0.3,1.5)--(7.1,1.5);
\node()at(7.1,1.5){\tiny$\bullet$};
\draw[-,dashed,green](0,0)--(0.3,1.5);
\draw[-,dashed,green](6,0)--(6.3,1.5);
\node()at(0,0){\tiny$\bullet$};
\node()at(6,0){\tiny$\bullet$};
\node()at(6.15,-0.25){\tiny$b+1$};
\node()at(0.3,1.5){\tiny$\bullet$};
\node()at(0.3,1.75){\tiny$a$};
\node()at(0,-0.25){\tiny$b$};
\node()at(6.3,1.5){\tiny$\bullet$};
\node()at(0.6,0){\tiny$\bullet$};
\node()at(1.2,0){\tiny$\bullet$};
\node()at(1.1,1.5){\tiny$\bullet$};
\node()at(2.7,1.5){\tiny$\bullet$};
\node()at(1.9,1.5){\tiny$\bullet$};
\draw[-](5.4,0)--(6.3,1.5);
\node()at(5.4,0){\tiny$\bullet$};
\node()at(5.35,-0.25){\tiny$b+\frac{q-1}{q}$};
\draw[-](3.6,0)--(6.3,1.5);
\node()at(3.6,0){\tiny$\bullet$};
\node()at(3,0){\tiny$\bullet$};
\draw[-](3,0)--(5.5,1.5);
\node()at(3.6,-0.25){\tiny$b+\frac{p-1}{q}$};
\draw[-,red](0,0)--(1.9,1.5);
\draw[-,red,dashed](5.4,0)--(7.1,1.5);
\draw[-](4.2,0)--(6.3,1.5);
\draw[-](0,0)--(1.1,1.5);
\draw[-,green](0.6,0)--(1.1,1.5);
\draw[-](0.6,0)--(1.9,1.5);
\draw[-,green](1.2,0)--(1.9,1.5);
\draw[-](1.2,0)--(2.7,1.5);
\draw[-,red](0.6,0)--(2.7,1.5);
\draw[-,green](5.5,1.5)--(3.6,0);
\draw[-,red](6.3,1.5)--(3,0);
\node()at(1.1,1.75){\tiny$a+\frac{1}{p}$};
\node()at(2,1.75){\tiny$a+\frac{2}{p}$};
\node()at(0.6,-0.25){\tiny$b+\frac{1}{q}$};
\node()at(1.45,-0.35){\tiny$b+\frac{2}{q}$};
\node()at(5.5,1.75){\tiny$a+\frac{p-1}{p}$};
\node()at(6.5,1.75){\tiny$a+1$};
\node()at(7.4,1.75){\tiny$a+\frac{p+1}{p}$};
\node()at(5.5,1.5){\tiny$\bullet$};
\node()at(4.2,0){\tiny$\bullet$};
\draw[line width=1pt,dotted](2.4,-0.3)--(2.7,-0.3);
\draw[line width=1pt,dotted](4.3,-0.3)--(4.7,-0.3);
\draw[line width=1pt,dotted](3.1,1.8)--(3.7,1.8);
\draw[line width=1pt,dotted](5.2,0.6)--(5.5,0.6);
\draw[line width=1pt,dotted](2.4,0.7)--(3.3,0.7);
\node()at(-1,1.5){\footnotesize{${\partial}$}};
\node()at(-1,0){\footnotesize{${\partial^{\prime}}$}};
\end{tikzpicture}
\end{figure}
\noindent More precisely,
\begin{align}\nonumber
&\mu_{2p-1}\cdot\mu_{2p-3}\cdots\cdot\mu_{1}(T^{a}_{b})\nonumber\\
&=[D^{a+\frac{p+1}{p}}_{b+\frac{q-1}{q}}]\oplus[D^{a+\frac{1}{p}}_{b}]\oplus[D^{a+\frac{2}{p}}_{b}]\oplus\cdots\oplus [D^{a+1}_{b+\frac{p-2}{q}}]\oplus\cdots\oplus[D^{a+1}_{b+\frac{q-1}{q}}]\nonumber\\
&=[D^{a}_{b-\frac{1}{q}}]\oplus[D^{a+\frac{1}{p}}_{b-\frac{1}{q}}]\oplus[D^{a+\frac{1}{p}}_{b}]\oplus[D^{a+\frac{2}{p}}_{b}]\oplus\cdots\oplus [D^{a+1}_{b+\frac{p-2}{q}}]\oplus\cdots\oplus[D^{a+1}_{b+\frac{q-2}{q}}]\nonumber\\
&=T^{a}_{b-\frac{1}{q}}.\nonumber
\end{align}
Similarly, one can obtain the other equation. We are done.
\end{proof}

The tilting bundle $T^{a}_{b}$ plays an important role in the next subsection for the connectedness of the tilting graph of vector bundles category.

\subsection{Connectedness of the tilting graph $\mathcal{G}(\mathcal{T}^{\nu}_{\mathbb{X}})$}

Recall from \cite{BKL2008, GengSF2020} that the \emph{tilting graph} $\mathcal{G}(\mathcal{T}_{\mathbb{X}})$ of ${\rm coh}\mbox{-}\mathbb{X}(p,q)$ has as vertices the isomorphism classes of tilting sheaves in ${\rm coh}\mbox{-}\mathbb{X}(p,q)$, while two vertices are connected by an edge if and only if the associated tilting sheaves differ by precisely one indecomposable direct summand. The \emph{tilting graph} $\mathcal{G}(\mathcal{T}^{\nu}_{\mathbb{X}})$ of ${\rm vec}\mbox{-}\mathbb{X}(p,q)$ is the full subgraph of $\mathcal{G}(\mathcal{T}_{\mathbb{X}})$ consisting of tilting bundles.

The connectedness of titling graph for weighted projective lines has been investigated widely in the literature through category aspect, see for example \cite{BKL2008, HU2005, GengSF2020, FG383}. In this subsection, we investigate the connectedness of $\mathcal{G}(\mathcal{T}^{\nu}_{\mathbb{X}})$ by using our geometric model.

Let $T$ and $T^{\prime}$ be tilting bundles in ${\rm coh}\mbox{-}\mathbb{X}(p,q)$. We say $T$ is \emph{bundle-mutations to} $T'$ if there is a sequence of tilting bundles $T=T^{0},\,T^1,\,\cdots, T^{n-1}$, $T^{n}=T^{\prime}$ such that $T^{i}$ is a mutation of $T^{i-1}$ for any $1\leq i\leq n$.
Under the bijection Theorem \ref{corresponding}, we call a triangulation $\Gamma$ is \emph{bundle-flips to} $\Gamma^{\prime}$ if the associated tilting bundles $T$ and $T'$ satisfy that $T$ is \emph{bundle-mutations to} $T'$.

By Convention \ref{order}, we know that the direct summands of any tilting bundle $T$ can be arranged in a unique order. Hence we can define a map $\iota$ from the set of tilting bundles in ${\rm coh}\mbox{-}\mathbb{X}(p,q)$ to $\mathbb{Z}_2^{p+q}$, in the sense that $\iota(T)_{i}=1$ if $\mu_{i}(T)$ is a tilting bundle, and $\iota(T)_{i}=0$ otherwise.
For example, in the case $p=2$ and $q=2$, if $\mu_{1}(T), \mu_2(T)$ are tilting bundles and $\mu_{3}(T), \mu_4(T)$ are not tilting bundles, we have $\iota(T)=(1,1,0,0)$.

\begin{lem}\label{4 to 3}
Let $T$ be a tilting bundle in ${\rm coh}\mbox{-}\mathbb{X}(p,q)$ and assume $$\iota(T)=(\underbrace{1, \cdots, 1}_{s_1\ \text{times}},\underbrace{0, \cdots, 0}_{t_1\ \text{times}},\underbrace{1, \cdots, 1}_{s_2\ \text{times}},\underbrace{0, \cdots, 0}_{t_2\ \text{times}}).$$ Then $T$ is bundle-mutations to some tilting bundle $T^{\prime}$ satisfying $$\iota(T^{\prime})=(\underbrace{1, \cdots, 1}_{m_1\ \text{times}},\underbrace{0, \cdots, 0}_{n_1\ \text{times}},\underbrace{1, \cdots, 1}_{m_2\ \text{times}})$$ and $m_1+m_2\geq s_1+s_2.$
\end{lem}

\begin{proof}
Note that $p+q=s_1+t_1+s_2+t_2$. We divide the proof into two cases by considering $s_2$ is even or not.

If $s_2$ is even, then $s_2=2s$ for some $s\in \mathbb{Z}_{\geq 1}$. Denote by $$\mu_{[a,b]}=\mu_a\cdot\mu_{a-1}\cdots\cdot\mu_{b}, \;\;b\leq a.$$

Let
$$\nu_{k}=\mu_{[p+q+1-2k, \,\,p+q-t_2+2-2k ]},\;\;1\leq k\leq s.$$
By iterative use of Proposition \ref{effect of tilting bundle-mutation}, we obtain $$\iota(\nu_{s}\cdot\nu_{s-1}\cdots\nu_1(T))=(\underbrace{1, \cdots, 1}_{s_1\ \text{times}},\underbrace{0, \cdots, 0}_{t_1+t_2\ \text{times}},\underbrace{1, \cdots, 1}_{s_2\ \text{times}}),$$
and $\nu_{s}\cdot\nu_{s-1}\cdots\nu_1(T)$ is a tilting bundle.

If $s_2$ is odd, then $s_2=2s+1$ for some $s\in \mathbb{Z}_{\geq 0}$.
We have
$$\iota(\nu_{s}\cdot\nu_{s-1}\cdots\nu_1(T))=(\underbrace{1, \cdots, 1}_{s_1\ \text{times}},\underbrace{0, \cdots, 0}_{t_1\ \text{times}},1,\underbrace{0, \cdots, 0}_{t_2\ \text{times}},\underbrace{1, \cdots, 1}_{2s\ \text{times}}).$$
If $t_2\geq t_1,$ let
$$\delta_{k}=\mu_{[p+q-2s-2k,\,\,p+q-t_2-2s+1-k]},\;\;1\leq k\leq t_1.$$

\noindent Then $$\iota(\delta_{t_1}\cdot\delta_{t_1-1}\cdots\delta_{1}\cdot\nu_{s}\cdot\nu_{s-1}\cdots\nu_{1}(T))=(\underbrace{1, \cdots, 1}_{s_1+1\ \text{times}},\underbrace{0, \cdots, 0}_{t_2-t_1\ \text{times}},\underbrace{1, \cdots, 1}_{2(s+t_1) \text{times}}),$$
\noindent and $s_{1}+1+2(s+t_1)=s_1+s_2+t_1\geq s_{1}+s_{2}$.

If $t_2 < t_1$, let
$$\delta_{k}=\mu_{s_1+2k}\cdot\mu_{s_1+2k+1}\cdots\cdot\mu_{s_1+t_1+k},\;\;1\leq k\leq t_2.$$
\noindent Then
$$\iota(\delta_{t_2}\cdot\delta_{t_2-1}\cdots\delta_{1}\cdot\nu_{s}\cdot\nu_{s-1}\cdots\nu_{1}(T))=(\underbrace{1, \cdots, 1}_{s_1+2t_2\ \text{times}},\underbrace{0, \cdots, 0}_{t_1-t_2\ \text{times}},\underbrace{1, \cdots, 1}_{s_2\ \text{times}}),$$
\noindent and $s_{1}+2t_{2}+s_2\geq s_{1}+s_{2}$.
We are done.
\end{proof}

Now we can state the main result of this subsection.
\begin{prop}\label{any tilting bundle transformed into unique maximal}
Any tilting bundle $T$ in ${\rm coh}\mbox{-}\mathbb{X}(p,q)$ is bundle-mutations to $T^{a}_{b}$ for some $a\in \frac{\mathbb{Z}}{p},\, b\in \frac{\mathbb{Z}}{q}$. Consequently, the tilting graph $\mathcal{G}(\mathcal{T}_{\mathbb{X}}^{\nu})$ is connected.
\end{prop}

\begin{proof}
According to equations (\ref{periodicity of covering map}) and Proposition \ref{the bound of index}, we can always assume $\mu_1(T)$ is a tilting bundle. Hence $\iota(T)$ must be one of the following two forms by considering $\mu_{p+q}(T)$ is a tilting bundle or not:
$$\iota(T)=(\underbrace{1, \cdots, 1}_{s_1\ \text{times}},\underbrace{0, \cdots, 0}_{t_1\ \text{times}},\underbrace{1, \cdots, 1}_{s_2\ \text{times}},\underbrace{0, \cdots, 0}_{t_2\ \text{times}}, \underbrace{1, \cdots, 1}_{s_3\ \text{times}},\cdots\cdots\underbrace{1, \cdots, 1}_{s_m\ \text{times}}),$$
\noindent or
$$\iota(T)=(\underbrace{1, \cdots, 1}_{s_1\ \text{times}},\underbrace{0, \cdots, 0}_{t_1\ \text{times}},\underbrace{1, \cdots, 1}_{s_2\ \text{times}},\underbrace{0, \cdots, 0}_{t_2\ \text{times}},\underbrace{1, \cdots, 1}_{s_3\ \text{times}},\cdots\cdots\underbrace{1, \cdots, 1}_{s_m\ \text{times}},\underbrace{0, \cdots, 0}_{t_m\ \text{times}}).$$

Using Lemma \ref{4 to 3}, we can reduce $\iota(T)$ to the following shapes:
$$(\underbrace{1, \cdots, 1}_{m_1^{\prime}\ \text{times}},\underbrace{0, \cdots, 0}_{n_1^{\prime}\ \text{times}},\underbrace{1, \cdots, 1}_{m_2^{\prime}+s_3\ \text{times}},\underbrace{0, \cdots, 0}_{t_3\ \text{times}}\cdots\cdots\underbrace{1, \cdots, 1}_{s_m\ \text{times}})$$
\noindent or
$$(\underbrace{1, \cdots, 1}_{m_1^{\prime}\ \text{times}},\underbrace{0, \cdots, 0}_{n_1^{\prime}\ \text{times}},\underbrace{1, \cdots, 1}_{m_2^{\prime}+s_3\ \text{times}},\underbrace{0, \cdots, 0}_{t_3\ \text{times}}\cdots\cdots\underbrace{1, \cdots, 1}_{s_m\ \text{times}},\underbrace{0, \cdots, 0}_{t_m\ \text{times}})$$
with $m_1^{\prime}+m_2^{\prime}\geq s_1+s_2$ respectively.

By iterating the above process, one finally obtains that $T$ is bundle-mutations to a tilting bundle $T^{\prime}$ such that $$\iota(T^{\prime})=(\underbrace{1, \cdots, 1}_{m_1\ \text{times}},\underbrace{0, \cdots, 0}_{n_1\ \text{times}},\underbrace{1, \cdots, 1}_{m_2\ \text{times}})$$ and $m_1+m_2\geq s_1+s_2+\cdots+s_m$.

Assume $T'=\bigoplus\limits_{i=1}^{p+q}[D^{a_{i}}_{b_{i}}]$ with $0= a_1< a_2\leq \cdots\leq a_{p+q}= 1,\,b_1\leq b_2\leq \cdots\leq b_{p+q}$.
According to (\ref{periodicity of covering map}), we have$$T'=(\bigoplus\limits_{i=1}^{m_{1}+n_{1}}[D^{a_{i}}_{b_{i}}])
\oplus(\bigoplus\limits_{i=m_{1}+n_{1}+1}^{p+q}[D^{a_{i}}_{b_{i}}])
=(\bigoplus\limits_{i=m_{1}+n_{1}+1}^{p+q}[D^{a_{i}-1}_{b_{i}-1}])
\oplus(\bigoplus\limits_{i=1}^{m_{1}+n_{1}}[D^{a_{i}}_{b_{i}}]),$$
which implies that $$\iota(T^{\prime})=(\underbrace{1, \cdots, 1,}_{m_1+m_2\ \text{times}}\underbrace{0, \cdots, 0}_{n_1\ \text{times}}).$$ By the proof of Proposition \ref{the bound of index}, we get $T^{\prime}=T^{a}_{b}$ for some $a\in \frac{\mathbb{Z}}{p},\, b\in \frac{\mathbb{Z}}{q}$.
Combining with Proposition \ref{translation}, we obtain the connectedness of $\mathcal{G}(\mathcal{T}_{\mathbb{X}}^{\nu})$. This finishes the proof.
\end{proof}

\section{Tilting graph}
In this section, we will give more explicit structure for the tilting graph $\mathcal{G}(\mathcal{T}_{\mathbb{X}})$ of ${\rm coh}\mbox{-}\mathbb{X}(p,q)$ and the tilting graph $\mathcal{G}(\mathcal{T}_{\mathbb{X}}^{\nu})$ of ${\rm vec}\mbox{-}\mathbb{X}(p,q)$ respectively by using the geometric model.

\subsection{Local shape of the tilting graphs}

In this subsection, we consider the local shape of the tilting graphs $\mathcal{G}(\mathcal{T}_{\mathbb{X}})$ and $\mathcal{G}(\mathcal{T}_{\mathbb{X}}^{\nu})$ respectively.

For the case $p=q=1$, i.e., for the classical projective line $\mathbb{P}_{\mathbf k}^1$ case, it is well known that any tilting sheaf has the form $T_{n}:=\co(n)\oplus \co(n+1)$ for some $n\in\mathbb{Z}$. Moreover, there exists a tilting mutation between $T_{n}$ and $T_{m}$ if and only if $m=n\pm1$. Therefore, the tilting subgraph $\mathcal{G}(\mathcal{T}_{\mathbb{X}}^{\nu})$ coincides with the tilting graph $\mathcal{G}(\mathcal{T}_{\mathbb{X}})$, which has the form:

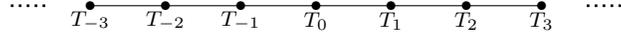
\begin{figure}[H]
\begin{tikzpicture}
\node()at(0,0){\tiny$\bullet$};
\node()at(1,0){\tiny$\bullet$};
\node()at(2,0){\tiny$\bullet$};
\node()at(3,0){\tiny$\bullet$};
\node()at(-1,0){\tiny$\bullet$};
\node()at(-2,0){\tiny$\bullet$};
\node()at(-3,0){\tiny$\bullet$};
\draw[-](0,0)--(1,0);
\draw[-](1,0)--(2,0);
\draw[-](2,0)--(3,0);
\draw[-](0,0)--(-1,0);
\draw[-](-1,0)--(-2,0);
\draw[-](-2,0)--(-3,0);
\node()at(0,-0.2){\tiny$T_0$};
\node()at(1,-0.2){\tiny$T_1$};
\node()at(2,-0.2){\tiny$T_2$};
\node()at(3,-0.2){\tiny$T_3$};
\node()at(-1,-0.2){\tiny$T_{-1}$};
\node()at(-2,-0.2){\tiny$T_{-2}$};
\node()at(-3,-0.2){\tiny$T_{-3}$};
\draw[line width=1pt,dotted](3.6,0)--(4.1,0);
\draw[line width=1pt,dotted](-3.6,0)--(-4.1,0);
\end{tikzpicture}
\caption{Tilting graph of ${\rm coh}\mbox{-}\mathbb{P}_{\mathbf k}^1$}
\end{figure}

From now on we focus on the weighted projective line $\mathbb{X}(p,q)$ with $(p,q)\neq (1,1)$.

\begin{prop}\label{pentagon}
Assume $(p,q)\neq (1,1)$. Let $T=\bigoplus\limits_{i=1}^{p+q}T_{i}$ be a tilting sheaf in ${\rm coh}\mbox{-}\mathbb{X}(p,q)$ and $\Gamma$ be the associated triangulation of $A_{p,q}$. Let $\gamma_i$ be the arc in $\Gamma$ corresponding to $T_{i}$ for $1\leq i\leq p+q$. For any $1\leq i\neq j\leq p+q,$ one of the following holds:

\begin{itemize}
  \item[(1)] if two arcs $\gamma_i$ and $\gamma_j$ are in the same triangle, then $\mu_{i}\mu_{j}\mu_{i}(T)=\mu_{i}\mu_{j}(T);$
\item[(2)] if two arcs $\gamma_i$ and $\gamma_j$ are not in the same triangle, then $\mu_{i}\mu_{j}(T)=\mu_{j}\mu_{i}(T).$
  \end{itemize}
\end{prop}
\begin{proof} For any arc $\gamma_i$ in $\Gamma,$ there exists a quadrilateral formed by the two triangles of $\Gamma$ containing $\gamma_i$, drawn as follows:
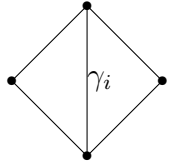
\begin{figure}[H]
\begin{tikzpicture}
\draw[](0,0)--(1,1);
\draw[](0,0)--(1,-1);
\draw[](1,1)--(2,0);
\draw[](1,-1)--(2,0);
\draw[](1,1)--(1,-1);
\node()at(1,1){\tiny$\bullet$};
\node()at(1,-1){\tiny$\bullet$};
\node()at(2,0){\tiny$\bullet$};
\node()at(0,0){\tiny$\bullet$};
\node()at(1.17,0){$\gamma_i$};
\end{tikzpicture}
\end{figure}

(1) If two arcs $\gamma_i$ and $\gamma_j$ are in the same triangle, without loss of generality, we may assume that the position of $\gamma_i$ and $\gamma_j$ is drawn as below:

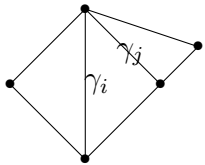
\begin{figure}[H]
\begin{tikzpicture}
\draw[](0,0)--(1,1);
\draw[](0,0)--(1,-1);
\draw[](1,1)--(2,0);
\draw[](1,-1)--(2,0);
\draw[](1,1)--(1,-1);
\node()at(1,1){\tiny$\bullet$};
\node()at(1,-1){\tiny$\bullet$};
\node()at(2,0){\tiny$\bullet$};
\node()at(0,0){\tiny$\bullet$};
\node()at(1.15,0){$\gamma_i$};
\node()at(1.6,0.4){$\gamma_j$};
\draw[](1,1)--(2.5,0.5);
\draw[](2,0)--(2.5,0.5);
\node()at(2.5,0.5){\tiny$\bullet$};
\end{tikzpicture}
\end{figure}
\noindent By applying the flips $\mu_{i}$ and $\mu_{j}$ on the corresponding arcs, and noting that $(p,q)\neq (1,1)$, we have the following commutative diagram:

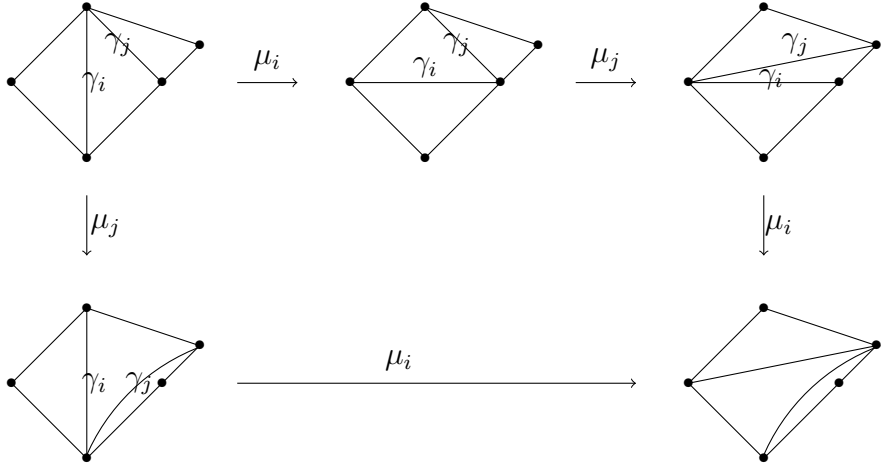
\begin{figure}[H]
\begin{tikzpicture}
\draw[](0,0)--(1,1);
\draw[](0,0)--(1,-1);
\draw[](1,1)--(2,0);
\draw[](1,-1)--(2,0);
\draw[](1,1)--(1,-1);
\node()at(1,1){\tiny$\bullet$};
\node()at(1,-1){\tiny$\bullet$};
\node()at(2,0){\tiny$\bullet$};
\node()at(0,0){\tiny$\bullet$};
\node()at(1.1,0){$\gamma_i$};
\node()at(1.42,0.5){$\gamma_j$};
\draw[](1,1)--(2.5,0.5);
\draw[](2,0)--(2.5,0.5);
\node()at(2.5,0.5){\tiny$\bullet$};
\draw[->](3,0)--(3.8,0);
\draw[->](1,-1.5)--(1,-2.3);
\node()at(1.25,-1.9){${\mu_j}$};
\node()at(3.4,0.3){${\mu_i}$};
\draw[](4.5,0)--(5.5,1);
\draw[](4.5,0)--(5.5,-1);
\draw[](5.5,1)--(6.5,0);
\draw[](5.5,-1)--(6.5,0);
\draw[](4.5,0)--(6.5,0);
\node()at(5.5,1){\tiny$\bullet$};
\node()at(5.5,-1){\tiny$\bullet$};
\node()at(6.5,0){\tiny$\bullet$};
\node()at(4.5,0){\tiny$\bullet$};
\node()at(5.5,0.2){$\gamma_i$};
\node()at(5.92,0.5){$\gamma_j$};
\draw[](5.5,1)--(7,0.5);
\draw[](6.5,0)--(7,0.5);
\node()at(7,0.5){\tiny$\bullet$};
\draw[->](7.5,0)--(8.3,0);
\node()at(7.9,0.3){${\mu_j}$};
\draw[](9,0)--(10,1);
\draw[](9,0)--(10,-1);
\draw[](9,0)--(11.5,0.5);
\draw[](10,-1)--(11,0);
\draw[](9,0)--(11,0);
\node()at(10,1){\tiny$\bullet$};
\node()at(10,-1){\tiny$\bullet$};
\node()at(11,0){\tiny$\bullet$};
\node()at(9,0){\tiny$\bullet$};
\node()at(10.1,0.05){$\gamma_i$};
\node()at(10.42,0.5){$\gamma_j$};
\draw[](10,1)--(11.5,0.5);
\draw[](11,0)--(11.5,0.5);
\node()at(11.5,0.5){\tiny$\bullet$};
\draw[->](10,-1.5)--(10,-2.3);
\node()at(10.2,-1.9){$\mu_{i}$};
\draw[](0,-4)--(1,-3);
\draw[](0,-4)--(1,-5);
\draw[](1,-5)--(2,-4);
\draw[](1,-3)--(1,-5);
\node()at(1,-3){\tiny$\bullet$};
\node()at(1,-5){\tiny$\bullet$};
\node()at(2,-4){\tiny$\bullet$};
\node()at(0,-4){\tiny$\bullet$};
\node()at(1.1,-4){$\gamma_i$};
\draw[](1,-3)--(2.5,-3.5);
\draw[](2,-4)--(2.5,-3.5);
\node()at(2.5,-3.5){\tiny$\bullet$};
\draw[->](3,-4)--(8.3,-4);
\node()at(5.15,-3.7){${\mu_i}$};
\draw[](9,-4)--(10,-3);
\draw[](9,-4)--(11.5,-3.5);
\draw[](9,-4)--(10,-5);
\draw[](10,-5)--(11,-4);
\node()at(10,-3){\tiny$\bullet$};
\node()at(10,-5){\tiny$\bullet$};
\node()at(11,-4){\tiny$\bullet$};
\node()at(9,-4){\tiny$\bullet$};
\draw[](10,-3)--(11.5,-3.5);
\draw[](11,-4)--(11.5,-3.5);
\node()at(11.5,-3.5){\tiny$\bullet$};
\node()at(1.7,-4){$\gamma_j$};
\draw[](2.55,-3.5)arc(110:158:2.7);
\draw[](11.55,-3.5)arc(110:158:2.7);
\end{tikzpicture}
\end{figure}
\noindent Therefore $\mu_{i}\mu_{j}\mu_{i}(T)=\mu_{i}\mu_{j}(T).$

\vspace{1mm}

(2) If two arcs $\gamma_i$ and $\gamma_j$ are not in the same triangle, then the mutations $\mu_{T}(T_i)$ and $\mu_{T}(T_j)$ do not affect each other. Therefore, we have $\mu_{i}\mu_{j}(T)=\mu_{j}\mu_{i}(T).$
\end{proof}

As an immediate consequence of the above proposition, we obtain the following result, c.f. \cite[Theorem 3.10]{FST2008}).

\begin{cor}
Assume $(p,q)\neq (1,1)$, then the tilting graph $\mathcal{G}(\mathcal{T}_{\mathbb{X}})$ is composed of quadrilaterals and pentagons.
\end{cor}

By considering the tilting subgraph $\mathcal{G}(\mathcal{T}^{\nu}_{\mathbb{X}})$ of tilting bundles, we have the following result.
\begin{prop}
\begin{itemize}
  \item[(1)] The tilting graph $\mathcal{G}(\mathcal{T}^{\nu}_{\mathbb{X}})$ of ${\rm vec}\mbox{-}\mathbb{X}(1,q)$ is a line.
\item[(2)] For $2\leq p\leq q$, the tilting graph $\mathcal{G}(\mathcal{T}^{\nu}_{\mathbb{X}})$ of ${\rm vec}\mbox{-}\mathbb{X}(p,q)$ is composed of quadrilaterals.
  \end{itemize}
\end{prop}

\begin{proof}
(1) If $\bbX$ has weight type $(1,q)$, then by Theorem \ref{triangulation and tilting} $T=T^0_b$ for some $b\in \frac{\mathbb{Z}}{q}$, and the bundle-mutation index $n(T)=2$ by Proposition \ref{the bound of index}. Combining with Proposition \ref{translation}, we obtain that the tilting graph $\mathcal{G}(\mathcal{T}^{\nu}_{\mathbb{X}})$ of ${\rm vec}\mbox{-}\mathbb{X}(1,q)$ is a line.

(2) Assume $\bbX$ has weight type $(p,q)$ with $2\leq p\leq q$. Let $T=\bigoplus\limits_{i=1}^{p+q}T_{i}$ be a tilting bundle. Assume $\mu_{i}(T)$ and $\mu_{j}(T)$ are both tilting bundles for some $i,j$, let $\gamma_i$ and $\gamma_j$ be the corresponding arcs of $T_i$ and $T_j$ respectively. If $\mu_{j}\mu_{i}(T)$ is also a tilting bundle, then $\gamma_i$ and $\gamma_j$ are not in the same triangle by Proposition \ref{effect of tilting bundle-mutation}. It follows that $\mu_{i}\mu_{j}(T)=\mu_{j}\mu_{i}(T)$ by Proposition \ref{pentagon} $(2)$. Hence, the tilting graph $\mathcal{G}(\mathcal{T}^{\nu}_{\mathbb{X}})$ of vector bundle category is composed of quadrilaterals.
\end{proof}
\subsection{Proof of Theorem \ref{description of tilting graph of vector bundles}}
Recall that $\Lambda_{(p,q)}$ is the graph with vertex set
$$\Lambda_{(p,q)}^0=\{(c_1, \cdots, c_p)\in\mathbb{Z}^{p}|c_1\leq \cdots\leq c_p\leq c_1+q\},$$ and there exists an edge between two vertices $(c_1, \cdots, c_p)$ and $(d_1, \cdots, d_p)$ if and only if $\sum_{i=1}^{p}|c_i-d_i|=1$, or equivalently, there exists a unique $j$ such that
$$d_i=
  \left\{
  \begin{array}{lll}
    c_i\pm 1&& i=j;\\
   c_i&&  i\neq j.
  \end{array}
\right.$$

By Theorem \ref{bijection between tilting bundle and combination vertices}, the vertices of $\mathcal{G}(\mathcal{T}^{\nu}_{\mathbb{X}})$ and $\Lambda_{(p,q)}^0$ coincide. Let $T$ be a tilting bundle, recall from (\ref{the map phi}) that we can assume $\varphi(T)=(c_1, \cdots, c_{i-1}, c_i, c_{i+1}, \cdots, c_p),$ and we define
$$\mu_{i}^{+}(T):=(T\backslash [D^{\frac{i}{p}}_{\frac{c_i}{q}}])\oplus [D^{\frac{i-1}{p}}_{\frac{c_i+1}{q}}]$$ and $$\mu_{i}^{-}(T):=(T\backslash [D^{\frac{i-1}{p}}_{\frac{c_i}{q}}])\oplus [D^{\frac{i}{p}}_{\frac{c_i-1}{q}}].$$
Then $\mu_{i}^{+}(T)$ and $\mu_{i}^{-}(T)$ are both tilting bundles with
$$\varphi(\mu_{i}^{\pm}(T))=(c_1, \cdots, c_{i-1}, c_i\pm 1, c_{i+1}, \cdots, c_p)=\varphi(T)\pm\epsilon_{i},$$
here, $\epsilon_{i}=(0,\cdots,0,1,0,\cdots,0)$ is $i$-th canonical row vector in $\mathbb{Z}^{p}.$ Therefore, if there exists an edge in the tilting graph $\mathcal{G}(\mathcal{T}_{\mathbb{X}}^{\nu})$,
then there exists a corresponding edge in $\Lambda_{(p,q)}.$
Moreover, according to Proposition \ref{the bound of index}, there are $n(T)=2(p-r(T))$-many edges attached to the tilting bundle $T$ in the graph $\mathcal{G}(\mathcal{T}_{\mathbb{X}}^{\nu}).$

On the other hand, there is an edge between $\alpha=(c_1, c_2, \cdots,c_p)$ and $\beta=(d_1, d_2, \cdots, d_p)$ if and only if there exists a unique $j$ such that

$$d_i=
  \left\{
  \begin{array}{lll}
    c_i\pm 1&& i=j\\
   c_i&&  i\neq j;
  \end{array}
\right.$$
if and only if $\beta=\alpha\pm\epsilon_i$.
Observe that
$\beta=\alpha+\epsilon_i\in \Lambda_{(p,q)}^0$ if and only if $$c_1\leq c_2\leq\cdots\leq c_i\textless c_{i+1}\leq c_{i+2}\leq \cdots\leq c_p,$$ and $\beta=\alpha-\epsilon_i\in \Lambda_{(p,q)}^0$ if and only if $$c_1\leq c_2\leq \cdots\leq c_{i-1}\textless c_{i}\leq c_{i+1}\cdots\leq c_{p}.$$
Therefore, $c_{i}=c_{i+1}$ if and only if $\alpha+\epsilon _i\notin \Lambda_{(p,q)}^0$, $\alpha-\epsilon _{i+1}\notin \Lambda_{(p,q)}^0$. Assume $T=\varphi^{-1}(\alpha),$ then there are $2(p-r(T))=n(T)$-many edges attached to $\alpha$ in $\Lambda_{(p,q)}$. This finishes the proof.

\subsection{Typical examples of tilting graphs}
\begin{exm}
For the case of $p=1$ and $q=2$, recall that $$T^{0}_{\frac{a}{2}}:=\co(-a\vec{x}_2)\oplus \co(-a\vec{x}_2+\vec{x}_2)\oplus\co(-a\vec{x}_2+\vec{c}),\,\,a\in\mathbb{Z}$$ is a tilting bundle. We denote by $$T_{\frac{a}{2}}=\co(-a\vec{x}_2)\oplus\co(-a\vec{x}_2+\vec{c})\oplus S_{0,1},$$ hence $T_{\frac{a}{2}}$ is a tilting sheaf but not a tilting bundle.

By using Proposition \ref{pentagon}, the tilting graph $\mathcal{G}(\mathcal{T}_{\mathbb{X}})$ of ${\rm coh}\mbox{-}\mathbb{X}(1,2)$ can be drawn as below (see also \cite{W2008} Figure 8.1), where the second line is the tilting graph $\mathcal{G}(\mathcal{T}_{\mathbb{X}}^{\nu})$ of ${\rm vec}\mbox{-}\mathbb{X}(1,2)$.
\begin{figure}[H]
\begin{tikzpicture}
\draw[-](0.2,0)--(1.8,0);
\node()at(-0.1,0){\tiny$T_{-\frac{1}{2}}$};
\draw[-](-0.4,0.01)--(-2.4,0.01);
\node()at(2,-0.05){\tiny$T_{\frac{1}{2}}$};
\draw[-,red](0.2,-1)--(0.8,-1);
\draw[-,red](1.2,-1)--(1.8,-1);
\node()at(2,-1.1){{\red \tiny$T^{0}_{\frac{1}{2}}$}};
\draw[-](0,-0.2)--(0,-0.8);
\node()at(0,-1.1){{\red \tiny$T^{0}_{-\frac{1}{2}}$}};
\draw[-](2,-0.8)--(2,-0.8);
\draw[-](2.2,0)--(3.8,0);
\node()at(4,0){\tiny$T_{\frac{3}{2}}$};
\draw[-](4.2,0)--(6.2,0);
\draw[-](4,-0.2)--(4,-0.8);
\draw[-](2,-0.2)--(2,-0.8);
\node()at(4,-1.1){{\red \tiny$T^{0}_{\frac{3}{2}}$}};
\draw[-,red](2.2,-1)--(2.8,-1);
\node()at(3,-1){{\red \tiny$T_{1}^{0}$}};
\draw[-,red](3.2,-1)--(3.8,-1);
\draw[-,red](-0.9,-1)--(-0.3,-1);
\node()at(-1.15,-1){{\red \tiny$T^{0}_{-1}$}};
\draw[-,red](-1.4,-1)--(-2,-1);
\node()at(-1.1,-2){\tiny$T_{-1}$};
\draw[-](-1.4,-2)--(-3.3,-2);
\node()at(1,-2){\tiny$T_{0}$};
\node()at(3,-2){\tiny$T_{1}$};
\node()at(5,-2){\tiny$T_2$};
\draw[-](5.3,-2)--(7.3,-2);
\node()at(5,-1){{\red \tiny$T^{0}_{2}$}};
\draw[-,red](5.2,-1)--(5.8,-1);
\draw[-](1,-1.2)--(1,-1.8);
\draw[-](1.2,-2)--(2.8,-2);
\draw[-](-0.8,-2)--(0.8,-2);
\draw[-](-1,-1.2)--(-1,-1.8);
\draw[-](3,-1.2)--(3,-1.8);
\draw[-](3.2,-2)--(4.8,-2);
\draw[-,red](4.2,-1)--(4.8,-1);
\draw[-](5,-1.2)--(5,-1.8);
\node()at(1,-1){{\red \tiny$T^{0}_{0}$}};
\draw[line width=1pt,dotted](6.1,-1)--(6.6,-1);
\draw[line width=1pt,dotted](-2.3,-1)--(-2.8,-1);
\draw[line width=1pt,dotted](6.5,0)--(7,0);
\draw[line width=1pt,dotted](-2.7,0)--(-3.2,0);
\draw[line width=1pt,dotted](7.5,-2)--(8,-2);
\draw[line width=1pt,dotted](-3.6,-2)--(-4.1,-2);
\end{tikzpicture}
\caption{The tilting graph $\mathcal{G}(\mathcal{T}_{\mathbb{X}})$ of ${\rm coh}\mbox{-}\mathbb{X}(1,2)$}
\end{figure}
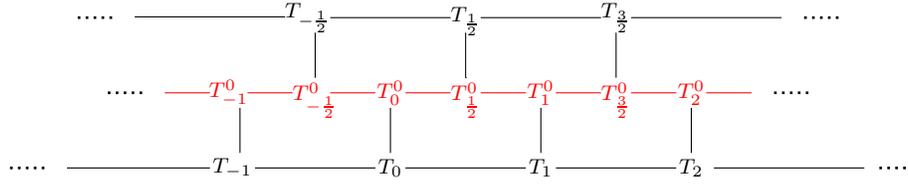
\end{exm}

\begin{exm}
For the case of $p=2$ and $q=2,$ by Theorem \ref{description of tilting graph of vector bundles}, we obtain that the tilting graph $\mathcal{G}(\mathcal{T}_{\mathbb{X}}^{\nu})$ of ${\rm vec}\mbox{-}\mathbb{X}(2,2)$ can be described by the graph $\Lambda_{(2,2)}$, it follows that the shape of $\mathcal{G}(\mathcal{T}_{\mathbb{X}}^{\nu})$ can be drawn as below:
\begin{figure}[H]
\begin{tikzpicture}
\node()at(-2,-2){\tiny$\bullet$};
\node()at(-2,-1){\tiny$\bullet$};
\node()at(-2.5,-1.2){\tiny$(-2,-1)$};
\node()at(-2,0){\tiny$\bullet$};
\node()at(-2.6,0){\tiny$(-2,0)$};
\node()at(-3,-1){\tiny$\bullet$};
\node()at(-3.7,-1){\tiny$(-3,-1)$};
\node()at(-1,-1){\tiny$\bullet$};
\node()at(-0.3,-1){\tiny$(-1,-1)$};
\node()at(-1.3,-2){\tiny$(-2,-2)$};
\node()at(-1,0){\tiny$\bullet$};
\node()at(-0.5,0.2){\tiny$(-1,0)$};
\node()at(-1,1){\tiny$\bullet$};
\node()at(-1.6,1){\tiny$(-1,1)$};
\node()at(0,0){\tiny$\bullet$};
\node()at(0.5,0){\tiny$(0,0)$};
\node()at(0,1){\tiny$\bullet$};
\node()at(0.4,1.2){\tiny$(0,1)$};
\node()at(0,2){\tiny$\bullet$};
\node()at(-.5,2){\tiny$(0,2)$};
\node()at(1,1){\tiny$\bullet$};
\node()at(1.5,1){\tiny$(1,1)$};
\node()at(2.5,2){\tiny$(2,2)$};
\node()at(1,2){\tiny$\bullet$};
\node()at(1.4,2.2){\tiny$(1,2)$};
\node()at(1,3){\tiny$\bullet$};
\node()at(0.5,3){\tiny$(1,3)$};
\node()at(2,2){\tiny$\bullet$};
\draw[-](-2,-2)--(-2,-1);
\draw[-](-3,-1)--(-2,-1);
\draw[-](-2,-1)--(-2,0);
\draw[-](-2,0)--(-1,0);
\draw[-](-1,0)--(0,0);
\draw[-](-1,0)--(-1,1);
\draw[-](0,0)--(0,1);
\draw[-](-1,1)--(0,1);
\draw[-](0,1)--(0,2);
\draw[-](0,1)--(1,1);
\draw[-](1,1)--(1,2);
\draw[-](1,2)--(1,3);
\draw[-](1,2)--(2,2);
\draw[-](-1,-1)--(-2,-1);
\draw[-](-1,-1)--(-1,0);
\draw[-](0,2)--(1,2);

\draw[-,red](-3,-3)--(3,3);
\draw[-,red](2,4)--(-4,-2);
\draw[line width=1pt,dotted](2,3)--(2.4,3.4);
\draw[line width=1pt,dotted](-3,-2)--(-3.4,-2.4);
\end{tikzpicture}
\caption{The tilting graph $\mathcal{G}(\mathcal{T}_{\mathbb{X}}^{\nu})$ of ${\rm vec}\mbox{-}\mathbb{X}(2,2)$}
\end{figure}
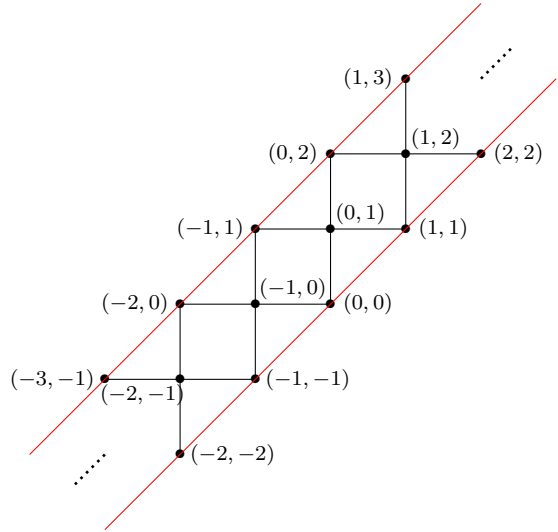
Combining with Proposition \ref{pentagon}, the tilting graph $\mathcal{G}(\mathcal{T}_{\mathbb{X}})$ of ${\rm coh}\mbox{-}\mathbb{X}(2,2)$ can be drawn as below (see also \cite{W2008} Figure 7.2):
{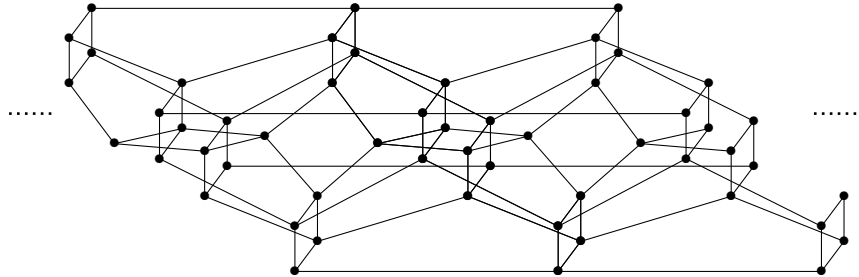
\begin{figure}[h]
\begin{tikzpicture}
\draw[-](0,0)--(0,-0.6);
\draw[-](0.3,-0.2)--(0,-0.6);
\draw[-](0,0)--(0.3,0.4);
\draw[-](0.3,0.4)--(0.3,-0.2);
\node at (0,0) {{\tiny{$\bullet$}}};
\node at (0,-0.6) {{\tiny{$\bullet$}}};
\node at (0.3,0.4) {{\tiny{$\bullet$}}};
\node at (0.3,-0.2) {{\tiny{$\bullet$}}};
\draw[-](0.6,-0.5)--(0.6,-1.1);
\draw[-](0.9,-0.7)--(0.6,-1.1);
\draw[-](0.9,-0.7)--(0.9,-0.1);
\draw[-](0.6,-0.5)--(0.9,-0.1);
\node at (0.6,-0.5) {{\tiny{$\bullet$}}};
\node at (0.6,-1.1) {{\tiny{$\bullet$}}};
\node at (0.9,-0.7) {{\tiny{$\bullet$}}};
\node at (0.9,-0.1) {{\tiny{$\bullet$}}};
\draw[-](1.8,-1.5)--(1.8,-2.1);
\draw[-](2.1,-1.7)--(1.8,-2.1);
\draw[-](2.1,-1.7)--(2.1,-1.1);
\draw[-](1.8,-1.5)--(2.1,-1.1);
\node at (1.8,-1.5) {{\tiny{$\bullet$}}};
\node at (1.8,-2.1) {{\tiny{$\bullet$}}};
\node at (2.1,-1.7) {{\tiny{$\bullet$}}};
\node at (2.1,-1.1) {{\tiny{$\bullet$}}};
\draw[-](-1.2,1)--(-1.2,0.4);
\draw[-](-0.9,0.8)--(-1.2,0.4);
\draw[-](-1.2,1)--(-0.9,1.4);
\draw[-](-0.9,1.4)--(-0.9,0.8);
\node at (-1.2,1) {{\tiny{$\bullet$}}};
\node at (-1.2,0.4) {{\tiny{$\bullet$}}};
\node at (-0.9,1.4) {{\tiny{$\bullet$}}};
\node at (-0.9,0.8) {{\tiny{$\bullet$}}};
\draw[-](0,-0.6)--(1.8,-1.5);
\draw[-](3.5,-0.6)--(1.8,-1.5);
\draw[-](0,0)--(3.5,0);
\draw[-](0.3,0.4)--(2.3,1);
\draw[-](3.8,0.4)--(2.3,1);
\node at (1.4,-0.3) {{\tiny{$\bullet$}}};
\draw[-](0.3,-0.2)--(1.4,-0.3);
\draw[-](2.3,0.4)--(1.4,-0.3);
\draw[-](2.1,-1.1)--(1.4,-0.3);
\node at (2.9,-0.4) {{\tiny{$\bullet$}}};
\node at (-0.6,-0.4) {{\tiny{$\bullet$}}};
\draw[-](2.3,0.4)--(2.9,-0.4);
\draw[-](2.1,-1.1)--(2.9,-0.4);
\draw[-](3.8,-0.2)--(2.9,-0.4);
\draw[-](3.5,0)--(3.5,-0.6);
\draw[-](3.8,-0.2)--(3.5,-0.6);
\draw[-](3.5,0)--(3.8,0.4);
\draw[-](3.8,0.4)--(3.8,-0.2);
\node at (3.5,0) {{\tiny{$\bullet$}}};
\node at (3.5,-0.6) {{\tiny{$\bullet$}}};
\node at (3.8,0.4) {{\tiny{$\bullet$}}};
\node at (3.8,-0.2) {{\tiny{$\bullet$}}};
\draw[-](4.1,-0.5)--(4.1,-1.1);
\draw[-](4.4,-0.7)--(4.1,-1.1);
\draw[-](4.4,-0.7)--(4.4,-0.1);
\draw[-](4.1,-0.5)--(4.4,-0.1);
\node at (4.1,-0.5) {{\tiny{$\bullet$}}};
\node at (4.1,-1.1) {{\tiny{$\bullet$}}};
\node at (4.4,-0.7) {{\tiny{$\bullet$}}};
\node at (4.4,-0.1) {{\tiny{$\bullet$}}};
\draw[-](5.3,-1.5)--(5.3,-2.1);
\draw[-](5.6,-1.7)--(5.3,-2.1);
\draw[-](5.6,-1.7)--(5.6,-1.1);
\draw[-](5.3,-1.5)--(5.6,-1.1);
\node at (5.3,-1.5) {{\tiny{$\bullet$}}};
\node at (5.3,-2.1) {{\tiny{$\bullet$}}};
\node at (5.6,-1.7) {{\tiny{$\bullet$}}};
\node at (5.6,-1.1) {{\tiny{$\bullet$}}};
\draw[-](2.3,1)--(2.3,0.4);
\draw[-](2.6,0.8)--(2.3,0.4);
\draw[-](2.3,1)--(2.6,1.4);
\draw[-](2.6,1.4)--(2.6,0.8);
\node at (2.3,1) {{\tiny{$\bullet$}}};
\node at (2.3,0.4) {{\tiny{$\bullet$}}};
\node at (2.6,1.4) {{\tiny{$\bullet$}}};
\node at (2.6,0.8) {{\tiny{$\bullet$}}};
\draw[-](2.6,1.4)--(-0.9,1.4);
\draw[-](-1.2,1)--(0.3,0.4);
\draw[-](-1.2,0.4)--(-0.6,-0.4);
\draw[-](-0.6,-0.4)--(0.3,-0.2);
\draw[-](-0.9,0.8)--(0.9,-0.1);
\draw[-](2.6,0.8)--(0.9,-0.1);
\draw[-](-0.6,-0.4)--(0.6,-0.5);
\draw[-](0.6,-0.5)--(1.4,-0.3);
\draw[-](2.1,-1.7)--(0.6,-1.1);
\draw[-](4.4,-0.7)--(0.9,-0.7);
\draw[-](4.1,-1.1)--(2.1,-1.7);
\draw[-](5.3,-2.1)--(1.8,-2.1);
\draw[-](5.6,-1.7)--(4.1,-1.1);
\draw[-](5.3,-1.5)--(3.5,-0.6);
\draw[-](4.1,-0.5)--(2.9,-0.4);
\draw[-](4.4,-0.1)--(2.6,0.8);
\draw[-](3.5,0)--(3.5,-0.6);
\draw[-](3.8,-0.2)--(3.5,-0.6);
\draw[-](3.5,0)--(3.8,0.4);
\draw[-](3.8,0.4)--(3.8,-0.2);
\node at (3.5,0) {{\tiny{$\bullet$}}};
\node at (3.5,-0.6) {{\tiny{$\bullet$}}};
\node at (3.8,0.4) {{\tiny{$\bullet$}}};
\node at (3.8,-0.2) {{\tiny{$\bullet$}}};

\draw[-](4.1,-0.5)--(4.1,-1.1);
\draw[-](4.4,-0.7)--(4.1,-1.1);
\draw[-](4.4,-0.7)--(4.4,-0.1);
\draw[-](4.1,-0.5)--(4.4,-0.1);
\node at (4.1,-0.5) {{\tiny{$\bullet$}}};
\node at (4.1,-1.1) {{\tiny{$\bullet$}}};
\node at (4.4,-0.7) {{\tiny{$\bullet$}}};
\node at (4.4,-0.1) {{\tiny{$\bullet$}}};
\draw[-](5.3,-1.5)--(5.3,-2.1);
\draw[-](5.6,-1.7)--(5.3,-2.1);
\draw[-](5.6,-1.7)--(5.6,-1.1);
\draw[-](5.3,-1.5)--(5.6,-1.1);
\node at (5.3,-1.5) {{\tiny{$\bullet$}}};
\node at (5.3,-2.1) {{\tiny{$\bullet$}}};
\node at (5.6,-1.7) {{\tiny{$\bullet$}}};
\node at (5.6,-1.1) {{\tiny{$\bullet$}}};
\draw[-](2.3,1)--(2.3,0.4);
\draw[-](2.6,0.8)--(2.3,0.4);
\draw[-](2.3,1)--(2.6,1.4);
\draw[-](2.6,1.4)--(2.6,0.8);
\node at (2.3,1) {{\tiny{$\bullet$}}};
\node at (2.3,0.4) {{\tiny{$\bullet$}}};
\node at (2.6,1.4) {{\tiny{$\bullet$}}};
\node at (2.6,0.8) {{\tiny{$\bullet$}}};
\draw[-](3.5,-0.6)--(5.3,-1.5);
\draw[-](7,-0.6)--(5.3,-1.5);
\draw[-](3.5,0)--(7,0);
\draw[-](3.8,0.4)--(5.8,1);
\draw[-](7.3,0.4)--(5.8,1);
\node at (4.9,-0.3) {{\tiny{$\bullet$}}};
\draw[-](3.8,-0.2)--(4.9,-0.3);
\draw[-](5.8,0.4)--(4.9,-0.3);
\draw[-](5.6,-1.1)--(4.9,-0.3);
\node at (6.4,-0.4) {{\tiny{$\bullet$}}};
\node at (2.9,-0.4) {{\tiny{$\bullet$}}};
\draw[-](5.8,0.4)--(6.4,-0.4);
\draw[-](5.6,-1.1)--(6.4,-0.4);
\draw[-](7.3,-0.2)--(6.4,-0.4);
\draw[-](7,0)--(7,-0.6);
\draw[-](7.3,-0.2)--(7,-0.6);
\draw[-](7,0)--(7.3,0.4);
\draw[-](7.3,0.4)--(7.3,-0.2);
\node at (7,0) {{\tiny{$\bullet$}}};
\node at (7,-0.6) {{\tiny{$\bullet$}}};
\node at (7.3,0.4) {{\tiny{$\bullet$}}};
\node at (7.3,-0.2) {{\tiny{$\bullet$}}};
\draw[-](7.6,-0.5)--(7.6,-1.1);
\draw[-](7.9,-0.7)--(7.6,-1.1);
\draw[-](7.9,-0.7)--(7.9,-0.1);
\draw[-](7.6,-0.5)--(7.9,-0.1);
\node at (7.6,-0.5) {{\tiny{$\bullet$}}};
\node at (7.6,-1.1) {{\tiny{$\bullet$}}};
\node at (7.9,-0.7) {{\tiny{$\bullet$}}};
\node at (7.9,-0.1) {{\tiny{$\bullet$}}};
\draw[-](8.8,-1.5)--(8.8,-2.1);
\draw[-](9.1,-1.7)--(8.8,-2.1);
\draw[-](9.1,-1.7)--(9.1,-1.1);
\draw[-](8.8,-1.5)--(9.1,-1.1);
\node at (8.8,-1.5) {{\tiny{$\bullet$}}};
\node at (8.8,-2.1) {{\tiny{$\bullet$}}};
\node at (9.1,-1.7) {{\tiny{$\bullet$}}};
\node at (9.1,-1.1) {{\tiny{$\bullet$}}};
\draw[-](5.8,1)--(5.8,0.4);
\draw[-](6.1,0.8)--(5.8,0.4);
\draw[-](5.8,1)--(6.1,1.4);
\draw[-](6.1,1.4)--(6.1,0.8);
\node at (5.8,1) {{\tiny{$\bullet$}}};
\node at (5.8,0.4) {{\tiny{$\bullet$}}};
\node at (6.1,1.4) {{\tiny{$\bullet$}}};
\node at (6.1,0.8) {{\tiny{$\bullet$}}};
\draw[-](6.1,1.4)--(2.6,1.4);
\draw[-](2.3,1)--(3.8,0.4);
\draw[-](2.3,0.4)--(2.9,-0.4);
\draw[-](2.9,-0.4)--(3.8,-0.2);
\draw[-](2.6,0.8)--(4.4,-0.1);
\draw[-](6.1,0.8)--(4.4,-0.1);
\draw[-](2.9,-0.4)--(4.1,-0.5);
\draw[-](4.1,-0.5)--(4.9,-0.3);
\draw[-](5.6,-1.7)--(4.1,-1.1);
\draw[-](7.9,-0.7)--(4.4,-0.7);
\draw[-](7.6,-1.1)--(5.6,-1.7);
\draw[-](8.8,-2.1)--(5.3,-2.1);
\draw[-](9.1,-1.7)--(7.6,-1.1);
\draw[-](8.8,-1.5)--(7,-0.6);
\draw[-](7.6,-0.5)--(6.4,-0.4);
\draw[-](7.9,-0.1)--(6.1,0.8);
\draw[line width=1pt,dotted](8.7,0)--(9.3,0);
\draw[line width=1pt,dotted](-2,0)--(-1.4,0);
\end{tikzpicture}
\caption{The tilting graph $\mathcal{G}(\mathcal{T}_{\mathbb{X}})$ of ${\rm coh}\mbox{-}\mathbb{X}(2,2)$}
\end{figure}}
\end{exm}

\begin{exm}
Consider the case of $p=2$ and $q=3,$ by using Theorem \ref{description of tilting graph of vector bundles}, the shape of the tilting graph $\mathcal{G}(\mathcal{T}_{\mathbb{X}}^{\nu})$ of ${\rm vec}\mbox{-}\mathbb{X}(2,3)$ is drawn as below:
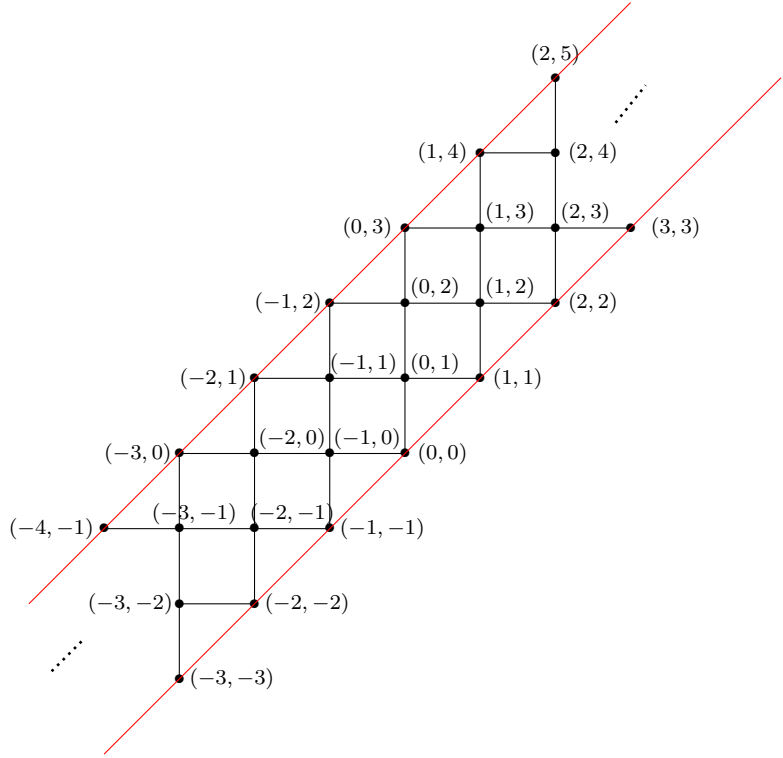
\begin{figure}[H]
\begin{tikzpicture}
\node()at(-3,-3){\tiny$\bullet$};
\node()at(-2.3,-3){\tiny$(-3,-3)$};
\node()at(-3,-2){\tiny$\bullet$};
\node()at(-3.65,-2){\tiny$(-3,-2)$};
\node()at(-3,-1){\tiny$\bullet$};
\node()at(-2.8,-0.8){\tiny$(-3,-1)$};
\node()at(-4,-1){\tiny$\bullet$};
\node()at(-4.7,-1){\tiny$(-4,-1)$};
\node()at(-3,0){\tiny$\bullet$};
\node()at(-3.55,0){\tiny$(-3,0)$};
\node()at(-2,-2){\tiny$\bullet$};
\node()at(-1.3,-2){\tiny$(-2,-2)$};
\node()at(-2,-1){\tiny$\bullet$};
\node()at(-1.5,-0.8){\tiny$(-2,-1)$};
\node()at(-2,0){\tiny$\bullet$};
\node()at(-1.5,0.2){\tiny$(-2,0)$};
\node()at(-2,1){\tiny$\bullet$};
\node()at(-2.55,1){\tiny$(-2,1)$};
\node()at(-1,-1){\tiny$\bullet$};
\node()at(-0.3,-1){\tiny$(-1,-1)$};
\node()at(-1,0){\tiny$\bullet$};
\node()at(-0.5,0.2){\tiny$(-1,0)$};
\node()at(-1,1){\tiny$\bullet$};
\node()at(-0.55,1.2){\tiny$(-1,1)$};
\node()at(-1,2){\tiny$\bullet$};
\node()at(-1.55,2){\tiny$(-1,2)$};
\node()at(0,0){\tiny$\bullet$};
\node()at(0.5,0){\tiny$(0,0)$};
\node()at(0,1){\tiny$\bullet$};
\node()at(0.4,1.2){\tiny$(0,1)$};
\node()at(0.4,2.2){\tiny$(0,2)$};
\node()at(0,2){\tiny$\bullet$};
\node()at(0,3){\tiny$\bullet$};
\node()at(-0.5,3){\tiny$(0,3)$};
\node()at(1,1){\tiny$\bullet$};
\node()at(1.5,1){\tiny$(1,1)$};
\node()at(1,2){\tiny$\bullet$};
\node()at(1.4,2.2){\tiny$(1,2)$};
\node()at(1,3){\tiny$\bullet$};
\node()at(1,4){\tiny$\bullet$};
\node()at(0.5,4){\tiny$(1,4)$};
\node()at(2,2){\tiny$\bullet$};
\node()at(2.5,2){\tiny$(2,2)$};
\node()at(2,3){\tiny$\bullet$};
\node()at(2.4,3.2){\tiny$(2,3)$};
\node()at(1.4,3.2){\tiny$(1,3)$};
\node()at(2,4){\tiny$\bullet$};
\node()at(2.5,4){\tiny$(2,4)$};
\node()at(2,5){\tiny$\bullet$};
\node()at(3,3){\tiny$\bullet$};
\node()at(3.6,3){\tiny$(3,3)$};
\draw[-](-3,-3)--(-3,-2);
\draw[-](2,3)--(3,3);
\draw[-](-3,-2)--(-3,-1);
\draw[-](-3,-2)--(-2,-2);
\draw[-](-3,-1)--(-2,-1);
\draw[-](-3,-1)--(-3,0);
\draw[-](-3,0)--(-2,0);
\draw[-](-2,-1)--(-2,0);
\draw[-](-2,-2)--(-2,-1);
\draw[-](-2,0)--(-2,1);
\draw[-](-2,0)--(-1,0);
\draw[-](-1,0)--(-1,1);
\draw[-](-1,0)--(0,0);
\draw[-](-1,0)--(-1,-1);
\draw[-](-1,1)--(-1,2);
\draw[-](-1,1)--(0,1);
\draw[-](0,1)--(0,0);
\draw[-](0,1)--(0,2);
\draw[-](-3,-1)--(-4,-1);
\draw[-](0,2)--(-1,2);
\draw[-](1,1)--(0,1);
\draw[-](1,1)--(1,2);
\draw[-](0,2)--(0,3);
\draw[-](0,2)--(1,2);
\draw[-](1,2)--(1,3);
\draw[-](1,2)--(2,2);
\draw[-](0,3)--(1,3);
\draw[-](1,3)--(1,4);
\draw[-](1,4)--(2,4);
\draw[-](1,3)--(2,3);
\draw[-](2,2)--(2,3);
\draw[-](2,3)--(2,4);
\draw[-](2,4)--(2,5);
\draw[-](-1,-1)--(-2,-1);
\draw[-](-1,1)--(-2,1);
\node()at(2,5.3){\tiny$(2,5)$};
\draw[-,red](-4,-4)--(5,5);
\draw[-,red](3,6)--(-5,-2);
\draw[line width=1pt,dotted](-4.3,-2.5)--(-4.7,-2.9);
\draw[line width=1pt,dotted](2.8,4.4)--(3.2,4.9);
\end{tikzpicture}
\caption{The tilting graph $\mathcal{G}(\mathcal{T}_{\mathbb{X}}^{\nu})$ of ${\rm vec}\mbox{-}\mathbb{X}(2,3)$}
\end{figure}
It can be imaginable that the shape of the tilting graph $\mathcal{G}(\mathcal{T}_{\mathbb{X}})$ of ${\rm coh}\mbox{-}\mathbb{X}(2,3)$ is very difficult, here we omit the description of the tilting graph of ${\rm coh}\mbox{-}\mathbb{X}(2,3)$.
\end{exm}

\section{Geometric interpretation of automorphism group of ${\rm coh}\mbox{-}\mathbb{X}(p,q)$}

This section is devoted to giving a geometric interpretation for the automorphism group $\Aut({\rm coh}\mbox{-}\mathbb{X}(p,q))$ of the category ${\rm coh}\mbox{-}\mathbb{X}(p,q)$. We show that $\Aut({\rm coh}\mbox{-}\mathbb{X}(p,q))$ is isomorphic to the mapping class group of $A_{p,q}.$

\subsection{Mapping class group} Let's recall from \cite{ASS2012} for the mapping class group of the marked annulus $A_{p,q}$ first.
Write $A_{p,q}=(S, M)$, where
$S$ is an annulus and $M$ is the set of all the marked points in $A_{p,q}.$

Recall that two homeomorphisms $f,\, g$ of $S$ are  \emph{isotopic} if there is a continuous function $H: S\times [0, 1]\longrightarrow S$ such that $H(x, 0)=f$, $H(x, 1)=g$ and $H(x, t):S\longrightarrow S$ is a homeomorphism for each $t\in [0,1]$.

Denote by ${\rm Homeo^{+}}(S, M)$ the group of orientation preserving homeomorphisms from $S$ to $S$ which map $M$ to $M$. Note that there do not require that the points in $M$ are fixed, neither that the points on the boundary of $S$ are fixed nor that each boundary component is mapped to itself. However, if a boundary component is mapped to another component, then $p=q$.

\begin{defn}
A homeomorphism $f: S\longrightarrow S$ is \emph{isotopic to the identity relative to} $M$, if $f$ is isotopic to the identity via an isotopy $H$ that fixes $M$ pointwise, i.e., $H(x, t)=x$ for all $x\in M$ and $t\in [0,1]$.
\end{defn}

Let ${\rm Homeo_0}(S, M)$ be the set $$\{f\in{\rm Homeo^{+}}(S, M)| f \;{\rm is\; isotopic\; to\; the\; identity\; relative\; to\;} M\}.$$ The \emph{mapping class group} $\mathcal{MG} (S, M)$ of $(S, M)$ is defined to be the quotient
$$\mathcal{MG} (S, M)={\rm Homeo^{+}}(S, M)/{\rm Homeo_0}(S, M).$$
Denote by $H_{p,q}=\langle r_1, r_2|r_1r_2=r_2r_1, r_1^{p}=r_2^{q}\rangle$ and
\begin{equation}\label{mapping class group}
\widetilde{H}_{p,q}=
  \left\{
  \begin{array}{ll}
   H_{p,q}, & p\neq q;\\
   H_{p,q}\times\mathbb{Z}_2, &  p=q.\\
  \end{array}
\right.\end{equation}

It is well known that the mapping class group $\mathcal{MG} (A_{p,q})$ of the marked annulus $A_{p,q}$ is isomorphism to $\widetilde{H}_{p,q}$, c.f. \cite{ASS2012}. We identify them in the following.

\begin{rem}\label{rotation}
In our paper, let $r_1$ (\emph{resp.} $r_2$) be the rotation in anti-clockwise direction (\emph{resp.} in clockwise direction) sending each marked point in $\partial$ (\emph{resp.} $\partial^{\prime}$) to the next marked point in $\partial$ (\emph{resp.} $\partial^{\prime}$), i.e., for any $a, b\in\mathbb{Z},$
\begin{align*}&r_1((\frac{a}{p})_{\partial})=(\frac{a+1}{p})_{\partial}, &&r_1((\frac{b}{q})_{\partial^{\prime}})=(\frac{b}{q})_{\partial^{\prime}};\\
&r_2((\frac{a}{p})_{\partial})=(\frac{a}{p})_{\partial}, &&r_2((\frac{b}{q})_{\partial^{\prime}})=(\frac{b-1}{q})_{\partial^{\prime}}.
\end{align*}
For the case of $p=q$, there exists an involution $\sigma$ of $\mathcal{MG} (A_{p,q})$ maps $\partial$ to $\partial^{\prime}$, satisfying that $\sigma((\frac{a}{p})_{\partial})=(-\frac{a}{p})_{\partial^{\prime}}$ for any $a\in\mathbb{Z}$, in particular, $\sigma(0_\partial)=0_{\partial^{\prime}}$.
\end{rem}

Let $\Aut({\rm coh}\mbox{-}\mathbb{X}(p,q))$ be the automorphism group of the category ${\rm coh}\mbox{-}\mathbb{X}(p,q)$ consists of isomorphism classes of $\mathbf k$-linear self-equivalences on ${\rm coh}\mbox{-}\mathbb{X}(p,q)$.
Then we have
\begin{prop}\label{isomorphism}
$\Aut({\rm coh}\mbox{-}\mathbb{X}(p,q))\cong \mathcal{MG} (A_{p,q}).$
\end{prop}

\begin{proof}
By \cite[Theorem 3.4]{LenzingMelter2000}, we have
\begin{equation}\label{automorphsim group formula}
\Aut({\rm coh}\mbox{-}\mathbb{X}(p,q))\cong
  \left\{
  \begin{array}{ll}
   \mathbb{L}(p,q), & p\neq q;\\
   \mathbb{L}(p,q)\times\mathbb{Z}_2, &  p=q.\\
  \end{array}
\right.\end{equation}

\noindent Observe that $\mathbb{L}(p,q)$ is generated by $\vec{x}_1, \vec{x}_2$ subject to the relation $p\vec{x}_1=q\vec{x}_2$.
By comparing \eqref{mapping class group} and \eqref{automorphsim group formula}, we see that the map $\vec{x}_1\mapsto r_1, \,\,\vec{x}_2\mapsto r_2$ gives a group isomorphism from $\mathbb{L}(p,q)$ to $H_{p,q}$, which induces a group isomorphism from $\Aut({\rm coh}\mbox{-}\mathbb{X}(p,q))$ to $\mathcal{MG} (A_{p,q}).$
\end{proof}

\begin{rem}\label{rem of psi}
In the case $p=q$, recall from \cite{LenzingMelter2000} that there exists a geometric automorphism of $\mathbb{X}(p,q)$ by exchanging the exceptional points $\infty$ and $0$, which induces isomorphism
\begin{equation}\label{involution on category}\sigma_{1,2}: {\rm coh}\mbox{-}\mathbb{X}(p,q)\longrightarrow{\rm coh}\mbox{-}\mathbb{X}(p,q),
\end{equation}
by exchanging $S_{\infty,i}^{(j)}$ with $S_{0,i}^{(j)}$ for any $0\leq i\leq p-1$ and $j\in\mathbb{Z}_{\geq 1}$, and sending line bundles $\co(l_1\vec{x}_1+l_2\vec{x}_2+l\vc)$ to $\co(l_1\vec{x}_2+l_2\vec{x}_1+l\vc)$ for any $0\leq l_i\leq p-1$ with $i=1,2$ and $l\in\mathbb{Z}$.

For any $\vec{x}\in\mathbb{L}(p,q)$, denote by $(\vec{x}):=-\otimes \co(\vec{x})$, which is the auto-equivalence functor in $\Aut({\rm coh}\mbox{-}\mathbb{X}(p,q))$ given by the grading twist with $\vec{x}$. Then we have an isomorphism of groups
\begin{equation}\label{automor psi}\psi: \mathcal{MG} (A_{p,q})\longrightarrow\Aut({\rm coh}\mbox{-}\mathbb{X}(p,q)),
\end{equation}
which sends
$$r_1\mapsto (\vec{x}_1);\quad r_2\mapsto (\vec{x}_2); \quad {\text {and}} \quad \sigma\mapsto \sigma_{1,2} \;\;({\text {when}}\;\; p=q).$$
\end{rem}

\subsection{Proof of Theorem \ref{compatible}}

For any triangulation $\Gamma$ of the marked annulus $A_{p,q}$, we assume $\Gamma$ consists of the following positive bridging arcs and peripheral arcs:
$$[D^{\frac{a_{i}}{p}}_{\frac{b_{i}}{q}}] \;(1\leq i\leq n_1),\; [D^{\frac{i_k-j_k-1}{p}, \frac{i_k}{p}}]\;(1\leq k\leq n_2),\; [D_{\frac{-i_k}{q}, \frac{j_k-i_k+1}{q}}]\;(1\leq k\leq n_3).$$
According to the map (\ref{map}), we obtain a tilting sheaf $$\phi(\Gamma)=\oplus_{i=1}^{n_1}\mathcal{O}(a_{i}\vec{x}_1-b_{i}\vec{x}_2)\oplus(\oplus_{k=1}^{n_2}S_{\infty,i_k}^{(j_k)})\oplus(\oplus_{k=1}^{n_3}S_{0,i_k}^{(j_k)}).$$

Recall from \eqref{automor psi} that $\psi$ is a group isomorphism, we only need to prove that (\ref{compatible of groups action}) holds for any generator $f\in\mathcal{MG} (A_{p,q}).$ First assume $f=r_1$, then we have $\psi(f)=(\vec{x}_1)$. Recall that $$r_1((\frac{a}{p})_{\partial})=(\frac{a+1}{p})_{\partial},\,\, r_1((\frac{b}{q})_{\partial^{\prime}})=(\frac{b}{q})_{\partial^{\prime}},\,\,a, b\in\mathbb{Z}.$$ It follows that the triangulation $f(\Gamma)$ consists of the following arcs:
$$[D^{\frac{a_{i}+1}{p}}_{\frac{b_{i}}{q}}] \;(1\leq i\leq n_1),\; [D^{\frac{i_k-j_k}{p}, \frac{i_k+1}{p}}]\;(1\leq k\leq n_2),\; [D_{\frac{-i_k}{q}, \frac{j_k-i_k+1}{q}}]\;(1\leq k\leq n_3).$$
Then \begin{align*}\phi(f(\Gamma))=&\oplus_{i=1}^{n_1}\mathcal{O}((a_{i}+1)\vec{x}_1-b_{i}\vec{x}_2)
\oplus(\oplus_{k=1}^{n_2}S_{\infty,i_k+1}^{(j_k)})\oplus(\oplus_{k=1}^{n_3}S_{0,i_k}^{(j_k)})\\
=&\phi(\Gamma)(\vec{x}_1)\\
=&\psi(f)(\phi(\Gamma)).
\end{align*}
Hence (\ref{compatible of groups action}) holds in this case. For $f=(r_2)$, the proof is similar.

If $p=q$, then $\mathcal{MG} (A_{p,q})$ has another generator (involution) $f=\sigma$ such that $\sigma((\frac{a}{p})_{\partial})=(-\frac{a}{p})_{\partial^{\prime}}$.
It follows that the triangulation $f(\Gamma)$ consists of the following arcs:
$$[D^{-\frac{b_{i}}{p}}_{-\frac{a_{i}}{p}}] \;(1\leq i\leq n_1),\; [D_{-\frac{i_k}{p}, \frac{j_k-i_k+1}{p}}]\,(1\leq k\leq n_2),\, [D^{\frac{i_k-j_k-1}{p}, \frac{i_k}{p}}]\,(1\leq k\leq n_3).$$
Then $$\phi (f(\Gamma))=\oplus_{i=1}^{n_1}\mathcal{O}(a_{i}\vec{x}_2-b_{i}\vec{x}_1)\oplus(\oplus_{k=1}^{n_2}S_{0,i_k}^{(j_k)})\oplus(\oplus_{k=1}^{n_3}S_{\infty,i_k}^{(j_k)})=\psi(f)(\phi(\Gamma)).$$
Hence (\ref{compatible of groups action}) holds.
This finishes the proof.

Recall from \cite{BQ2012} that the \emph{exchange graph} $EG(A_{p,q})$ of $A_{p,q}$ has as vertices the triangulations of $A_{p,q}$, with an edge between two vertices $\Gamma$ and $\Gamma^{\prime}$ whenever $\Gamma^{\prime}$ is obtained from $\Gamma$ by the flip of an arc. By Theorem \ref{compatible}, we have the following corollary.

\begin{cor}
For any $f\in \mathcal{MG} (A_{p,q})$, we have $\phi(f(EG(A_{p,q}))=\psi(f)(\phi(EG(A_{p,q}))$.
\end{cor}

\subsection{An involution on $\Lambda_{(p,p)}$}

Recall from (\ref{the map phi}) that there exists a bijective map
$$\varphi: \mathcal{T}^{\nu}_{\mathbb{X}}\longrightarrow \Lambda_{(p,p)}^0; \quad T\mapsto (c_1, c_2, \cdots, c_p)$$
where $\Lambda_{(p,p)}^0=\{(c_1, \cdots, c_p)\in\mathbb{Z}^{p}|c_1\leq \cdots\leq c_p\leq c_1+p\}.$

In case $p=q$, by \eqref{involution on category} there exists an involution $\sigma_{1,2}$ on ${\rm coh}\mbox{-}\mathbb{X}(p,p)$, which induces an automorphism on $\mathcal{G}(\mathcal{T}^{\nu}_{\mathbb{X}})$, still denoted by $\sigma_{1,2}$.

\begin{prop}\label{bijective map of vertex such that commutative}
There exists an involution $\rho: \Lambda_{(p,p)}\longrightarrow \Lambda_{(p,p)}$ such that the following diagram commutes
\begin{figure}[H]\label{commutative of varphirho}
\begin{tikzpicture}
\draw[->](0.25,0)--(1.7,0);
\node()at(-0.4,0){$\mathcal{G}(\mathcal{T}^{\nu}_{\mathbb{X}})$};
\node()at(1,0.2){\tiny{$\sigma_{1,2}$}};
\node()at(1,-2.3){\tiny{$\rho$}};
\node()at(-0.6,-1.2){\tiny{$\varphi$}};
\node()at(2.6,-1.2){\tiny{$\varphi$}};
\node()at(2.3,0){$\mathcal{G}(\mathcal{T}^{\nu}_{\mathbb{X}})$};
\draw[->](-0.4,-0.3)--(-0.4,-2.3);
\draw[->](2.4,-0.3)--(2.4,-2.3);
\node()at(2.45,-2.55){{\tiny{$\Lambda_{(p,p)}$}}};
\node()at(-0.35,-2.55){{\tiny{$\Lambda_{(p,p)}$}}};
\draw[->](0,-2.5)--(2,-2.5);
\end{tikzpicture}
\end{figure}
\end{prop}

\begin{proof}
Let $T$ be a tilting bundle in ${\rm coh}\mbox{-}\mathbb{X}(p,p),$ recall from (\ref{the map phi}) that we can assume $$\varphi(T)=(c_1, \cdots, c_{i-1}, c_i, c_{i+1}, \cdots, c_p),$$ that is, $T$ corresponds to a triangulation $\Gamma$ of (a parallelogram in) $\mathbb{U}$ as follows:

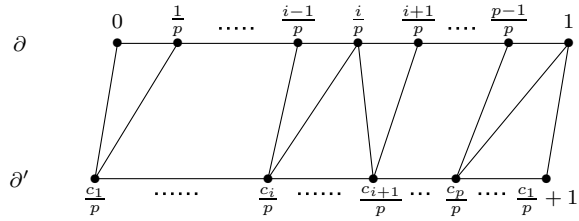
\begin{figure}[H]
\begin{tikzpicture}
\draw[-](0,0)--(0.3,1.8);
\draw[-](0,0)--(6,0);
\draw[-](6,0)--(6.3,1.8);
\draw[-](0.3,1.8)--(6.3,1.8);
\node()at(0,0){\tiny$\bullet$};
\node()at(6,0){\tiny$\bullet$};
\node()at(0.3,1.8){\tiny$\bullet$};
\node()at(6.3,1.8){\tiny$\bullet$};
\node()at(1.1,1.8){\tiny$\bullet$};
\draw[-](0,0)--(1.1,1.8);
\node()at(2.7,1.8){\tiny$\bullet$};
\node()at(3.5,1.8){\tiny$\bullet$};
\node()at(4.3,1.8){\tiny$\bullet$};
\node()at(2.3,0){\tiny$\bullet$};
\node()at(3.7,0){\tiny$\bullet$};
\draw[-](2.7,1.8)--(2.3,0);
\draw[-](3.5,1.8)--(2.3,0);
\draw[-](3.5,1.8)--(3.7,0);
\draw[-](4.3,1.8)--(3.7,0);
\node()at(5.5,1.8){\tiny$\bullet$};
\node()at(4.8,0){\tiny$\bullet$};
\draw[-](5.5,1.8)--(4.8,0);
\draw[-](6.3,1.8)--(4.8,0);
\node()at(0,-0.3){\tiny$\frac{c_1}{p}$};

\node()at(2.3,-0.3){\tiny$\frac{c_i}{p}$};
\node()at(3.8,-0.3){\tiny$\frac{c_{i+1}}{p}$};
\node()at(4.8,-0.3){\tiny$\frac{c_p}{p}$};
\node()at(0.3,2.1){\tiny$0$};
\node()at(1.1,2.1){\tiny$\frac{1}{p}$};
\node()at(2.7,2.1){\tiny$\frac{i-1}{p}$};
\node()at(3.5,2.1){\tiny$\frac{i}{p}$};
\node()at(4.3,2.1){\tiny$\frac{i+1}{p}$};
\node()at(5.5,2.1){\tiny$\frac{p-1}{p}$};
\node()at(6.3,2.1){\tiny$1$};
\node()at(6,-0.3){\tiny$\frac{c_1}{p}+1$};
\draw[line width=1pt,dotted](0.8,-0.2)--(1.4,-0.2);
\draw[line width=1pt,dotted](4.2,-0.2)--(4.5,-0.2);
\draw[line width=1pt,dotted](5.1,-0.2)--(5.5,-0.2);
\draw[line width=1pt,dotted](1.65,2)--(2.15,2);
\draw[line width=1pt,dotted](4.7,2)--(5.1,2);
\draw[line width=1pt,dotted](2.7,-0.2)--(3.3,-0.2);
\node()at(-1,0){\tiny$\partial^{\prime}$};
\node()at(-1,1.8){\tiny$\partial$};
\end{tikzpicture}
\end{figure}
\noindent According to Remark \ref{rotation}, we obtain that $\sigma(\Gamma)$ is a triangulation of the following form:

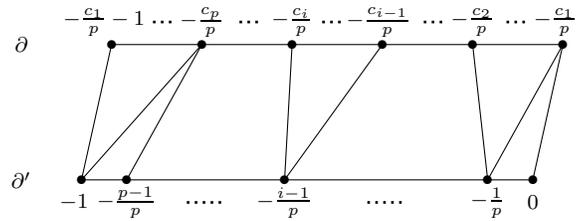
\begin{figure}[H]
\begin{tikzpicture}
\draw[-](-0.7,0)--(-0.3,1.8);
\draw[-](-0.7,0)--(5.3,0);
\draw[-](5.7,1.8)--(5.3,0);
\draw[-](-0.3,1.8)--(5.7,1.8);
\node()at(-0.7,0){\tiny$\bullet$};
\node()at(5.3,0){\tiny$\bullet$};
\node()at(-0.3,1.8){\tiny$\bullet$};
\node()at(5.7,1.8){\tiny$\bullet$};
\node()at(5.3,-0.3){\tiny$0$};
\node()at(-0.8,-0.3){\tiny$-1$};
\node()at(-0.4,2.1){\tiny$-\frac{c_1}{p}-1$};
\node()at(5.6,2.1){\tiny$-\frac{c_1}{p}$};
\node()at(0.9,1.8){\tiny$\bullet$};
\node()at(2.1,1.8){\tiny$\bullet$};
\node()at(-0.1,0){\tiny$\bullet$};
\node()at(4.5,1.8){\tiny$\bullet$};
\node()at(4.7,0){\tiny$\bullet$};
\node()at(4.7,-0.3){\tiny$-\frac{1}{p}$};
\node()at(4.5,2.1){\tiny$-\frac{c_2}{p}$};
\draw[-](5.7,1.8)--(4.7,0);
\draw[-](4.5,1.8)--(4.7,0);
\node()at(0.9,2.1){\tiny$-\frac{c_p}{p}$};
\node()at(2.1,2.1){\tiny$-\frac{c_{i}}{p}$};
\draw[-](0.9,1.8)--(-0.7,0);
\draw[-](0.9,1.8)--(-0.1,0);
\node()at(-0.1,-0.3){\tiny$-\frac{p-1}{p}$};
\node()at(-1.5,0){\tiny$\partial^{\prime}$};
\node()at(-1.5,1.8){\tiny$\partial$};
\node()at(3.3,1.8){\tiny$\bullet$};
\node()at(3.25,2.1){\tiny$-\frac{c_{i-1}}{p}$};
\node()at(2,-0.3){\tiny$-\frac{i-1}{p}$};
\node()at(2,0){\tiny$\bullet$};
\draw[-](2,0)--(3.3,1.8);
\draw[-](2,0)--(2.1,1.8);
\draw[line width=1pt,dotted](0.25,2.1)--(0.55,2.1);
\draw[line width=1pt,dotted](1.4,2.1)--(1.7,2.1);
\draw[line width=1pt,dotted](4.95,2.1)--(5.25,2.1);
\draw[line width=1pt,dotted](2.5,2.1)--(2.8,2.1);
\draw[line width=1pt,dotted](3.85,2.1)--(4.15,2.1);
\draw[line width=1pt,dotted](0.7,-0.3)--(1.2,-0.3);
\draw[line width=1pt,dotted](3.1,-0.3)--(3.6,-0.3);
\end{tikzpicture}
\end{figure}
\noindent which corresponds to the tilting bundle $\sigma_{1,2}(T)$; see (\ref{involution on category}).
By using (\ref{periodicity of covering map}), the above triangulation can be converted to its normal form (\ref{assumption on chains}) as follows:
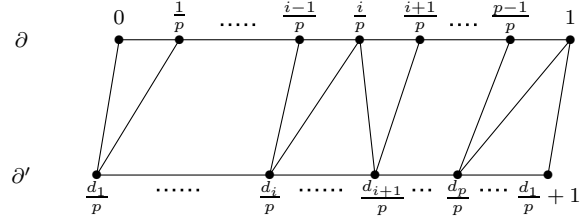
\begin{figure}[H]
\begin{tikzpicture}
\draw[-](0,0)--(0.3,1.8);
\draw[-](0,0)--(6,0);
\draw[-](6,0)--(6.3,1.8);
\draw[-](0.3,1.8)--(6.3,1.8);
\node()at(0,0){\tiny$\bullet$};
\node()at(6,0){\tiny$\bullet$};
\node()at(0.3,1.8){\tiny$\bullet$};
\node()at(6.3,1.8){\tiny$\bullet$};
\node()at(1.1,1.8){\tiny$\bullet$};
\draw[-](0,0)--(1.1,1.8);
\node()at(2.7,1.8){\tiny$\bullet$};
\node()at(3.5,1.8){\tiny$\bullet$};
\node()at(4.3,1.8){\tiny$\bullet$};
\node()at(2.3,0){\tiny$\bullet$};
\node()at(3.7,0){\tiny$\bullet$};
\draw[-](2.7,1.8)--(2.3,0);
\draw[-](3.5,1.8)--(2.3,0);
\draw[-](3.5,1.8)--(3.7,0);
\draw[-](4.3,1.8)--(3.7,0);
\node()at(5.5,1.8){\tiny$\bullet$};
\node()at(4.8,0){\tiny$\bullet$};
\draw[-](5.5,1.8)--(4.8,0);
\draw[-](6.3,1.8)--(4.8,0);
\node()at(0,-0.3){\tiny$\frac{d_1}{p}$};

\node()at(2.3,-0.3){\tiny$\frac{d_i}{p}$};
\node()at(3.8,-0.3){\tiny$\frac{d_{i+1}}{p}$};
\node()at(4.8,-0.3){\tiny$\frac{d_p}{p}$};
\node()at(0.3,2.1){\tiny$0$};
\node()at(1.1,2.1){\tiny$\frac{1}{p}$};
\node()at(2.7,2.1){\tiny$\frac{i-1}{p}$};
\node()at(3.5,2.1){\tiny$\frac{i}{p}$};
\node()at(4.3,2.1){\tiny$\frac{i+1}{p}$};
\node()at(5.5,2.1){\tiny$\frac{p-1}{p}$};
\node()at(6.3,2.1){\tiny$1$};
\node()at(6,-0.3){\tiny$\frac{d_1}{p}+1$};
\draw[line width=1pt,dotted](0.8,-0.2)--(1.4,-0.2);
\draw[line width=1pt,dotted](4.2,-0.2)--(4.5,-0.2);
\draw[line width=1pt,dotted](5.1,-0.2)--(5.5,-0.2);
\draw[line width=1pt,dotted](1.65,2)--(2.15,2);
\draw[line width=1pt,dotted](4.7,2)--(5.1,2);
\draw[line width=1pt,dotted](2.7,-0.2)--(3.3,-0.2);
\node()at(-1,0){\tiny$\partial^{\prime}$};
\node()at(-1,1.8){\tiny$\partial$};
\end{tikzpicture}
\end{figure}
\noindent This induces a well-defined map:
$$\rho: \Lambda_{(p,p)}^0\longrightarrow \Lambda_{(p,p)}^0, \quad (c_1, c_2, \cdots, c_p)\mapsto (d_1, d_2, \cdots, d_p).$$
By construction we have $\varphi(\sigma_{1,2}(T))=\rho(\varphi(T))$.
According to the proof of Theorem \ref{description of tilting graph of vector bundles}, we see that $\varphi$ is an isomorphism between the graphs $\mathcal{G}(\mathcal{T}^{\nu}_{\mathbb{X}})$ and $\Lambda_{(p,p)}$. This proves the commutativity of the diagram. Moreover,
since $\sigma_{1,2}$ is an involution, it follows that $\rho$ is also an involution.
\end{proof}

\begin{exm}
For the case of $p=q=4,$ let $T$ be a tilting bundle and $\varphi(T)=(0, 0, 1, 4),$ that is, $T$ corresponds to a triangulation $\Gamma$ of (a parallelogram in) $\mathbb{U}$ as follows:
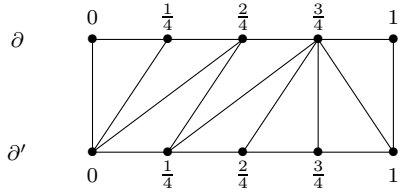
\begin{figure}[H]
\begin{tikzpicture}
\draw[-](0,0)--(0,1.5);
\draw[-](0,0)--(4,0);
\draw[-](0,1.5)--(4,1.5);
\draw[-](4,0)--(4,1.5);
\node()at(0,0){\tiny$\bullet$};
\node()at(4,0){\tiny$\bullet$};
\node()at(0,1.5){\tiny$\bullet$};
\node()at(4,1.5){\tiny$\bullet$};
\node()at(1,0){\tiny$\bullet$};
\node()at(2,0){\tiny$\bullet$};
\node()at(3,0){\tiny$\bullet$};
\node()at(1,1.5){\tiny$\bullet$};
\node()at(2,1.5){\tiny$\bullet$};
\node()at(3,1.5){\tiny$\bullet$};
\node()at(0,1.8){\tiny$0$};
\node()at(1,1.8){\tiny$\frac{1}{4}$};
\node()at(2,1.8){\tiny$\frac{2}{4}$};
\node()at(3,1.8){\tiny$\frac{3}{4}$};
\node()at(4,1.8){\tiny$1$};
\node()at(0,-0.3){\tiny$0$};
\node()at(1,-0.3){\tiny$\frac{1}{4}$};
\node()at(2,-0.3){\tiny$\frac{2}{4}$};
\node()at(3,-0.3){\tiny$\frac{3}{4}$};
\node()at(4,-0.3){\tiny$1$};
\node()at(-1,1.5){\tiny$\partial$};
\node()at(-1,0){\tiny$\partial^{\prime}$};

\draw[-](0,0)--(1,1.5);
\draw[-](0,0)--(2,1.5);
\draw[-](1,0)--(2,1.5);
\draw[-](1,0)--(3,1.5);
\draw[-](2,0)--(3,1.5);
\draw[-](3,0)--(3,1.5);
\draw[-](4,0)--(3,1.5);
\end{tikzpicture}
\end{figure}
\noindent Then $\sigma_{1,2}(T)$ corresponds to a triangulation $\sigma(\Gamma)$ of (a parallelogram in) $\mathbb{U}$ as follows:

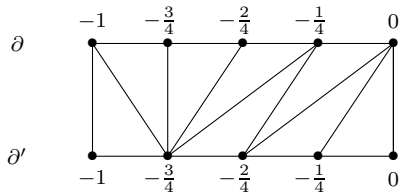
\begin{figure}[H]
\begin{tikzpicture}
\draw[-](0,0)--(0,1.5);
\draw[-](0,0)--(4,0);
\draw[-](0,1.5)--(4,1.5);
\draw[-](4,0)--(4,1.5);
\node()at(0,0){\tiny$\bullet$};
\node()at(4,0){\tiny$\bullet$};
\node()at(0,1.5){\tiny$\bullet$};
\node()at(4,1.5){\tiny$\bullet$};
\node()at(1,0){\tiny$\bullet$};
\node()at(2,0){\tiny$\bullet$};
\node()at(3,0){\tiny$\bullet$};
\node()at(1,1.5){\tiny$\bullet$};
\node()at(2,1.5){\tiny$\bullet$};
\node()at(3,1.5){\tiny$\bullet$};
\node()at(0,1.8){\tiny$-1$};
\node()at(0.9,1.8){\tiny$-\frac{3}{4}$};
\node()at(1.9,1.8){\tiny$-\frac{2}{4}$};
\node()at(2.9,1.8){\tiny$-\frac{1}{4}$};
\node()at(4,1.8){\tiny$0$};
\node()at(0,-0.3){\tiny$-1$};
\node()at(0.9,-0.3){\tiny$-\frac{3}{4}$};
\node()at(1.9,-0.3){\tiny$-\frac{2}{4}$};
\node()at(2.9,-0.3){\tiny$-\frac{1}{4}$};
\node()at(4,-0.3){\tiny$0$};
\node()at(-1,1.5){\tiny$\partial$};
\node()at(-1,0){\tiny$\partial^{\prime}$};

\draw[-](1,0)--(0,1.5);
\draw[-](1,0)--(1,1.5);
\draw[-](1,0)--(2,1.5);
\draw[-](1,0)--(3,1.5);
\draw[-](2,0)--(3,1.5);
\draw[-](2,0)--(4,1.5);
\draw[-](3,0)--(4,1.5);
\draw[-](4,0)--(4,1.5);
\end{tikzpicture}
\end{figure}
\noindent By using (\ref{periodicity of covering map}), we can change the above triangulation to its normal form (red arcs):
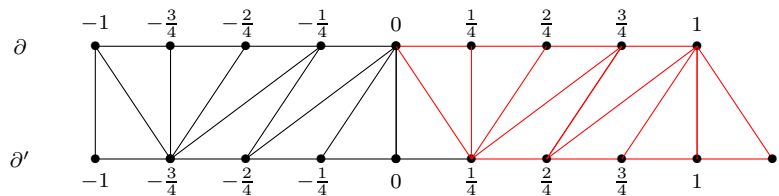
\begin{figure}[H]
\begin{tikzpicture}
\draw[-](0,0)--(0,1.5);
\draw[-](0,0)--(4,0);
\draw[-](0,1.5)--(4,1.5);
\draw[-](4,0)--(4,1.5);
\node()at(0,0){\tiny$\bullet$};
\node()at(4,0){\tiny$\bullet$};
\node()at(0,1.5){\tiny$\bullet$};
\node()at(4,1.5){\tiny$\bullet$};
\node()at(1,0){\tiny$\bullet$};
\node()at(2,0){\tiny$\bullet$};
\node()at(3,0){\tiny$\bullet$};
\node()at(1,1.5){\tiny$\bullet$};
\node()at(2,1.5){\tiny$\bullet$};
\node()at(3,1.5){\tiny$\bullet$};
\node()at(0,1.8){\tiny$-1$};
\node()at(0.9,1.8){\tiny$-\frac{3}{4}$};
\node()at(1.9,1.8){\tiny$-\frac{2}{4}$};
\node()at(2.9,1.8){\tiny$-\frac{1}{4}$};
\node()at(4,1.8){\tiny$0$};
\node()at(0,-0.3){\tiny$-1$};
\node()at(0.9,-0.3){\tiny$-\frac{3}{4}$};
\node()at(1.9,-0.3){\tiny$-\frac{2}{4}$};
\node()at(2.9,-0.3){\tiny$-\frac{1}{4}$};
\node()at(4,-0.3){\tiny$0$};
\node()at(-1,1.5){\tiny$\partial$};
\node()at(-1,0){\tiny$\partial^{\prime}$};
\draw[-](1,0)--(0,1.5);
\draw[-](1,0)--(1,1.5);
\draw[-](1,0)--(2,1.5);
\draw[-](1,0)--(3,1.5);
\draw[-](2,0)--(3,1.5);
\draw[-](2,0)--(4,1.5);
\draw[-](3,0)--(4,1.5);
\draw[-](4,0)--(4,1.5);

\draw[-](4,0)--(5,0);
\draw[-,red](5,0)--(9,0);
\draw[-,red](4,1.5)--(8,1.5);
\draw[-,red](8,0)--(8,1.5);
\node()at(4,0){\tiny$\bullet$};
\node()at(8,0){\tiny$\bullet$};
\node()at(4,1.5){\tiny$\bullet$};

\node()at(8,1.5){\tiny$\bullet$};
\node()at(1,0){\tiny$\bullet$};
\node()at(6,0){\tiny$\bullet$};
\draw[-,red](6,0)--(8,1.5);
\node()at(7,0){\tiny$\bullet$};
\node()at(5,1.5){\tiny$\bullet$};
\node()at(6,1.5){\tiny$\bullet$};
\node()at(7,1.5){\tiny$\bullet$};
\node()at(5,1.8){\tiny$\frac{1}{4}$};
\node()at(6,1.8){\tiny$\frac{2}{4}$};
\node()at(7,1.8){\tiny$\frac{3}{4}$};
\node()at(8,1.8){\tiny$1$};
\node()at(5,-0.3){\tiny$\frac{1}{4}$};
\node()at(5,0){\tiny$\bullet$};
\node()at(6,-0.3){\tiny$\frac{2}{4}$};
\node()at(7,-0.3){\tiny$\frac{3}{4}$};
\node()at(8,-0.3){\tiny$1$};
\draw[-,red](5,0)--(4,1.5);
\draw[-,red](5,0)--(5,1.5);
\draw[-,red](5,0)--(6,1.5);
\draw[-,red](5,0)--(7,1.5);
\draw[-,red](6,0)--(7,1.5);
\draw[-,red](6,0)--(7,1.5);
\draw[-,red](7,0)--(8,1.5);
\draw[-,red](8,0)--(8,1.5);
\draw[-,red](8,1.5)--(9,0);

\node()at(9,0){\tiny$\bullet$};
\end{tikzpicture}
\end{figure}
\noindent Then $\rho(0, 0, 1, 4)=(1, 1, 1, 2)$, and $\varphi(\sigma_{1,2}(T))=\rho(\varphi(T))$.
\end{exm}

\subsection{Proof of Theorem \ref{compatible2}}
Assume $\nu=(c_1, c_2, \cdots, c_p)$. Let $T$ be a tilting bundle in ${\rm coh}\mbox{-}\mathbb{X}(p,q)$ corresponding to $\nu=(c_1, c_2, \cdots, c_p)$, that is, $\varphi(T)=(c_1, c_2, \cdots, c_p)$; see \eqref{the map phi}. According to the proof of Theorem \ref{bijection between tilting bundle and combination vertices}, we obtain
$$\varphi(T(\vec{x}_1))=(c_p-q, c_1, c_2, \cdots, c_{p-1}),\;\;\;\varphi(T(\vec{x}_2))=(c_1-1, c_2-1, \cdots, c_{p}-1).$$
The following assignments define bijective maps $$\rho_1: \Lambda_{(p,q)}^0\longrightarrow \Lambda_{(p,q)}^0;\quad (c_1, c_2, \cdots, c_{p-1}, c_p)\mapsto (c_p-q, c_1, c_2, \cdots, c_{p-1});$$
and $$\rho_2: \Lambda_{(p,q)}^0\longrightarrow \Lambda_{(p,q)}^0; \quad (c_1, c_2, \cdots, c_{p-1}, c_p)\mapsto (c_1-1, c_2-1, \cdots, c_{p}-1).$$

Observe that the grading shift by $(\vec{x}_i)$ on ${\rm coh}\mbox{-}\mathbb{X}(p,q)$ induces an automorphism on $\mathcal{G}(\mathcal{T}^{\nu}_{\mathbb{X}})$, still denoted by $(\vec{x}_i)$.
Similar as the proof of Proposition \ref{bijective map of vertex such that commutative}, we have the following commutative diagrams for $i=1,2$
\begin{figure}[H]
\begin{tikzpicture}
\draw[->](0.3,0)--(1.7,0);
\node()at(-0.4,0){$\mathcal{G}(\mathcal{T}^{\nu}_{\mathbb{X}})$};
\node()at(1,0.2){\tiny{$(\vec{x}_i)$}};
\node()at(1,-2.3){\tiny{$\rho_i$}};
\node()at(-0.6,-1.2){\tiny{$\varphi$}};
\node()at(2.6,-1.2){\tiny{$\varphi$}};
\node()at(2.3,0){$\mathcal{G}(\mathcal{T}^{\nu}_{\mathbb{X}})$};
\draw[->](-0.4,-0.3)--(-0.4,-2.3);
\draw[->](2.4,-0.3)--(2.4,-2.3);
\node()at(2.45,-2.55){{\tiny{$\Lambda_{(p,q)}$}}};
\node()at(-0.35,-2.55){{\tiny{$\Lambda_{(p,q)}$}}};
\draw[->](0,-2.5)--(2,-2.5);
\end{tikzpicture}
\end{figure}

It is easy to check that $\rho_1\rho_2=\rho_2\rho_1$ and $\rho_1^p=\rho_2^q$. Hence, the assignments $r_1\mapsto \rho_1, r_2\mapsto \rho_2$ defines a group homomorphism from $\widetilde{H}_{p,q}$ to $\Aut(\Lambda_{(p,q)})$.
Combining with Proposition \ref{bijective map of vertex such that commutative}, we obtain a group action of $\widetilde{H}_{p,q}$ on $\Lambda_{(p,q)}$.
Then (\ref{compatible of groups action2}) follows from Theorem \ref{description of tilting graph of vector bundles}, and hence we finish the proof of Theorem \ref{compatible2}.

\section{Geometric interpretation of perpendicular category}

In this section, we focus on the geometric interpretation of the perpendicular category of an exceptional object in ${\rm coh}\mbox{-}\mathbb{X}(p,q)$.

Recall that a coherent sheaf $E$ is called \emph{exceptional} if ${\rm Hom}(E,E)=\mathbf{k}$ and ${\rm Ext}^{1}(E,E)=0$.
For an exceptional sheaf $E\in {\rm coh}\mbox{-}\mathbb{X}(p,q)$, define the \emph{right perpendicular category}
$$E^{\bot}:=\{X\in{\rm coh}\mbox{-}\mathbb{X}(p,q) \,|{\rm Hom}(E, X)=0={\rm Ext^{1}}(E, X) \};$$
and the \emph{left perpendicular category}
$$^{\bot}E:=\{X\in{\rm coh}\mbox{-}\mathbb{X}(p,q)\, |{\rm Hom}(X, E)=0={\rm Ext^{1}}(X, E)\}.$$

By Proposition \ref{the properties of coherent sheaves} $(1)$, we obtain that  $^{\bot}E=(E(-\vec{\omega}))^{\bot}$, where $\vec{\omega}=-(\vec{x}_1+\vec{x}_2)$ is the dualizing element in $\mathbb{L}(p,q).$

\begin{lem}\label{exceptional object and arc}
Let $E$ be an indecomposable sheaf in ${\rm coh}\mbox{-}\mathbb{X}(p,q)$, then $E$ is an exceptional object if and only if $\phi^{-1}(E)$ is an arc in $A_{p,q}$.
\end{lem}

\begin{proof}
By \cite[Lemma 3.2.3]{M2004}, $E$ is exceptional in ${\rm coh}\mbox{-}\mathbb{X}(p,q)$ if and only if $\Ext^1(E,E)=0$. Combining with Theorem \ref{dimension and positive intersection}, we get the result.
\end{proof}

We recall some basic facts on cutting marked annulus $A_{p,q}$ along an arc, cf. \cite{MP2014}.

Let $\alpha$ be an arc in $A_{p,q}$. Denote by $A_{p,q}/\alpha$ the new marked surface obtained from $A_{p,q}$ by cutting along the arc $\alpha$ and then removing components which are homeomorphic to a triangle. Up to homeomorphism, it does not depend on the choice of representative of $\alpha$.

\begin{thm}\label{perpendicular category}
Let $E$ be an exceptional object in ${\rm coh}\mbox{-}\mathbb{X}(p,q)$, and assume the oriented arc corresponding to $E$ is $\alpha$. Then the marked surface $A_{p,q}/\alpha$ gives a geometric model for the category $E^{\bot}$.
\end{thm}

\begin{proof}
To prove the theorem, we consider the following two cases:
\vspace{1mm}

(1) $E\in {\rm vec}\mbox{-}\mathbb{X}(p,q).$
\vspace{1mm}

In this case, $E$ is a line bundle, hence $E^{\bot}\cong {\rm mod}(A_{p+q-1})$ by \cite[Proposition 2.14]{Len2011}. On the other hand, the marked surface $A_{p,q}/\alpha$ is a disk with $p+q+2$ marked points on its boundary. According to \cite{BS2021}, this gives a geometric model for the category ${\rm mod}(A_{p+q-1})\cong E^{\bot}$.

(2) $E\in {\rm coh}_{0}\mbox{-}\mathbb{X}(p,q).$

In this case, $E=S_{\infty,i}^{(j)}$, where $i\in\mathbb{Z}/p\mathbb{Z}$ and $1\leq j\leq p-1$; or $E=S_{0,i}^{(j)}$, where $i\in\mathbb{Z}/q\mathbb{Z}$ and $1\leq j\leq q-1$. We only consider the second case, the other one being similar.

Note that the composition factors of $S_{0,i}^{(j)}$ are given by $S_{0,i},\,S_{0,i-1},\cdots,\,\,S_{0,i-j+1}.$
According to \cite{GL1991}, $(S_{0,i}^{(j)})^{\bot}$ contains two disjoint components, one is $\{S_{0,i},\,S_{0,i-1},\cdots,\,\,S_{0,i-j+1}\}^{\bot}$, which equivalent to the weighted projective line of weight type $(p,q-j)$; the other one is $\langle S_{0,i-1},\,S_{0,i-2},\cdots,\,\,S_{0,i-j+1}\rangle$, which is equivalent to the module category of type $A_{j-1}$. Therefore,
$$(S_{0,i}^{(j)})^{\bot}\cong {\rm coh}\mbox{-}\mathbb{X}(p,q-j)\coprod {\rm mod} (A_{j-1}).$$

On the other hand, the marked surface $A_{p,q}/\alpha$ has two connected components: an annulus with $p$ marked points on the inner boundary and $q-j$ marked points on the outer boundary, and a disk with $j+2$ marked points on its boundary. Combining with \cite{BS2021}, this gives a geometric model for the category $(S_{0,i}^{(j)})^{\bot}$. We are done.
\vspace{2mm}
\end{proof}

\noindent {\bf Acknowledgements.}\quad
Jianmin Chen and Hongxia Zhang were partially supported by the National Natural Science Foundation of China
(Grant Nos. 12371040, 11971398, 12131018 and 12161141001). Shiquan Ruan was partially supported by the Natural Science Foundation of Xiamen (No. 3502Z20227184), the Natural Science Foundation of Fujian Province (No. 2022J01034), the National Natural Science Foundation of China (Nos. 12271448), and the Fundamental Research Funds for Central Universities of
China (No. 20720220043).

\vskip 5pt
\noindent {\tiny  \noindent Jianmin Chen, Shiquan Ruan and Hongxia Zhang\\
School of Mathematical Sciences, \\
Xiamen University, Xiamen, 361005, Fujian, PR China.\\
E-mails: chenjianmin@xmu.edu.cn, sqruan@xmu.edu.cn,
hxzhangxmu@163.com\\ }
\vskip 3pt
\end{document}